\documentclass[reqno,11pt]{amsart}
\usepackage{amsmath, latexsym, amsfonts, amssymb, amsthm, amscd}
\usepackage{enumitem}
\usepackage{comment}
\usepackage{graphicx}
\usepackage{scalerel}
\usepackage{stmaryrd}

\usepackage{hyperref} 
 
 \usepackage{tikz} 
 
\DeclareFontFamily{U}{BOONDOX-calo}{\skewchar\font=45 }
\DeclareFontShape{U}{BOONDOX-calo}{m}{n}{
  <-> s*[1.05] BOONDOX-r-calo}{}
\DeclareFontShape{U}{BOONDOX-calo}{b}{n}{
  <-> s*[1.05] BOONDOX-b-calo}{}
\DeclareMathAlphabet{\mathcalb}{U}{BOONDOX-calo}{m}{n}
\SetMathAlphabet{\mathcalb}{bold}{U}{BOONDOX-calo}{b}{n}
\DeclareMathAlphabet{\mathbcalb}{U}{BOONDOX-calo}{b}{n}

\usepackage{urwchancal}
\DeclareFontFamily{OT1}{pzc}{}
\DeclareFontShape{OT1}{pzc}{m}{it}{<-> s * [1.10] pzcmi7t}{}
\DeclareMathAlphabet{\mathcalc}{OT1}{pzc}{m}{it}

\usepackage{graphics,epsf,psfrag}
\usepackage{xcolor}

\numberwithin{equation}{section}

\newtheorem{theorem}{Theorem}[section]
\newtheorem{lemma}[theorem]{Lemma}
\newtheorem{proposition}[theorem]{Proposition}

\newtheorem{rem}[theorem]{Remark}

\DeclareMathOperator{\p}{\mathbb{P}}
\DeclareMathOperator{\sign}{\mathrm{sign}}

\newcommand{\ind}{\mathbf{1}}

\newcommand{\R}{\mathbb{R}}
\newcommand{\Z}{\mathbb{Z}}
\newcommand{\N}{\mathbb{N}}

\newcommand{\lbra}{\llbracket}
\newcommand{\rbra}{\rrbracket}

\newcommand{\pp}{\mathtt{p}_\gG}

\newcommand{\cB}{{\ensuremath{\mathcal B}} }

\newcommand{\bP}{{\ensuremath{\mathbf P}} }
\newcommand{\bE}{{\ensuremath{\mathbf E}} }

\DeclareMathSymbol{\leqslant}{\mathalpha}{AMSa}{"36} 
\DeclareMathSymbol{\geqslant}{\mathalpha}{AMSa}{"3E} 
\DeclareMathSymbol{\eset}{\mathalpha}{AMSb}{"3F}     
\renewcommand{\leq}{\;\leqslant\;}                   
\newcommand{\dd}{\,\text{\rm d}}             

\newcommand{\bbE}{{\ensuremath{\mathbb E}} }

\newcommand{\bbL}{{\ensuremath{\mathbb L}} }

\newcommand{\bbN}{{\ensuremath{\mathbb N}} }

\newcommand{\bbP}{{\ensuremath{\mathbb P}} }

\newcommand{\bbR}{{\ensuremath{\mathbb R}} }

\newcommand{\bbZ}{{\ensuremath{\mathbb Z}} }

\newcommand{\gb}{\beta}
\newcommand{\gd}{\delta}
\newcommand{\gep}{\varepsilon}       

\newcommand{\gG}{\Gamma}

\newcommand{\gD}{\Delta}
\newcommand{\gk}{\kappa}

\newcommand{\gO}{\Omega}

\newcommand{\gs}{\sigma}

\makeatletter
\def\captionfont@{\footnotesize}
\def\captionheadfont@{\scshape}

\long\def\@makecaption#1#2{%
  \vspace{2mm}
  \setbox\@tempboxa\vbox{\color@setgroup
    \advance\hsize-6pc\noindent
    \captionfont@\captionheadfont@#1\@xp\@ifnotempty\@xp
        {\@cdr#2\@nil}{.\captionfont@\upshape\enspace#2}%
    \unskip\kern-6pc\par
    \global\setbox\@ne\lastbox\color@endgroup}%
  \ifhbox\@ne 
    \setbox\@ne\hbox{\unhbox\@ne\unskip\unskip\unpenalty\unkern}%
  \fi
  \ifdim\wd\@tempboxa=\z@ 
    \setbox\@ne\hbox to\columnwidth{\hss\kern-6pc\box\@ne\hss}%
  \else 
    \setbox\@ne\vbox{\unvbox\@tempboxa\parskip\z@skip
        \noindent\unhbox\@ne\advance\hsize-6pc\par}%
\fi
  \ifnum\@tempcnta<64 
    \addvspace\abovecaptionskip
    \moveright 3pc\box\@ne
  \else 
    \moveright 3pc\box\@ne
    \nobreak
    \vskip\belowcaptionskip
  \fi
\relax
}
\makeatother
\def\writefig#1 #2 #3 {\rlap{\kern #1 truecm
\raise #2 truecm \hbox{#3}}}

\def\e{{\mathbb E}}
\def\p{{\mathbb P}}

\def\N{{\mathbb N}}

\def\ss{{\tt u}}
\def\td{{\tt t}^\downarrow}
\def\tup{{\tt t}^\uparrow}
\newcommand{\ttt}{\mathtt{t}}
\newcommand{\ttdown}{\mathtt{t}^\downarrow}
\newcommand{\ttup}{\mathtt{t}^\uparrow}
\newcommand{\tu}{\mathtt{u}}
\newcommand{\tuplus}{\mathtt{u}^+}
\newcommand{\tU}{\mathtt{U}}
\newcommand{\tudown}{\mathtt{u}^\downarrow}

\newcommand{\tuup}{\mathtt{u}^\uparrow}

\newcommand{\ts}{\mathtt{s}}
\newcommand{\tsdown}{\mathtt{s}^\downarrow}
\newcommand{\tsup}{\mathtt{s}^\uparrow}

\newcommand{\tvdown}{\mathtt{v}^\downarrow}

\newcommand{\Srev}{S\mystrut^{\scaleto{\mathtt{rv}\mathstrut}{5pt}}}

\newcommand{\tf}{\textsc{f}}

\newcommand{\ls}{\widehat{l}}
\newcommand{\rs}{\widehat{r}}
\newcommand{\supa}{s^\uparrow}
\newcommand{\sd}{s^\downarrow}
\newcommand{\ms}{\widehat{m}}

\newcommand{\loglog}{\log_{\circ 2}}

\newcommand\mystrut{\rule{0pt}{6pt}}

\newcommand{\tsp}{\mathtt{t}}

\renewcommand{\ell}{\mathcalb{l}}
\newcommand{\err}{\mathcalc{r}}

\begin{document}

\title[]{The random field Ising chain  domain-wall structure in the large interaction limit }

\author[O. Collin, G. Giacomin and Y. Hu]{Orph\'ee Collin, Giambattista Giacomin and Yueyun Hu}
\address[OC, GG]{Universit\'e  Paris Cit\'e,  Laboratoire de Probabilit{\'e}s, Statistiques  et Mod\'elisation, UMR 8001,
            F-75205 Paris, France}
    \address[YH]{Universit\'e Paris XIII, LAGA,  UMR 7539, Institut Galil\'ee, F-93430 Villetaneuse, France}

\begin{abstract}
We study the configurations of the nearest neighbor Ising ferromagnetic chain with IID centered and square integrable external random field in the limit 
in which the pairwise  interaction tends to infinity. The available free energy estimates for this model show a strong form of disorder relevance -- i.e., a strong effect of disorder on the free energy behavior -- and our aim is to make explicit how the disorder affects the spin configurations. 
We give a quantitative estimate that shows that the infinite volume spin configurations are 
 close to one explicit  disorder dependent configuration  when the interaction is large. Our results confirm predictions 
on this model obtained by D. S. Fisher and coauthors by applying the renormalization group method.
\bigskip

\noindent  \emph{AMS  subject classification (2010 MSC)}:
60K37,  
82B44, 
60K35, 
82B27  


\bigskip

\centerline{This version:  \today}

\end{abstract}

\maketitle

\section{The model and the results}
\label{sec:intro}

\subsection{The Ising chain with disordered external field}
\label{sec:model}
For $N\in \bbN:=\{1,2, \ldots\}$ we introduce the partition function of the Ising chain of length $N$, interaction $J$ and external magnetization $h=(h_n)$
\begin{equation}
\label{eq:ZRFIC}
Z_{N, J, h}^{ab}\, =\, 
\sum_{\gs \in \{-1, +1\}^{\{1,2,\ldots, N\}}} \exp\left( H^{ab}_{N, J, h}(\gs)\right) \, ,
\end{equation}
where
\begin{equation}
\label{eq:HRFIC}
H^{ab}_{N, J, h}(\gs) \, :=\, 
J\sum_{n=1}^{N+1} \gs _{n-1} \gs_{n} + \sum_{n=1}^{N} h_n \gs_n\, ,
\end{equation}
with
$\gs_0=a$ and $\gs_{N+1}=b$ ($a,b \in\{ -, +\}\cong\{-1,+1\}$), and  $h=(h_n)$ is a sequence of real numbers. Our results will be  for the case in which  $(h_n)$ is a realization of an IID sequence of centered random variables with bounded non zero variance
\begin{equation}
\label{hyp-S} 
\bbE[h_1]\, =0  \  \text{ and } \vartheta^2\, :=\, \bbE\left[ h_1^2\right]  \in (0, \infty)\,.
\end{equation} 
In spite of the fact that for $(Z_{N, J, h})_{N \in \bbN}$ we only need     $(h_n)_{n \in \bbN}$, we will often deal with the Ising model on arbitrary subsets of $\bbZ$, so we consider
$(h_n)_{n \in \bbZ}$ from the start.

The Gibbs measure $\bP^{ab}_{N, J, h}$ associated to the partition function \eqref{sec:model} is the probability on $\{-1, +1\}^{\{1,\ldots, N\}}$ identified by the discrete density 
\begin{equation}
\label{eq:Gibbs}
\bP^{ab}_{N, J, h}\left( \gs \right) \, =\, \frac{\exp\left( H^{ab}_{N, J, h}(\gs)\right) }{Z_{N, J, h}^{ab}} \, . 
\end{equation} 
Of course a spin configuration $\gs \in \{-1,+1\}^{\{0,1,\ldots\}}$, given  $\gs_0$, can be encoded in the increasing sequence $(n_j)$ defined by $n_1:=
\inf\{n\in \bbN:\, \gs_n\neq \gs_0\}$ and, for $j>1$, $n_j:=\inf\{n > n_{j-1}:\, \gs_n \neq \gs_{n-1}\}$. Therefore one can see the spin configuration $\gs$ as a sequence of domains in which the  spins are constant. The boundaries (or walls) of the domains are at 
  $(n_j-1/2)_{j \in \bbN}$ and the sign of the spins switches crossing the domain walls. 

We aim at understanding the behavior of $\bP^{ab}_{N, J, h}$ in the thermodynamical limit   $N \to \infty$ when the interaction $J$ is large. 

For $J$ large one certainly expects that neighboring spins that have the same sign are greatly favored, so the observed trajectories should contain \emph{very long} domains. 

This is easily established when $\vartheta =0$, i.e. $h_n=0$ for every $n$. In fact in this case, see Remark~\ref{rem:h=0}, $(n_j)$ is a renewal process with geometric inter-arrival distribution: the expectation of the inter-arrival variable, i.e. the expected length of the domains, is $\exp(2J)+1$. 

What is expected in presence of centered disorder with finite variance agrees with the non disordered case only because the domain sizes  diverge with $J\to \infty$, but it is otherwise radically different. The \emph{Imry-Ma argument} -- by now a textbook  argument in physics explained in our context for example in  \cite[p.~373]{cf:IM}, but also in the introduction of \cite{cf:CGH1} --
 strongly suggests that 
\smallskip

\begin{enumerate}
\item the typical length of the domains diverges with $J$, but only like $J^2$;
\item the location of the walls between domains heavily depends on the disordered variables realization.
\end{enumerate}
\smallskip

In the physical literature the analysis of this problem has been pushed well beyond the Imry-Ma predictions: the key word here is Renormalization Group (RG) and
notably the substantial refinement by D. S. Fisher \cite{cf:F92,cf:F95}  of ideas initiated in particular by C. Dasgupta and S.-K. Ma \cite{DM}.
In the next subsection, \S~\!\ref{sec:Fisher}, we are going to give the definitions that are necessary to state our results: these definitions actually introduce some 
stochastic processes that can be directly related to the RG predictions. And our aim is proving a statement that shows that Fisher RG predictions \cite{cf:FLDM01}  do hold for the infinite volume ferromagnetic Ising chain at equilibrium with centered external field, in the limit in which the interaction tends to infinity.

\subsection{The Fisher domain-wall configuration}
\label{sec:Fisher}
Let us set 
\begin{equation}
\gG\, :=\, 2 J\, ,
\end{equation}
and let us introduce from now the two-sided random walk $(S_n)_{n \in \bbZ}$ defined by 
\begin{equation}
\label{eq:Sn} 
S_n\, := \begin{cases}
\sum_{j=1}^n h_j & \text{ if } n \in \bbN\, ,\\
0  & \text{ if } n=0\, ,\\
-\sum_{j=n+1}^0 h_j& \text{ if }  -n \in \bbN\, ,
\end{cases}
\end{equation}
even if the two-sided aspect will play an explicit role only from Section~\ref{sec:auxiliary}.
Let us introduce   also for $m\le n$
\begin{equation}
 S^\uparrow_{m,n}\,:=\,  \max_{m\le i\le j\le n} (S_j-S_i), \qquad S^\downarrow_{m,n}:= \max_{m\le i\le j\le n} (S_i-S_j)\,,  
 \end{equation}
 and for every $\gG>0$  we introduce the \emph{first time of $\gG$-decrease} 
\begin{equation}
 \label{eq:completenotation1}
 \ttt_1(\gG)\,:=\, \inf\left\{n > 0: S^\downarrow_{0,n} \ge \gG\right\} \, . 
 \end{equation}
 For almost all realizations of $(h_n)$ we have that $\limsup_n S_n=-\liminf_n S_n= \infty$: we assume that we work on such realizations,
so the infimum in \eqref{eq:completenotation1} can be replaced by a minimum and
$\ttt_1(\gG)<\infty$.
Next we introduce (with the notation $\lbra j, k\rbra:=[j, k] \cap \bbZ$ for the integers $j\le k$)
\begin{equation}
 \label{eq:completenotation2}
\tu_1(\gG)\,:=\, \min\left\{ n \in \lbra 0,  \ttt_1(\gG)\rbra: S_n= \max_{i \in \lbra 0, \ttt_1(\gG)\rbra} S_i\right\}\,,
\end{equation}
so $\tu_1(\gG)$ is the first time that $S$ reaches its maximum before having a decrease of at least $\gG$: 
this is what we call \emph{location of the first  $\gG$-maximum}.
Note however that $S$ may reach this maximum also at later times and before time $\ttt_1(\gG)$. In fact,
this happens with positive probability (at least for $\gG$ large) if and only 
if there exists a positive integer $n$ and real numbers $x_1,  \dots, x_n$ (not necessarily disctinct) which are atoms of the law of $h_1$ and such that $x_1+\dots+ x_n=0$. 
Therefore, we introduce also 
\begin{equation}
 \label{eq:completenotation2.1}
\tuplus_1(\gG)\,:=\, \max\left\{ n \in \lbra 0,  \ttt_1(\gG)\rbra:\,  S_n= \max_{i\in \lbra 0, \ttt_1(\gG)\rbra} S_i\right\}\,. 
\end{equation}
Now we proceed by looking for the first time of $\gG$-increase after $ \ttt_1(\gG)$
\begin{equation}
 \ttt_2(\gG)\,:=\, \min\left\{n> \ttt_1(\gG):\,  S^\uparrow_{\ttt_1(\gG),n} \ge \gG\right\} \, ,
 \end{equation}
 and we set 
 \begin{equation}
 \tu_2(\gG) \,:=\,  \min\left\{n\in \lbra \ttt_1(\gG), \ttt_2(\gG)\rbra:\,  S_n= \min_{i \in \lbra \ttt_1(\gG), \ttt_2(\gG)\rbra} S_i\right\}\,,
 \end{equation}
 and
 \begin{equation}
  \tuplus_2(\gG) \,:=\,  \max\left\{n \in \lbra \ttt_1(\gG), \ttt_2(\gG)\rbra:\,  S_n= \min_{i \in \lbra \ttt_1(\gG), \ttt_2(\gG)\rbra} S_i\right\}\, .
  \end{equation}
So  $\ttt_2(\gG)$ is the first time at which there is an increase of $S$  at least $\gG$ after the decrease time $\ttt_1(\gG)$.
And $ \tu_2(\gG)$, respectively $\tuplus_2(\gG)$, is the first (respectively last) absolute minimum of the walk between 
$\ttt_1(\gG)$ and $\ttt_2(\gG)$.

Now  (almost surely  in the realization of the $h$ sequence) we can  iterate this procedure to build the increasing sequences $(\ttt_j(\gG))_{j \in \bbN}$, 
$(\tu_j(\gG))_{j \in \bbN}$ and $(\tuplus_j(\gG))_{j \in \bbN}$ with 
\begin{equation}
\ttt_j(\gG) \, \le \tu_j(\gG)\, \le\,  \tuplus_j(\gG)\, < \, \ttt_{j+1}(\gG)\, ,
\end{equation} 
with  $\tu_j(\gG)$ and $ \tuplus_j(\gG)$ that are the locations of the (first and last) maxima (respectively minima) of $S$, more precisely the first and last $S$ absolute maxima in $[\ttt_j(\gG); \ttt_{j+1}]$,
if $j$ is odd (respectively even). Explicitly 
\begin{equation}
 \ttt_{j+1}(\gG)\,:=\begin{cases} 
 \min\big\{n> \ttt_j(\gG):\,  S^\downarrow_{\ttt_j(\gG),n} \ge \gG\big\} & \textrm{ if $j$ is even} \, ,
 \\
 \min\big\{n> \ttt_j(\gG):\,  S^\uparrow_{\ttt_j(\gG),n} \ge \gG\big\} & \textrm{ if $j$ is odd} \, ,
 \end{cases}
 \end{equation}
  \begin{equation}
  \label{eq:new-u}
 \tu_{j+1}(\gG) \,:=\begin{cases}  
 \min\left\{n \in \lbra \ttt_j(\gG),  \ttt_{j+1}(\gG)\rbra:\,  S_n= \max_{i \in \lbra \ttt_j(\gG), \ttt_{j+1}(\gG)\rbra} S_i\right\} &  \textrm{ if $j$ is even}\, ,
 \\
  \min\left\{n \in \lbra \ttt_j(\gG),  \ttt_{j+1}(\gG)\rbra:\,  S_n= \min_{i \in \lbra \ttt_j(\gG), \ttt_{j+1}(\gG)\rbra} S_i\right\} &  \textrm{ if $j$ is odd}\, ,
 \end{cases}
 \end{equation}
 and $\tu_{j+1}^+(\gG)$ is defined like $\tu_{j+1}(\gG)$ with the minimum for 
 $n \in \lbra \ttt_j(\gG),  \ttt_{j+1}(\gG)\rbra$ replaced by a maximum. 
 
 Like before, we  say that $\tu_{j}(\gG)$ and  $\tu_{j}^+(\gG)$ are location of $\gG$-extrema: maxima (respectively minima) if $j$ is odd (respectively even).
 In $\lbra \tu_{j}(\gG), \tu_{j}^+(\gG)\rbra$ there can be other $\gG$-maxima if $j$ is odd (or $\gG$-minima if $j$ is even).

\smallskip 

\begin{rem}
\label{rem:arbitrary-choice}
Of course the choice of starting with $\ttt_1(\gG)$ which is a time of $\gG$-decrease is arbitrary: we could have chosen to start by looking for a time of $\gG$-increase. This   
leads to a sequence of  $\gG$-extrema locations $(\tu_j)$ and $(\tuplus_j)$  that differs only for the the first entry of the sequence and for the labeling of the sequence.
More precisely, if we set 
\begin{equation}
\label{eq:t-down-up}
 \ttt^\downarrow(\gG)\,:=\, \inf\left\{n\ge 1:\,  S^\downarrow_{0,n} \ge \gG\right\} 
 \ \text{ and } \  \ttt^\uparrow(\gG)\,:=\, \inf\left\{n\ge 1:\,  S^\uparrow_{0,n} \ge \gG\right\} \, ,
\end{equation}
and if we choose the opposite convention of looking first for a $\gG$-increase (we call standard the other convention)
\begin{itemize}[leftmargin=0.6 cm]
\item either
$\ttt^\uparrow(\gG)< \ttt^\downarrow(\gG)$ 
and the first detected $\gG$-extremum is the first of the absolute minima 
in $\lbra 0,\ttt^\uparrow(\gG)\rbra$, and the second one is the  $\gG$-maximum we found at $\ttt_1(\gG)$ with  
the standard convention: after that the two procedures coincide except for the shift of $1$ in the labels;
\item or $ \ttt^\downarrow(\gG)< \ttt^\uparrow(\gG)$ and $\ttt_1(\gG)=\ttt^\downarrow(\gG)$ goes undetected. The first 
detected extremum is the first absolute minimum  in $\lbra 0, \ttt^\uparrow(\gG)\rbra$, which coincides with the 
first absolute minimum  in $\lbra  \ttt^\downarrow(\gG), \ttt^\uparrow(\gG)\rbra$ and it is the first minimum, at $\ttt_2(\gG)$, detected with the standard convention: also in this case  the two procedures coincide at all later times, expect for the label shift of $1$.
\end{itemize}
\end{rem}

\smallskip

We are now ready to define $\left(s_n^{(F, +)}\right)_{n \in \bbN}$:  \begin{equation}
\label{eq:sF+}
 s_n^{(F, +)}\, :=\, \begin{cases}
 +1 & \text{ if } \exists j \text{ even such that }   n \in [\![\tu^+_j(\gG)+1,  \tu_{j+1}(\gG))]\!]
  \text{ and } S_{\tu_{j+1}(\gG)}-S_{\tu^+_j(\gG)} > \gG,
 \\
 -1 & \text{ if } \exists j \text{ odd\,\,  such that } n \in [\![\tu^+_j(\gG)+1, \tu_{j+1}(\gG)]\!]
  \text{ and } S_{\tu^+_j(\gG)}-S_{\tu_{j+1}(\gG)}> \gG,
 \\ \phantom{1} 0 & \text{ otherwise.}
 \end{cases}
\end{equation} 



\begin{figure}[h]
\begin{center}

\begin{tikzpicture}[scale=0.7]

\draw[->] (0, 0) -- (19, 0) node[below] {{\scriptsize $n$}}; 
\draw[->] (0, -1.5) -- (0, 5.5) node[left] {{\scriptsize $S_n$}}; 

\foreach \y in {1, 2, 3, 4, 5} {
    \draw (-0.2, \y) -- (0, \y) node[left] {\tiny \y};
}
\foreach \y in {-1} {
    \draw (-0.2, \y) -- (0, \y) node[left] {\tiny \y};
}

 
 \foreach \x in {0.5, 1, 1.5, 2.5, 3,  3.5, 4, 4.5, 5, 5.5, 6, 6.5, 7, 7.5, 8, 8.5, 9, 9.5, 10, 10.5, 11, 11.5, 12, 12.5,  13, 13.5, 14,  14.5, 15, 15.5, 16, 16.5,  17 , 17.5, 18, 18.5} {
     \draw (\x, 0) -- (\x, -0.2);
 }


\draw (0.5, 0) -- (0.5, -0.2) node[below]  {\tiny $1$};
\draw (1.5, 0) -- (1.5, -0.2) node[below] {\tiny $3$};

\draw[line width=1.5pt] 
    (0, 0) -- (0.5, 1) -- (1, 2) -- (1.5, 1) -- (2, 2) -- (3, 0) -- (3.5, -1) -- (5, 2)
    -- (5.5, 3) -- (6, 2) -- (7, 4) -- (8.5, 1)-- (9.5, -1) -- (10.5, 1) -- (11, 0)
    -- (11.5, 1) -- (12, 0) -- (14.5, 5) -- (15, 4)-- (15.5, 5) -- (17.5, 1) -- (18, 2) -- (18.5, 1) -- (19, 2);

\node[below] at (1, -0.2) {\scriptsize $\tu_1$};
\draw[dotted] (1, 0) -- (1, 2);

\node[below] at (2, 0) {\scriptsize $\tu^+_1$};
\draw[dotted] (2, 0) -- (2, 2);

\node[above] at (3.5, 0) {\scriptsize $\ttt_1$};
\node[above] at (3.5, 0.5) {\scriptsize $\tu_2$};
\draw[dotted] (3.5, 0) -- (3.5, -1);

\node[below] at (5, -0.2) {\scriptsize $\ttt_2$};
\draw[dotted] (5, 0) -- (5, 2);
 
 \node[below] at (7, -0.2) {\scriptsize $\tu_3$};
\draw[dotted] (7, 0) -- (7, 4);

 \node[below] at (7, -0.2) {\scriptsize $\tu_3$};
\draw[dotted] (7, 0) -- (7, 4);

 \node[below] at (8.5, -0.2) {\scriptsize $\ttt_3$};
\draw[dotted] (8.5, 0) -- (8.5, 1);

 \node[above] at (9.5, 0) {\scriptsize $\tu_4$};
\draw[dotted] (9.5, 0) -- (9.5, -1);

\node[below] at (13, -0.2) {\scriptsize $\ttt_4$};
\draw[dotted] (13, 0) -- (13, 2);

\node[below] at (14.5, -0.2) {\scriptsize $\tu_5$};
\draw[dotted] (14.5, 0) -- (14.5, 5);

\node[below] at (15.5, 0) {\scriptsize $\tu^+_5$};
\draw[dotted] (15.5, 0) -- (15.5, 5);

\node[below] at (17, -0.2) {\scriptsize $\ttt_5$};
\draw[dotted] (17, 0) -- (17, 2);


\end{tikzpicture}

\end{center}
\caption{\label{fig:1} The sequence of $\gG$-extrema for the simple symmetric random walk with $\gG=5/2$: in fact, $\gG\in (2,3]$ leads to the very same 
sequence of $\gG$ extrema.  }
\end{figure}

\subsection{The main result}
\label{sec:main}
We let $\loglog(\cdot):= \log(\log(\cdot))$.

\medskip

\begin{theorem}
\label{th:main}
Assume that  \eqref{hyp-S}  holds.
Then there exists a positive (deterministic) function $\gG \mapsto D_\gG$ (see \eqref{eq:DgGdef} for a concrete expression) such that 
 for almost every realization of $(h_n)$ and every choice of the boundary spins
\begin{equation}
\label{eq:main1}
\lim_{N \to \infty} \frac 1N \left \vert \left \{ n=1,2 \ldots, N:\, \gs_n \neq s^{(F, +)}_n\right\} \right\vert \,=\,  D_\gG\,  ,
\end{equation} 
in $\bP^{ab}_{N, J, h}$ probability. Moreover for $\gG \to \infty$
\begin{equation}
\label{eq:main2}
D_\gG \,=\, O \left(\frac{\loglog \gG }\gG\right)\, .
\end{equation}
\end{theorem}

\medskip

Note that $D_\gG$ is naturally interpreted as  \emph{discrepancy density} between the spin configuration and the Fisher configuration. 
We will discuss this result at length in \S~\!\ref{sec:overview} notably about its relation with the analogous 
result proven in \cite{cf:CGH1}, which deals with a continuum version of the model we consider. The continuum model in \cite{cf:CGH1} arises as a
weak disorder scaling limit of the disordered Ising model we consider. {The continuum model is easily guessed  if one first remarks that
the Gibbs measure  we consider, i.e. \eqref{eq:Gibbs}, 
is an exponential modification involving the disorder that is given by the  IID sequence $(h_n)$ -- or, equivalently, by the associated random walk $(S_n)$ -- of an underlying free spin process in which the spins are IID Rademacher variables. The continuum model consists in replacing 
 the free spin process by 
a  continuous time Markov chain on the state space  $\{-1,+1\}$ with unit (Poisson) jump rates, and the  disorder $(S_n)$ by a Brownian motion.} In \cite[App.~C]{cf:CGH1} is shown that 
the partition function \eqref{eq:ZRFIC} converges when $N\to \infty$,  $\vartheta\to 0$ and $J \to \infty$ in a suitable joint way to the Continuum partition function in \cite{cf:CGH1}.
It is somewhat surprising that this weak disorder limit model preserves the main features of the underlying discrete model: for the continuum model a result qualitatively (and, to a certain extent, also quantitatively) close to
Theorem~\ref{th:main} does hold. This will be discussed further in \S~\!\ref{sec:overview}, but we anticipate here that  
the analogous construction of the sequence of $\gG$-extrema of \S~\!\ref{sec:Fisher} -- we call it Neveu-Pitman process -- was first  performed in \cite{cf:NP89} by J.~Neveu and J.~Pitman  for a Brownian trajectory, 
without the nuisance of having to deal with the possibility that $\tu_j< \tu_j^+$. Moreover,
in the Brownian case, by scaling, the value of $\gG$ is inessential and one can focus on $1$-extrema. We do not reproduce the construction here 
and we refer of course to \cite{cf:NP89}, but also to \cite{cf:Cheliotis,cf:BF08,cf:CGH1}.
For the location of the $1$-extrema of a standard Brownian motion 
 trajectory $B$ we use the notation $\upsilon_1, \upsilon_2, \ldots$.  In fact, according to \cite{cf:NP89},    the sequence of $1$-extrema times $(\upsilon_j)_{j \in \bbN}$  forms a renewal sequence and $((\upsilon_{j+1}-\upsilon_j, \vert B_{\upsilon_{j+1}}-B_{\upsilon_{j}}\vert) )_{n\in \bbN}$ is an IID sequence too.
$(\upsilon_{2}-\upsilon_1, \vert B_{\upsilon_{2}}-B_{\upsilon_{1}}\vert)$ is a random variable in $(0, \infty)\times (1, \infty)$ whose law is explicitly known.
More precisely the Laplace transform of this two dimensional law 
is known explicitly: the two marginal laws are explicit, notably $ \vert B_{\upsilon_{2}}-B_{\upsilon_{1}}\vert-1$ is an exponential random variable of parameter $1$  (\cite{cf:NP89}, see also \cite{cf:FLDM01}).

\smallskip

Going back to the discrete case, the situation is different, but not substantially different, in part thanks to Donsker Invariance Principle. We collect the random walk analog of the Neveu-Pitman result in the next statement.
\medskip

\begin{proposition}
\label{th:scaling} 
We have that
\begin{equation}
\label{eq:ren-struc}
 \left( \tu_{n+1}(\gG) - \tu_n(\gG), \left \vert S_{\tu_{n+1}(\gG)} - S_{\tu_n(\gG)}\right\vert \right)_{n\in \bbN}\, ,
\end{equation}
is a sequence of independent random vectors taking values in $\bbN \times [\Gamma, \infty)$. Moreover the subsequence with even (respectively odd) indices, i.e. $n\in 2 \bbN$ (respectively $n \in 2 \bbN -1$), is IID and, 
with the usual notation for convergence in law, we have that for every $n \in \bbN$
\begin{equation}
\label{eq:scaling} 
 \left( \frac{\vartheta^2}{ \gG ^2}\left(\tu_{n+1}(\gG) - \tu_n(\gG)\right), \frac{\left \vert S_{\tu_{n+1}(\gG)} - S_{\tu_n(\gG)}\right\vert }{ \gG }\right)
 \stackrel{\gG \to \infty}
  {\Longrightarrow} 
  \left(\upsilon_{2}-\upsilon_1, \vert B_{\upsilon_{2}}-B_{\upsilon_{1}}\vert\right)\,.
\end{equation}
\end{proposition}

\medskip

Proposition~\ref{th:scaling} is in part already present in \cite{cf:BF08} and we prove it in App.~\ref{sec:scaling}. It is of great conceptual importance because it provides an explicit quantitative link with the RG predictions \cite{cf:FLDM01}: in particular Proposition~\ref{th:scaling}
makes explicit the $\gG^2$ (spatial) scale for the Fisher configuration, as predicted by the Imry-Ma argument.
We add that, in spite of a rather different presentation,  our definition of $\gG$-extrema differs from the one in  \cite{cf:BF08} 
only for the fact that we consider  the $\gG$-extrema for a random walk with time running from $0$ to infinity and, in the proof, also for the bi-infinite random walk, while in \cite{cf:BF08} the $\gG$-extrema are introduced for a finite portion of the random walk. Moreover the labeling 
 of the $\gG$-extrema in \cite{cf:BF08} is different because it orders the $\gG$-minima, \emph{potential valleys} in that case, according to their depth. In   \cite[App.~A]{cf:BF08} one finds for a result directly linking Fisher RG iterative procedure
 with their labeling procedure: in our case this order is not important and, in fact, even impossible because we consider infinite systems. 
 We give in Proposition~\ref{th:RG} a statement  analogous to the one in  \cite[App.~A]{cf:BF08}, but which  is tailored to
 the RG procedure in  \cite{cf:FLDM01}, showing that the RG procedure in \cite{cf:FLDM01} does capture all the $\gG$-extrema.

\subsection{Discussion of the result and overview of the rest of the paper}
\label{sec:overview} 
\subsubsection{A quick overview of our main arguments}
The proof of Theorem~\ref{th:main}  relies on the transfer matrix representation of $Z_{N, J, h}$:
in transfer matrix terms we have 
\begin{equation}
Z^{++}_{N, J, h}\,=\, (M_0M_1M_2\ldots M_N)_{1,1}\exp\left(-h_0\right)\,,
\end{equation}
  where $M_k=T(h_k)Q$, with $T(h_k)$ and $Q$ that are $2\times 2$ matrices with  non negative entries explicitly given in \eqref{eq:transmat}. 
We are therefore dealing with a product of IID random matrices and the key tool in analyzing random matrix products
is the analysis of the action of these matrices on the (projective space of the) direction of vectors, see e.g. \cite{cf:Viana}. 
 Since the dimension is two the direction of a vector $v\in \bbR^2\setminus \{0\}$ is encoded by an  angle $\theta\in [0, \pi]$ and,  since the matrices have positive entries, one can  restrict to  $[0, \pi/2]$: $v = \vert v \vert (\cos(\theta), \sin(\theta))$ and $\theta$ identifies the direction. In more explicit terms, $(v_n)$ defined recursively, given $v_0$ (say, non random), by $v_{n}=M_{n} v_{n-1}$ is a Markov chain with state space $[0, \infty)^2\setminus \{(0,0)\}$, and it directly induces the \emph{projective} Markov chain $(\theta_n)$, with state space $[0 ,\pi/2]$.
For us it is more practical to work with $l_{1, n}:= \log \tan\theta_n$, a Markov chain on $\bbR$ (in reality, we will see that the state space is smaller: $(-\gG, \gG)$).  
 We point out that  $l_{1,n}$ depends on $h_1, \ldots, h_n$ and the subscript $1$ is there to recall that $n=1, 2, \ldots$. In fact we will need to consider also the time reversed 
 Markov chain $(r_{N-n, N})_{n=0,1, \ldots}$, with initial condition $r_{N,N}$, which has the \emph{same} transition mechanism as $(l_{1,n})$, except that $h$ is replaced by $-h$: 
 by this we mean that $l_{1, n}=f_{h_n}(l_{1, n-1})$ and  $r_{N-n, N}=f_{-h_{N-n}}(r_{N- n+1, N})$, for a suitable function $f_h$ (given in \eqref{eq:f_h}). 
 So $r_{n+1, N}$ depends on $h_{n+1}, h_{n+2}, \ldots, h_{N}$.

We are going to explain in Section~\ref{sec:tm} that for every $n\in \{1, \ldots, N\}$
we have the identity
\begin{equation}
\label{eq:Psigma0.0}
\bP^{ab}_{ N, J, h}\left( \gs_n=+1\right)\, =\, 
\frac 1 {1+ \exp\left(-\left(l_{1, n-1} +2h_n+ r_{n+1,N}\right)\right)}\, , 
\end{equation}
and $l_{1, n-1}$,  $h_n$ and $r_{n+1, N}$ are independent random variables because of their dependence on the sequence $(h_n)$ of independent variables. The boundary condition $a$ (resp. $b$) determines $l_{1,1}$ (resp. $r_{N,N}$). The intuitive explanation is that $l_{1, \cdot}$, respectively $r_{\cdot, N}$, carries the \emph{disorder} information from the left (respectively right) boundary to the bulk of the system.

 \smallskip
 
Here are three important facts for our analysis:
\smallskip

\begin{enumerate}[leftmargin=0.6 cm]
\item We are going to consider $N \to \infty$ and sites $n$ away from the boundary, so the processes $l$ and $r$ may be considered to be stationary (and independent of the boundary conditions: so we omit the subscripts $1$ and $N$). Obtaining properties on their common invariant probability is one of the important steps in our arguments. 
 \item
 The Markov kernels of $(l_n)$ and of $(r_{-n})$ depend also on $J$. 
 In fact $(l_n)$  is  a 
  random walk with increments $(2 h_n)$  and a \emph{repulsion mechanism} that forces the walk never to leave  $(-\gG,  \gG)$. 
  This repulsion effectively acts only when the walk is at finite distance from $\pm \gG$, even if, strictly speaking, it is felt in the whole of $(-\gG, \gG)$. 
 As $J \to \infty$, the two walls become farther and farther so, on a large scale, this repulsion becomes more and more similar to the (simpler)  hard wall repulsion, i.e. the case 
 in which the process behaves exactly like a random walk, and if a jump would make it leave $(-\gG, \gG)$, the walk is just stopped at the boundary (i.e., at $\gG$ or at $-\gG$): of course the very same remarks apply to $(r_n)$,  by time reversal.
   The hard wall processes are denoted by  $( \widehat l _n )$ and $\left( \widehat r _n \right)$. 
 A second important step for us is controlling the error in replacing $l$ and $r$ with $\ls$ and $\rs$.
 \item The limit we consider is $N \to \infty$ before $J$. Since $n$ is away from the boundary, for large $J$ we expect $\vert l_{n-1}\vert$ and $\vert r_{n+1}\vert$ to be typically very large. Therefore $\vert l_{n-1}+2h_n+r_{n+1}\vert$ is expected to be large so the probability in 
\eqref{eq:Psigma0.0} is essentially $0$ or $1$. But whether it is $0$ or $1$ depends crucially on the sign of  $l_{n-1}+2 h_n + r_{n+1}$. From these observations we infer that, with high probability (as $J\to \infty$), $\gs_n$ is going to be equal to $\sign(l_{n-1}+2 h_n + r_{n+1})$
 and the spin configuration $(\sign(l_{n-1}+2 h_n + r_{n+1}))_n$
  is our best bet for the behavior of the true spin configuration $(\gs_n)$. On the basis of this discussion, one expects that $s^{(F,+)}_n$ is equal or very close to be equal to
$\sign(l_{n-1}+2 h_n + r_{n+1})$. What we will show is slightly different and very sharp:
  $\sign\left(\ls_{n-1} +2h_n+ \widehat r_{n+1}\right)=s^{(F,+)}_n$
  for $n\ge n_0$ with $n_0$ disorder dependent, but a.s. finite.
  Needless to say, this is another crucial step to control the discrepancies between the spin configuration and the Fisher configuration.
 \end{enumerate}

 \subsubsection{About the main result}
 Theorem~\ref{th:main} is analogous to the result proven in \cite{cf:CGH1} for  the continuum model arising in the   weak disorder rescaling  limit. Our approach is highly robust:  in the discrete setup and assuming only (1.3), we match the result obtained in the continuum and Gaussian context. While the argument of proof that we present here is built from  the  strategy exploited  in \cite{cf:CGH1},  we draw the reader's attention to  the following facts.
 \smallskip
 
 \begin{itemize}[leftmargin=0.4 cm]
 \item In  \cite{cf:CGH1} the analog of the processes $(l_n)$ and $(r_{-n})$ solve stochastic differential equations that are (to a certain 
 extent) solvable: in particular, the invariant probability has an explicit expression in terms of Bessel functions. We do not have an explicit expression of the  invariant probability of the Markov chains we deal with for the Ising model: 
 it is possible to show  that the rescaled process $(l_{\lceil \gG^2t\rceil}/\gG)_{t \ge 0}$  converges (as $\gG \to \infty$) to a Brownian motion reflected at the boundary of $[-1,1]$, and from this fact
one can extract that the invariant probability of the chain converges in the same limit to the uniform probability over $[-1,1]$. {But such a \emph{macroscopic} control on the invariant probability, i.e. on intervals of length $\propto\! \gG$, is of no help and, in general, on small scales the invariant probability is not expected to be close to the uniform measure (see Fig.~\ref{fig:inv-p} and its caption): and  we need a control on the invariant measure down to (almost \emph{microscopic}) intervals of length $\propto\! \loglog \gG$. }
 \item Rather than relying on Brownian approximations of random walks (that appear to yield rougher and much less general results), we perform the whole procedure at the level of Markov chains: this demands a number of estimates that are really discrete process estimates.
 Donsker Invariance Principle is used 
   only in Proposition~\ref{th:scaling} that states a fact that is very important in order to appreciate the behavior of the Fisher configuration $s^{(F)}$
 and the link with Fisher RG. But, mathematically,  Proposition~\ref{th:scaling} is a result independent of Theorem~\ref{th:main}.
 \item About the law of $h_1$ we just ask for $h_1 \in \bbL^2$.  This shows a remarkable universality of the result and it is even somewhat surprising from the statistical mechanics viewpoint: note that  the annealed version 
 does not exist at all under such a general assumption on the disorder. And we are also able to deal with  cases in which
 the law of $h$ has a  lattice component: the difficulty in this case is that there can be several extrema at the same height.  
 \item Working with Markov chains causes a priori more troubles than dealing with Brownian motion like in \cite{cf:CGH1}, but it gives also more flexibility at very small scales: in particular, the fact that we can work with processes   $l$ and $r$ on  a  bounded state space $(-\gG,\gG)$ follows by a
  transfer matrix/Markov chain trick. The boundedness 
 property does simplify certain estimates.
 \end{itemize}
 
 \smallskip
 
We do not repeat  here  the detailed discussion in  \cite[Sec.~2.4]{cf:CGH1} about the relation of the pathwise results we present  with the sharp free energy estimates available in the literature. But we recall that for the free energy density 
$\tf (J):= \lim_{N \to \infty}(1/N) \bbE\log Z_{N, J, h}^{ab}$, which does not depend on the choice of the boundary spins $a$ and $b$,
we have the sharp estimate for $J \to \infty$ 
\begin{equation}
\label{eq:kappa}
\tf (J) \,=\, J+\frac{\kappa_1}{2J + \kappa_2} +R(J)\, =\, 
J+ \kappa_1 \sum_{j=0}^\infty \frac{(-\kappa_2)^j}{(2J)^{j+1}}+R(J)\, ,
\end{equation}
with $\kappa_1>0$, $\kappa_2\in \bbR$ and $R(J)$ a remainder. The control of the remainder depends on  some regularity properties  of the law of $h_1$ -- the absolute continuity of the law of $h_1$ with respect to the Lebesgue measure suffices -- and on integrability properties of $h_1$ (we are in any case assuming  \eqref{hyp-S}).
Namely, if  there exists $c>0$ such that  $\bbE[ \exp( c \vert h_1 \vert )]< \infty$
then there exists $C>0$ such that $R(J)=O( \exp(-CJ))$
(this result is proven in \cite{cf:CGGH}, see also \cite{cf:GG22}). In \cite{cf:CGGH} 
it is also proven that if there exists $\xi>5$ such that
 $\bbE[ \vert h_1 \vert ^\xi] < \infty$, then  $R(J)=O( J^{-(\xi-4)})$: in this case of course the series 
 in the rightmost term of \eqref{eq:kappa} may equivalently be replaced by the finite sum containing only the first $\lfloor \xi-4\rfloor$ terms.

One of us  has actually shown \cite{cf:Orphee} that, assuming \eqref{hyp-S} and
$\bbE[\vert h_1\vert^p] <\infty$ for $p>(3 + \sqrt{5})/2\approx 2.618$,
 then $\tf (J) \,=\, J+{\vartheta^2}/{(2J) }+ o(1/J)$. The same result has also been shown in \cite{cf:DMSDS-B} assuming $h_1$ compactly supported. This yields in particular that when \eqref{eq:kappa} holds, then $\kappa_1=\vartheta^2$. The coefficient $\vartheta^2$
is strongly suggested, but not proven,  by Fisher RG
and by Theorem~\ref{th:main} that say that the observed spin configurations are close to $s^{(F, +)}$. In fact,  one can show, by restricting the computation of the partition function to $\gs=s^{(F, +)}$, that in great generality
$\liminf_{J} 2J (\tf (J)-J)\ge \vartheta^2$, but the corresponding upper bound does  not appear to be straightforward and can be proven under the  assumptions of \cite{cf:Orphee}. 
 In \cite{cf:FLDM01} a computation of the free energy to order $1/J^2$ is given, but the result is not reliable because the  
 $1/J^2$ correction to the free energy depends on the trajectory that are close to the optimal trajectory $s^{(F, +)}$
 and not just on $s^{(F, +)}$ \cite{cf:GG22}. As a matter of fact, it is quite clear that the RG prediction cannot go beyond the $1/J$ term in the free energy expansion and the finer corrections are model dependent (and, in our case, dependent on the details of the law of $h_1$ beyond the variance), even if the whole series in $1/J$ for the free energy density \eqref{eq:kappa}
 has the remarkable (and for now rather mysterious) property of just depending on the constant $\kappa_2$.

 There is clearly a gap between the pathwise results we present here: they hold  assuming only \eqref{hyp-S}  on the law of $h_1$,  while  
 the free energy estimate \eqref{eq:kappa} demands  moment conditions and, beyond the leading order result in \cite{cf:Orphee}, also regularity hypotheses on the law of $h_1$. Understanding the free energy $J \to \infty$ behavior assuming only \eqref{hyp-S} is an open problem.

\smallskip

Two additional open questions (discussed in more detail in the introduction of
\cite{cf:CGH1}) are:
\smallskip

\begin{enumerate}[leftmargin=0.6 cm]
\item The issue of whether the estimate $D_\gG=O( \loglog (\gG) / \gG)$ is optimal or not is open.
It is worth stressing that also in our case, and assuming $h_1$ to be Gaussian, the overlap estimate (\cite[Rem.~2.7]{cf:CGH1}) yields $\liminf \gG \, D_\gG\ge \vartheta/2$ (see Remark~\ref{rem:overlap} below).
\item The case $\bbE[h_1]\neq 0$ is tackled via RG too  in \cite{cf:FLDM01}. However the results claimed in \cite{cf:FLDM01} are not expected to be exact (and proven not to be exact \cite{cf:GGG17}):
Fisher RG is expected to be only \emph{asymptotically} exact (see also \cite{cf:IM,cf:Vojta}) when $\bbE[h_1]\neq 0$ (meaning a suitable limit 
as $\bbE[h_1]\to 0$) and a free energy result in this direction is proven in \cite{cf:GGG17}, but results beyond the free energy  behavior are open issues. 
\end{enumerate}

\smallskip

\begin{rem}
\label{rem:overlap} Let us derive a lower bound on $D_\gG$ when the disorder is Gaussian: for every $\vartheta>0$,  if $h_1 \sim N(0, \vartheta^2)$,  we denote the free energy by $\tf_\vartheta (J)$ and  the discrepancy (see Theorem~\ref{th:main})  by  $D_\gG(\vartheta)$.  We are going to show that $\liminf_\gG \gG D_\gG(\vartheta)\ge \vartheta/2$.

We first compare $D_\gG(\vartheta)$ with the left-derivative of $\tf_\vartheta (J)$ with respect to $\vartheta$. 
The function $\vartheta \mapsto \tf_\vartheta (J)$ is convex and a standard Gaussian integration by parts trick in ``replica computations" \cite{cf:Bov},
yields (see \cite[(C.15)]{cf:CGH1})
\begin{equation}
\partial^-_\vartheta \tf_\vartheta (J) \le 2\liminf_{N \to \infty} \bbE \bE^{\otimes 2}_{N, J, h}\left[\frac1{N} \sum_{n=1}^N
\ind_{ \gs^{(1)}_n\neq \gs^{(2)}_n} \right]\,,
\end{equation}
in fact, except for at most countably many values of $\vartheta$, the function is differentiable at $\vartheta$ 
which yields that the above inequality is an equality and one can replace  $\partial^-_\vartheta$ with $\partial_\vartheta$ and the inferior limit with a limit. 
Now, remarking that with $s_n:= s^{(F, +)}_n$
\begin{equation}
\bP_{N, J, h}^{\otimes 2}\left(\sigma_n^{(1)}\neq \sigma_n^{(2)}\right)
\le \bP_{N, J, h}^{\otimes 2}\left(\sigma_n^{(1)}\neq s_n \text{ or } \sigma_n^{(2)}\neq s_n\right)
\le 2 \bP_{N, J, h}\left(\sigma_n\neq s_n\right)\, ,
\end{equation}
we obtain  $\partial^-_\vartheta \tf_\vartheta (J) \le 4 D_\gG(\vartheta)$.
In the case of Gaussian disorder, the assumptions of \cite{cf:Orphee,cf:GG22} are satisfied, so \eqref{eq:kappa} holds, in particular $\tf_\vartheta (J)= J+ \vartheta^2/(2J)+ O(1/J^2)$.  {The arguments in  \cite{cf:GG22} yield that $\kappa_2$ is  continuous in $\vartheta$ and that the rest $O(\exp(-CJ))$ is locally uniform in $\vartheta$, so 
the bound on the term $O(1/J^2)$  is locally uniform in $\vartheta$}.
By convexity, for every $\vartheta>0$ and  every $\gd \in (0, \vartheta)$ we have  $\partial_{\vartheta}^- \tf_{\vartheta} (J) \ge ( \tf_\vartheta (J)-  \tf_{\vartheta-\gd} (J))/ \gd =
(\vartheta^2- (\vartheta- \gd)^2)/ (2J\gd) + O(1/J^2)$. Hence 
$\gG D_\gG(\vartheta) \,\ge \, (2\vartheta - \gd)/4+ O(1/ \gG)$ and we conclude that $\liminf_\gG \gG D_\gG(\vartheta)\ge \vartheta/2$ for every $\vartheta>0$.

\end{rem}

\smallskip

The review of the literature cannot be complete without mentioning that
\smallskip
 
\begin{itemize}[leftmargin=0.4 cm]
\item the Fisher RG prediction have been established for 
Sinai random walks (and diffusions) in random environment  \cite{cf:Cheliotis,cf:BF08}. The Sinai walk has a centering condition that
is analogous to   requiring that $h_1$ is centered  in  our case. Likewise, the results available for  
 transient one dimensional random walks  in random environments \cite{cf:ESZ}, corresponding to $\bbE[h_1]\neq 0$, expose  the limits of Fisher RG outside the centered disorder case. 
\item 
the case of Ising model with long range interaction and random magnetic field is treated in \cite{cf:AW,cf:COP1,cf:COP2,cf:DingHuandMaia}. There the focus is on the case in which the interaction is decaying so slowly that
 there is  still a (first order) phase  transition in the non disordered model and on the issue  of whether the first order transition survives to the presence of a centered disordered  field. 
   We refer to these works for details and we limit ourselves to signaling the open problem of obtaining results analogous to ours 
 for  Ising models with interactions beyond nearest neighbors, with the notable exception of the mean field case and of the related Kac  model 
 \cite{cf:COP0,cf:OP2009}. In particular  \cite{cf:OP2009}   contains a result relating the typical coarse grained domains of the two phases,
 for temperatures below the mean field critical temperature,
  to the Neveu-Pitman process: the result is obtained in the limit of the coarse grained scale going to infinity.
 \item there have been recent substantial progress related to  the Imry-Ma predictions on higher dimensional Ising model and beyond, see notably \cite{cf:AHP,cf:DingXia,cf:DingZhuang}.
 \end{itemize}

\subsubsection{Organization of the rest of the paper}
In  Section~\ref{sec:tm-mainproof} we introduce the transfer matrix formalism, along with the arising processes $l$ and $r$. In this section we also build  the infinite volume Gibbs measure and we relate it to the processes just introduced. Moreover we  introduce the ergodic set-up that we repeatedly exploit in the later sections. We conclude Section~\ref{sec:tm-mainproof} with Proposition~\ref{th:discrepancy0} which is a result in the spirit
of Theorem~\ref{th:main}: it is an important intermediate step in the proof of Theorem~\ref{th:main} and it exploits the control of the invariant probability of $l$ and $r$
 \eqref{eq:fctrl-inv}
that is consequence of Lemma~\ref{th:ctrl-inv}. 
 
   Section~\ref{sec:auxiliary} is devoted to introducing the Fisher domain-wall configuration 
   $s^{(F)}_\cdot$ on $\bbZ$
    that 
  is the natural translationally covariant companion to $s^{(F, +)}_\cdot$. 
 In this section we also 
 introduce the hard wall processes $ \ls$ and $\rs$ and we 
 state three results -- Proposition \ref{th:explicit-hat}, Proposition~\ref{th:lr-Fisher}  and Lemma-\ref{th:modelsclose2} --
 involving these processes. With these results in our hand (and using once again the control on the invariant probability
  \eqref{eq:fctrl-inv}), we give the proof of Theorem~\ref{th:main} at the end of Section~\ref{sec:auxiliary} .
 
 Section~\ref{sec:comp-nu} is devoted to the proof of Lemma~\ref{th:modelsclose2} (about the  closeness of $l_n$ and
 $\ls_n$) and to the proof of Lemma~\ref{th:ctrl-inv} (about the invariant probability).
 
 In Section~\ref{sec:proofergstat} we give the details about the ergodic statements stated in Section~\ref{sec:tm-mainproof}.
 
 App.~\ref{sec:A} is devoted to the proof of various  random walk estimates we need.
 Proposition \ref{th:explicit-hat} and  Proposition~\ref{th:lr-Fisher}
 are proven  in App.~\ref{sec:proofs-constrained}.
 The proof of  the convergence to the Neveu-Pitman process is in App.~\ref{sec:scaling}.
 Throughout the paper, we use the notation 
 \begin{equation} 
 \label{eq:epsilon}
 \gep\,:= \exp(-\gG)\, .
 \end{equation}

\section{Transfer matrices, infinite volume limit and ergodic properties}
\label{sec:tm-mainproof}

It is practical to consider models on $\lbra\ell, \err \rbra=\{\ell,\ell+1, \ldots, \err\}$ with $\ell \le \err$, so we set 
\begin{equation}\label{eq:defZellerr}
Z_{\ell, \err, J, h}^{ab}\, := \, \sum_{\gs \in \{-1, +1\}^{\lbra\ell, \err \rbra}} \exp\left( H^{ab}_{\ell, \err, J, h} (\gs)\right)\,, 
\end{equation}
and 
\begin{equation}\label{eq:defHellerr}
H^{ab}_{\ell, \err, J, h} (\gs)\, :=\, J \sum_{n=\ell}^{\err+1} \gs_{n-1}\gs_n + \sum_{n=\ell}^\err h_n \gs_n, \ \ \ \text{ with } \gs_{\ell-1}:=a \text{ and } 
 \gs_{\err+1}:=b\,.
\end{equation}
We recall that $a$ and $b$ are in $\{-1, +1\} \cong \{ -, +\}$.
We introduce also the corresponding Gibbs measure $\bP_{\ell, \err, J, h}^{ab}$ on $\{-1, +1\}^{\lbra\ell, \err \rbra}$.
Moreover for every $\ell \in \bbZ$ we set also  
\begin{equation}
\label{eq:boundary}
Z_{\ell, \ell-1, J, h}^{ab}\, := \, \exp\left( J ab\right)\, .
\end{equation} 
With these notations we have for $\ell \le n \le  \err$  
\begin{equation}
\bP_{\ell, \err, J, h}^{ab}\left( \gs_n=+1\right)=\frac{
{Z}_{\ell, n-1, J, h}^{a+} e^{h_n}{Z}_{n+1,\err, J, h}^{+b}}
{{Z}_{\ell, n-1, J, h}^{a+} e^{h_n}{Z}_{n+1,\err, J, h}^{+b}
+{Z}_{\ell, n-1, J, h}^{a-} e^{-h_n}{Z}_{n+1,\err, J, h}^{-b}}\, .
\end{equation}
So if we introduce for $n\le m+1$ 
 \begin{equation}
\label{eq:LaRb}
L^{(a)}_{n,m}\, :=\, \frac{ {Z}_{n, m, J, h}^{a+}}{ {Z}_{n, m, J, h}^{a-}}\ \  \ \text{ and } \ \ \
R^{(b)}_{n,m}\, :=\, \frac{{Z}_{n,m, J,h}^{+b}}{ {Z}_{n,m, J, h}^{-b}}\,,
\end{equation}
 as well as 
\begin{equation}
l^{(a)}_{n,m} \, :=\, \log L_{n,m}^{(a)} \ \ \ \text{ and } \ \ \  r^{(b)}_{n,m}\, :=\, \log R_{n,m}^{(b)} \, ,
\end{equation}
we have for $\ell\le n \le \err$
\begin{equation}
\label{eq:Psigmaj}
\bP^{ab}_{\ell, \err, J, h}\left( \gs_n=+1\right)\, =\, \frac  {L_{\ell,n-1}^{(a)} e^{2h_n} R_{n+1,\err}^{(b)}} {1+ L_{\ell,n-1}^{(a)} e^{2h_n} R_{n+1,\err}^{(b)}}\,=\, 
\frac 1 {1+ \exp\left(-\left(l_{\ell,n-1}^{(a)} +2h_n+ r_{n+1, \err}^{(b)}\right)\right)}\, .
\end{equation}
 Note that by \eqref{eq:boundary} we have $l^{(a)}_{n,n-1}=2Ja$ and  $r^{(b)}_{m+1,m}=2Jb$.

\subsection{The transfer matrix formalism} 
\label{sec:tm}

Using the \emph{quantum mechanics} notation for vectors and adjoint vectors we set
$\langle +\vert:=(1,0)$,  $\vert+ \rangle:=(1,0)^{\tsp}$, $\langle -\vert:=(0,1)$ and $\vert - \rangle := (0,1)^{\tsp}$.
We have 
\begin{equation}
\label{eq:transfer-form}
{Z}_{\ell, \err, J, h}^{ab}\, = \, \langle a \vert Q T_\ell Q T_{\ell+1} Q \ldots Q  T_{\err} Q
 \vert b \rangle \,,
\end{equation} 
with $T_n:=T(h_n)$ where for $h \in \bbR$
\begin{equation}\label{eq:transmat}
  T(h)\, :=
\begin{pmatrix}
e^{h} & 0\\
0 & e^{-h}
\end{pmatrix}
\ \ \ \text{ and } \ \ \ 
Q\, :=
\begin{pmatrix}
e^{J} & e^{-J} \\
e^{-J} & e^{J}
\end{pmatrix}
= \frac 1{\sqrt{\gep}}\begin{pmatrix}
1 & \gep\\
\gep & 1
\end{pmatrix}
 \, .
\end{equation} 
In other terms, if we index our $2 \times 2$ matrices by $\{+1, -1\} \times \{+1, -1\}$, then ${Z}_{\ell, \err, J, h}^{ab}$ is the coefficient of $Q T_\ell Q T_{\ell+1} Q \ldots Q  T_{\err} Q$ with index $(a,b)$. This holds also when $\err=\ell-1$, in which case the right-hand side of \eqref{eq:transfer-form} reduces to $\langle a \vert Q  \vert b \rangle$.
\medskip

\begin{rem}
\label{rem:h=0}
When $h_n=0$ for every $n$ the matrix product is reduced to a power of the matrix $Q$ and 
\begin{equation}
 Q\, =\, 2 \cosh(J) P\, :=\, 2 \cosh(J) 
  \begin{pmatrix}
p_J & 1- p_J \\
1- p_J & p_J
\end{pmatrix}
\ \text{ with } p_J\, =\, \frac 1{ 1+\exp(-2J)}\, .
\end{equation} 
It is therefore straightforward to see that in this case 
$\bP^{ab}_{1,N, J, h}$ is simply the law of a Markov chain on the state space $\{+1,-1\}$ with transition matrix $P$, initial condition $a$ at time $0$, observed up to time $N$ and conditioned to be in the state $b$ at time $N+1$. If we consider 
the limit $N\to \infty$,  the differences between switching state times is an IID sequence of  geometric random variables of parameter $1-p_J$. 
\end{rem}

\medskip

In view of \eqref{eq:LaRb} and \eqref{eq:transfer-form} we can write
\begin{equation}
\label{eq:LaRb-2.1}
 L_{n,m}^{(a)}\, =\, \frac
{\langle a \vert Q T_n Q  \ldots Q  T_{m} Q \vert + \rangle}
{\langle a \vert Q T_n  Q \ldots Q  T_{m} Q \vert - \rangle}
\,=\,  \frac
{\langle + \vert Q T_m Q  \ldots Q  T_{n} Q \vert a \rangle}
{\langle - \vert Q T_m  Q \ldots Q  T_{n} Q \vert a \rangle}\,  ,
\end{equation}
where we use that $Q$ and the $T_n$'s are self-adjoint.
Note that if we set for $j=0, 1,2, \ldots$
\begin{equation}
X_j\, :=\,  Q T_{n+j} Q T_{n+j-1} Q  \ldots Q  T_{n} Q \vert a\rangle\, ,
\end{equation}
then $(X_j)_{j=0,1, \ldots}$ is a Markov chain on $(0, \infty)^2$ { with initial condition 
 $X_0 = Q T_n  Q \vert a\rangle$ and  
$Q \vert a\rangle= (\exp(Ja),\exp(-Ja))^{\tsp}$.}
Therefore 
$(L_{n, n+j}^{(a)})_{j=0,1, \ldots}$ is just the (Markov) process with state space $(0, \infty)$ that describes the evolution of  the ratio of the two coordinates of $(X_j)_{j=0,1, \ldots}$. Note that with $h\in\bbR$ and $T=T(h)$
\begin{equation}
\label{eq:1-step.0}
QT \begin{pmatrix} x \\ 1 \end{pmatrix}\, =\,  \begin{pmatrix} e^{J+h}x +e^{-J-h} \\ e^{-J+h}x + e^{J-h} \end{pmatrix} 
 \, ,
\end{equation}
which implies that $x \mapsto \left( x+ e^{-2J-2h}\right)/\left(e^{-2J} x+ e^{-2h}\right)$ is the random  one step map from $(0, \infty)$ to $(0, \infty)$  encoding  the evolution of the ratio chain $(L_{n, n+j}^{(a)})_{j=0,1, \ldots}$.
Since $l_{n, n+j}^{(a)}= \log L_{n, n+j}^{(a)}$ 
 it is straightforward  that $(l_{n, n+j}^{(a)})_{j=0,1, \ldots}$ is also a Markov chain. More precisely it is the Markov chain with state space $\bbR$ generated by the one step map 
 (recall that $\gG= 2J$)
\begin{equation}
\label{eq:f_h}
y \mapsto \log \left( \frac{e^y+e^{-\gG -2h}}{e^{-\gG +y} + e^{-2h}} \right)\, =\, y+ 2h+ \log \left( \frac{1+e^{-\gG-(y+2h)}}{1+e^{-\gG +(y+2h)}} \right)\, =:
\, f_h(y) \, =\, b_\gG \circ \theta_h (y)\, ,
\end{equation}
where we have conveniently introduced (see Fig.~\ref{fig:bGamma})
\begin{equation}
\theta_h(y)\, :=\, y+2h\ \ \text{ and } 
b_\gG(y)\, :=\,  y+ \log \left( \frac{1+e^{-\gG-y}}{1+e^{-\gG +y}} \right)\,.
\end{equation}
 \begin{figure}[h]
\begin{center}
\includegraphics[width=16 cm]{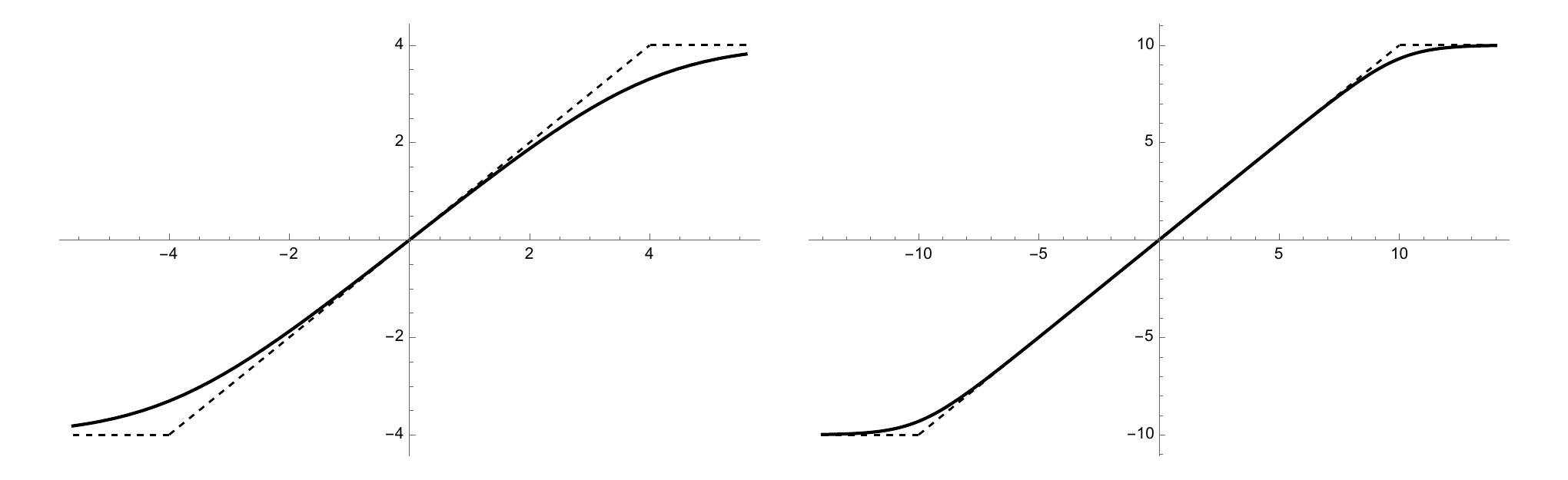}
\end{center}
\caption{\label{fig:bGamma}  
The function $b_\gG$ for $\gG=4$ (on the left) and $\gG=10$ (on the right). The dashed lines plot the function $\widehat{b}_\gG$ introduced in \S~\ref{sec:constrained}.
} 
\end{figure}

Explicitly, for $j=0,1, \ldots$
\begin{equation}
l^{(a)}_{n, n+j}\,=\, f_{h_{n+j}} \circ  f_{h_{n+j-1}} \circ \ldots \circ f_{h_{n}} \left(a \gG\right)\,,
\end{equation}
or (equivalently), for $n\le m$
\begin{equation}
\label{eq:lfkrn}
l^{(a)}_{n, m}\, =\, 
 f_{h_{m}} \circ  f_{h_{m-1}} \circ \ldots \circ f_{h_n} \left(a \gG\right)\, .
\end{equation} 
It is at times practical to speak of Markov chains in terms of their (one step) transition kernel: 
the  transition kernel of $(l_{n, m}^{(a)})_{m=n, n+1, \ldots}$ is  $\pp: \bbR \times \cB(\bbR) \to [0, 1]$ given by  $\pp (l, B):=\bbP( f_{h_1}(l)\in B)$, for every $l\in \bbR$ and every $B\in \cB(\bbR)$.

An important observation at this stage is that $b_\gG$ maps $\bbR$ to $(-\gG, \gG)$,
so the state space of the Markov chain $(l_{n, m}^{(a)})_{m=n, n+1, \ldots}$ may be taken to be  $(-\gG, \gG)$ and not the whole of $\bbR$. 
Consequently,  the domain of the transition kernel can be restricted to $(-\gG, \gG)  \times \cB((-\gG, \gG))$.

\smallskip 

We can now repeat these steps for the $R$ process: from \eqref{eq:LaRb} and \eqref{eq:transfer-form} we have for $n \le m$
\begin{equation}
\label{eq:LaRb-R}
R^{(b)}_{n,m}\, =\, \frac{{Z}_{n,m, J, h}^{+b}}{ {Z}_{n,m, J, h}^{-b}}\,=\, 
\frac{\langle + \vert Q T_n Q  \ldots Q  T_{m} Q \vert b \rangle}
{\langle - \vert Q T_n  Q \ldots Q  T_{m} Q \vert b \rangle}\,,
\end{equation}
and we readily see that this time $(R^{(b)}_{m-j,m})_{j=0,1, \ldots}$ and  
$(r^{(b)}_{m-j,m})_{j=0,1, \ldots}$ are Markov chains. More precisely, for $j=0, 1, \dots$
\begin{equation}
r^{(b)}_{m-j,m}\, =\, 
 f_{h_{m-j}} \circ  f_{h_{m-j+1}} \circ \ldots \circ f_{h_{m}} \left(b \gG\right)\,,
\end{equation} 
or (equivalently), for $n \le m$
\begin{equation}
r^{(b)}_{n,m}\, =\, 
 f_{h_{n}} \circ  f_{h_{n+1}} \circ \ldots \circ f_{h_{m}} \left(b \gG\right)\,.
\end{equation}

\medskip

From \eqref{eq:Psigmaj}, we derive the following formula, which holds for every  $s\in\{+1, 0, -1\}$  and $\ell,  n ,  \err \in \bbZ$ such that $\ell \le n \le \err$:
\begin{equation}
\label{eq:Psigmajneq}
\bP_{ \ell, \err, J, h}^{ab}\left( \gs_n \neq s\right)\, =\, 
\frac {1} {1+ \exp\left(s \,  m^{(ab)}_{\ell, n, \err}\right)}+\frac12 \ind_{s=0}\, ,
\end{equation}
with the notation $m^{(ab)}_{\ell, n, \err}:=l^{(a)}_{\ell, n-1} + 2h_n+ r^{(b)}_{n+1, \err}$. We mention that when $\ell=1$ and $\err=N$, $\bP_{ \ell, \err, J, h}^{ab}$ coincides with $\bP_{N, J, h}^{ab}$ defined in \eqref{eq:Gibbs}.  Note that $s$ may be chosen depending on the sequence $h=(h_n)_{n\in\bbZ}$.

\subsection{The infinite volume system}

We are interested in \eqref{eq:Psigmajneq} with $\ell=1$ and $\err=N$, but we will also look at the limit as $\ell\to -\infty$ and $\err\to \infty$.

\medskip

\begin{lemma}
\label{th:from-infty} For every sequence $(h_n)_{n\in\bbZ}\in\bbR^\bbZ$ we have the following:
for every $m$ the limit $l_m:= \lim_{n \to -\infty} l^{(a)}_{n, m}$ exists and  is independent of $a$. Moreover $(l_m)_{m\in \bbZ}$ is a  sequence that satisfies $l_{m}=f_{h_m}(l_{m-1})$ for every $m\in\bbZ$. 
Likewise, for every $n$, the limit $r_n:= \lim_{m \to +\infty} r^{(b)}_{n, m}$ exists, is independent of $b$ and $r_{n}=f_{h_n}(r_{n+1})$ for every $n\in\bbZ$. 

If $h=(h_n)_{n \in \bbZ}$ is an IID sequence of real random variables, then $(l_n)_{n \in \bbZ}$ and $(r_{-n})_{n\in \bbZ}$ are stationary Markov chains with the 
same transition kernel $\pp$ (given right after \eqref{eq:lfkrn}), which admits a unique invariant probability. 
\end{lemma}
\medskip


\begin{rem}
\label{rem:time-reversal}
Of course the results proven for one of the two processes $l$ or $r$  can be transferred to the other one.
Let us be more explicit about this point (that will be used several times) by introducing the random walk $(\Srev_n)_{n \in \bbZ}$
\begin{equation}
\label{eq:Srev}
\Srev_n\, :=\, S_{-n}\,,
\end{equation}
and $\mathtt{r}_{n}:= r_{-n+1}$. 
 So $\mathtt{r}_n =f_{h_{-n+1}}( \mathtt{r}_{n-1})$, but $h_{-n+1}= -\left( 
 \Srev_n -\Srev_{n-1} \right)$, and in the end we have
 \begin{equation}
 \mathtt{r}_n \, =\, f_{ -\left( 
 \Srev_n -\Srev_{n-1} \right)}( \mathtt{r}_{n-1})\, ,
 \end{equation}
that is to be compared with $l_{n}= f_{S_n-S_{n-1}}\left( l_{n-1} \right)$.
Therefore the (shifted) time reversal of the $r$ process coincides with the $l$ process if we use the increments of $-\Srev$ rather than the increments of $S$ as driving noise. 

{Another useful observation is the following: since the function $b_\gG$ is odd, the process $(-l_n)_{n\in\bbZ}$ coincides with the process $l$ if we use the increments of $-S$ as driving noise.}
\end{rem}
\medskip

\begin{proof}
$b_\gG$ is an odd increasing function and one readily sees that
\begin{equation}
\max_x b_\gG' (x)\, = \, b_\gG' (0)\, =\, \frac{1-\gep}{1+ \gep}\,\le \exp(-\gep) <1\, .
\end{equation}
Therefore $\vert f'_h(x)\vert= f'_h(x)\le \exp(-\gep)$ uniformly in $x$, so
$\vert f_h(x)-f_h(y)\vert \le \exp(-\gep) \vert x-y \vert$ and
\begin{equation}
\label{eq:17JLct}
\sup_{x, y \in [-\gG, \gG]}
 \left \vert f_{h_{m}} \circ  f_{h_{m-1}} \circ \ldots \circ f_{h_n} \left(x\right)- f_{h_{m}} \circ  f_{h_{m-1}} \circ \ldots \circ f_{h_n} \left(y\right)\right\vert
 \, \le \, 2 \gG e^{-\gep (m-n+1)} \, ,
\end{equation}
and from such an estimate it is straightforward to see $\lim_{n \to -\infty} l^{(a)}_{n, m}$ exists for every (deterministic!) choice of the sequence $h\in \bbR^\bbZ$ and that the limit does not depend on $a$. More than that: the limit does not depend on the initial condition 
$l^{(a)}_{n, n}$ which may also depend on $n$ and on the sequence $h$; note that in our case  $l^{(a)}_{n, n}=f_{h_n}(a \gG)$.
It is then straightforward to check that $l_{m}=f_{h_m}(l_{m-1})$ for every $m\in\bbZ$, as well as the other properties we claimed
and this completes the proof. We remark that
the  Markov chain we consider is a special case in the class of \emph{contractive Markov chains}, see for example \cite[Ch.~2]{cf:MC}, in particular Theorem~2.1.9, where the  case
of Markov chains such  that $\bbE[ \log \vert f'_{h_1}(x)\vert] <0$ is considered: 
this \emph{non uniform} contractive property suffices to see that the kernel $\pp$ has a unique invariant probability, to obtain that $\lim_{n \to -\infty} l^{(a)}_{n, m}$ exists a.s.  for every $m$ and that the arising process $l_\cdot$, which does not depend on the value of $a$,  is a stationary Markov chain with kernel $\pp$.
\end{proof}

\medskip

Since it is straightforward to write the probability of every cylindrical (i.e., local) event in $\{-1,1\}^\bbZ$ with the help of the $(l_{n, m}^{(a)})$ and $(r_{n,m}^{(b)})$ processes
(we have done so explicitly in \eqref{eq:Psigmaj}  for the event $\{\gs:\, \gs_j=+1\}$), it is easy  to see that, by Lemma~\ref{th:from-infty}, the two index sequence of probabilities $(\bP^{ab}_{\ell, \err, J, h})$ converges weakly -- the topology on $\{-1,1\}^\bbZ$ is the product one -- as
$\ell \to -\infty$ and $\err \to \infty$ and no matter how the limit is taken. For definiteness we set 
$\bP_{ J, h}:= \lim_{\err \to \infty}\bP^{++}_{-\err, \err, J, h}$, but the limit does not depend on the chosen sequence and it does not depend on the boundary conditions either.
From \eqref{eq:Psigmaj} and Lemma \ref{th:from-infty} we obtain
\begin{equation}
\label{eq:Psigmajinf}
\bP_{ J, h}\left( \gs_n=+1\right)\, =\, 
\frac 1 {1+ \exp\left(-\left(l_{n-1} + 2h_n+ r_{n+1}\right)\right)}\, .
\end{equation}
or equivalently, for every $s\in\{-1,0, +1\}$,
\begin{equation}\label{eq:Psigmajneqinf}
\bP_{ J, h}\left( \gs_n\neq s\right)\, =\, 
\frac {1} {1+ \exp\left(s\left(l_{n-1} + 2h_n+ r_{n+1}\right)\right)}+\frac12\ind_{s=0}\, .
\end{equation}
Thus, we denote for all $n\in\bbZ$, \begin{equation}    m_n:=l_{n-1} + 2h_n+ r_{n+1}, \label{def-mn}\end{equation} in analogy with \eqref{eq:Psigmajneq}.

\subsection{The ergodic system}

As already hinted to, there is an ergodic system in play. Recall that $h$ is the bi-infinite IID sequence $(h_n)_{n\in\bbZ}$. 
We denote $\Theta:\bbR^\bbZ\to \bbR^\bbZ$ the shift application, defined by $\Theta h=(h_{n+1})_{n\in \bbZ}$ if $h=(h_n)_{n\in\bbZ}$. We have of course $\Theta^{\circ m} h=(h_{n+m})_{n\in\bbZ}$ for every $m\in\bbZ$. $\bbR^\bbZ$ is equipped with the product topology and the probability we consider on this space is 
 the law $\bbP$ of the IID sequence $h=(h_n)_{n\in\bbZ}$ and the system is   $\Theta$ invariant. Furthermore the system satisfies the mixing property (because $(h_n)$ is IID), so it is ergodic. Hence, for every $\bbL^1$ function $G:\bbR^\bbZ \to  \bbR$ we have that $\bbP$-almost surely (see for example \cite[Ch.~5]{cf:MC})
\begin{equation}
\label{eq:generalergstat}
\frac1N\sum_{n=1}^N G(\Theta^{\circ n}(h))\underset{N\to\infty}{\longrightarrow} \bbE[G(h)]\, .
\end{equation}
\medskip

For the next result we need the notion of discrepancy density between two configurations: for $\sigma$ and  $\sigma'$ in $ \{-1, 0, +1\}^{\{1, \ldots, N\}}$ or in $\{-1, 0, +1\}^{\bbN}$ we set 
\begin{equation}
\label{eq:D_N}
D_N(\sigma, \sigma')\, :=\, \frac1N \left \vert \left\{ n=1,2 \ldots, N:\, \gs_n \neq \gs'_n\right\}\right\vert = \frac1N \sum_{n=1}^N \ind_{\gs_n\neq \gs_n'}\, .
\end{equation}

\medskip
\begin{proposition}
\label{th:ergstatK}
Consider a measurable  function $K: \bbR^\bbZ\to \{-1,0, +1\}$ and set 
$s^{(K)}_n:=K(\Theta^{\circ n} h)$ for  $n\in\bbZ$.
Then $\bbP(\dd h)$-a.s. (that is, almost surely in the realization of the IID sequence $(h_n)_{n\in\bbZ}$)
\begin{equation}
\label{eq:ergstatK}
D_N\left(\sigma, s^{(K)}\right) \underset{N\to\infty}{\longrightarrow} \bbE\left[\left(1+\exp\left(s^{(K)}_0 m_0\right)\right)^{-1}+\frac12\ind_{s^{(K)}_0 =0}\right]\, ,
\end{equation}
in $\bP^{ab}_{N, J, h}$-probability, as well as in $\bP_{J, h}$-probability.
\end{proposition}

\medskip 

The proof requires no new ingredients. We however leave it to Section \ref{sec:proofergstat} for presentation purposes.
But let us mention that from the definitions we have just given and by \eqref{eq:Psigmajneqinf} and \eqref{eq:generalergstat}
\begin{equation}\label{eq:ergstatconvinexpect}
\begin{split}
\bE_{J, h}\left[ D_N\left(\sigma, s^{(K)}\right) \right]\, &=\, \frac 1 N\sum_{n=1}^N \frac1{1+\exp\left(s^{(K)}_n m_n\right)}+\frac12\ind_{s^{(K)}_n =0}
\\
&\underset{N \to \infty}{\overset{\mathrm{a.s.}}
 \longrightarrow }
\bbE\left[\frac 1{\left(1+\exp\left(s^{(K)}_0 m_0\right)\right)}+\frac12\ind_{s^{(K)}_0 =0}\right]
\, .
\end{split}
\end{equation}
The completion of the proof (in Section \ref{sec:proofergstat}) is about controlling the boundary effects and obtaining a concentration estimate, that is achieved by a variance estimate.

\subsection{A first discrepancy estimate}

Recall that our aim is to show that the typical configurations under $\bP_{N, J, h}^{ab}$ are close to one given configuration. A natural candidate is the configuration $(s_n)$ such that $s=s_n$ maximizes the right-hand side of \eqref{eq:Psigmajneq} with $\ell=1$ and $\err=N$, that is the configuration equal to the sign of $m^{(ab)}_{1, n, N}$. 
We make the slightly different (and slightly sub-optimal) choice 
\begin{equation}\label{eq:defsm}
s^{(m)}_n\, =\begin{cases} +1&  \text{ if } m_n\ge 0,
\\ -1 &\text{ otherwise,}
\end{cases}
\end{equation}
 for $n\in \bbZ$,  where $m_n$ was defined in \eqref{def-mn}. 
The configuration $s^{(m)}$ on $\bbZ$  satisfies the hypotheses of Proposition \ref{th:ergstatK} and therefore, in $\bP_{N, J, h}^{ab}$-probability, we have that
\begin{equation}
\label{eq:discrepgssm}
D_N(\sigma, s^{(m)}) \underset{N\to \infty}{\longrightarrow} \bbE\left[\frac1{1+\exp\left(s^{(m)}_0 m_0\right)}\right]=\bbE\left[\frac1{1+\exp\left(|m_0|\right)}\right]\, .
\end{equation}

\medskip

\begin{proposition}
\label{th:discrepancy0}
We have that
\begin{equation}
\label{eq:discrepancy0}
\bbE\left[\frac1{1+\exp\left(|m_0|\right)}\right]\overset{\gG \to \infty} = O \left( \frac{\loglog \gG}{\gG} \right)\,.
\end{equation}
\end{proposition}
\medskip

The proof of Proposition~\ref{th:discrepancy0} depends on an estimate on the invariant probability for the $\pp$ transition kernel, that is 
 the common law of the random variables  $l_m$ and $r_n$: this estimate, Lemma~\ref{th:ctrl-inv},  is   in section \ref{sec:comp-nu}.  
Here and later on we will use the following result (that can be extracted in a straightforward way from Lemma~\ref{th:ctrl-inv}):
there exist two positive constants $C_1$ and $C_2$ such that for $\gG$ sufficiently large we have
\begin{equation}
\label{eq:fctrl-inv}
\bbP (l_{-1} \in [x, y])\, \le \, C_1 \frac{y-x}\gG\, ,
\end{equation}
for every $x<y$ satisfying
$y-x \ge C_2 \loglog (\gG)$.
We remark that the same estimate holds for  $r_{1}$, because it has the same law as $l_{-1}$, but we are not using this in the proof.

\medskip

\begin{proof}[Proof of Proposition~\ref{th:discrepancy0} using \eqref{eq:fctrl-inv}]
With the constants involved in \eqref{eq:fctrl-inv},
we aim at showing the following \emph{more quantitative} version of \eqref{eq:discrepancy0}: for $\gG$ large
\begin{equation}
\label{eq:claim-r1-1}
\bbE \left[
\frac {1} {1+ \exp\left(\left\vert m_0\right\vert\right)} 
\right]\, \le \, 
3C_1 C_2 \frac {\loglog (\gG)} \gG\, .
\end{equation}
To show \eqref{eq:claim-r1-1} we start by
recalling that $m_0=l_{-1}+2h_0+r_1$ and by observing that the three terms in this sum are independent: in fact, $l_{-1}$ is measurable with respect to $(h_{n})_{n=-1, -2, \dots}$ and $r_1$ is measurable with respect to $(h_n)_{n=1,2, \ldots}$. It is therefore enough to show that for each value $\lambda\in\bbR$:
\begin{equation}
\bbE \left[
\frac {1} {1+ \exp\left(\left\vert l_{-1}+\lambda\right\vert\right)} 
\right]\, \le \, 
3C_1 C_2 \frac {\loglog (\gG)} \gG\, .
\end{equation}
Setting $c_\gG:= C_2 \loglog(\gG)$, let us observe that if we introduce for every $\lambda\in \bbR$ and for $j=0,1, \ldots$
\begin{equation}
A_{\lambda, j}\,:=\, \left \{l\in\bbR:\, j c_\gG \le \vert l+\lambda \vert \le  (j+1) c_\gG  \right\}\, ,
\end{equation}
we have 
\begin{equation}
\bbE \left[
\frac {1} {1+ \exp\left(\left\vert l_{-1}+\lambda\right\vert\right)} 
\right]\,  \le \, 
\sum_{j =0}^\infty \exp\left(-j c_\gG\right) \bbP\left(l_{-1} \in A_{\lambda, j} \right)\, .
\end{equation}
Since $A_{\lambda, j}= [- (j+1)c_\gG -\lambda , - jc_\gG -\lambda]\cup  [jc_\gG -\lambda , (j+1)c_\gG -\lambda]$, { we apply \eqref{eq:fctrl-inv} and get that $\bbP\left(l_{-1} \in A_{\lambda, j} \right)\le 2C_1 C_2  \loglog (\gG)/ \gG $ for every $\lambda$.}  Therefore,   we obtain
 that \eqref{eq:claim-r1-1} holds  for $\gG$ large
 and the proof of Proposition~\ref{th:discrepancy0} is complete. 
\end{proof}

\medskip

\begin{rem}
Under the stronger assumptions of \cite{cf:CGGH}, 
i.e. the existence of $n_0\in \N$ such that the $h_1+\dots +h_{n_0}$ has a density, and the existence of $\xi>5$ such that $h_1$ has a finite $\xi$-th moment,
the bound in Proposition  \ref{th:discrepancy0} can be improved to $O\left(\frac{1}{\gG}\right)$. 
Indeed, in \cite{cf:CGGH}, an estimate (see Equation (1.13) there) is given on the Wasserstein-1 distance between the invariant measure of the Markov chain $(l_n)$ and a \emph{guess} measure which is built to approximate it and is well controlled. 
Since the function $x\mapsto \frac{1}{1+\exp(|x|)}$ is 1-Lipschitz, this estimate, together with the fact that the integral of that function against the \emph{guess} measure is $O\left(\frac{1}{\gG}\right)$, yields the above claim. 
\end{rem}

\medskip

\begin{rem}
\label{rem:gen-erg}
Section~\ref{sec:tm-mainproof} has been developed with the language of Markov chains, and the (one step) Markov character is due to the fact that $(h_n)$ is a sequence of independent random variables. But the first part of Lemma~\ref{th:from-infty} is stated for general deterministic sequences $(h_n)$ and the proof of Prop.~\ref{th:ergstatK} just requires ergodicity 
of the sequence $(h_n)$: as a consequence \eqref{eq:discrepgssm} holds for ergodic sequences. In this generalized ergodic set-up  one deals with a more general discrete time stochastic process, or stochastic dynamical system, always defined by iterated application of a random function. But we stress that
 estimating the right-most term in \eqref{eq:discrepgssm}, like in Proposition~\ref{th:discrepancy0},  demands an understanding of the invariant probability of the stochastic dynamical system, which we achieve under the hypothesis that $(h_n)$ is IID. And obtaining  Theorem~\ref{th:main} demands a substantially deeper understanding of the 
stochastic dynamical system. 
\end{rem}

\section{Fisher domain-wall configuration on $\bbZ$ and reflected random walk}
\label{sec:auxiliary}

In the previous section, we have identified $s^{(m)}$ as a configuration around which the typical configurations are \emph{concentrated} (in the sense of the discrepancy density). This configuration has the disadvantage of being rather implicit: it is a complicated  function of the disorder $(h_n)$, involving Markov chains driven by the disorder and limit procedures. The aim of the rest of the paper is to show that $s^{(m)}$ can be replaced with the Fisher configuration $s^{(F, +)}$ (given in \eqref{eq:sF+}), and this will yield  Theorem~\ref{th:main}

A first step in this program is to define a translation invariant version $s^{(F)}$ of the Fisher configuration $s^{(F, +)}$. 

\subsection{The Fisher domain-wall configuration on $\bbZ$}
\label{sec:Fisher-Z}

We introduce the Fisher configuration $s^{(F)}$ on $\bbZ$ by relying on  the two-sided random walk $S$ defined in \eqref{eq:Sn}. For every $N\in \bbN$ we can repeat the construction of the one-sided  Fisher configuration $s^{(F,+)}$,  replacing 
$(S_n)_{n \in \bbZ}$ with  $(S_{n-N}-S_{-N})_{n \in \bbZ}$. This generates 
a new sequence of $\gG$-extrema locations $\tu_{1,N}(\gG), \tu_{2,N}(\gG), \ldots$, along with the corresponding $+$ sequence
$\tu^+_{1,N}(\gG), \tu^+_{2,N}(\gG), \ldots$.
In order to take into account the translation we have made, the (ordered) set of $\gG$-extrema we are interested in is\begin{equation}
 \tu_{1,N}(\gG)-N \, \le \, \tu^+_{1,N}(\gG)-N\, < \, 
  \tu_{2,N}(\gG)-N \, \le \, \tu^+_{2,N}(\gG)-N\, < \, \ldots
\end{equation}
but it is practical to relabel this sequence so that 
the location of the $\gG$-extremum with label $0$ is the one at the largest negative location. More formally
we introduce $J_N:= \vert \{ j\in \bbN:\, \tu_{j,N}(\gG)-N < 0 \} \vert$ and then we set for $k\in \{ -J_N+1, -J_N+2,  \ldots\}$
\begin{equation}
{\mathbf u}_{k, N}\, :=\,   \tu_{J_N+k,N}(\gG)-N \ \ \text{ and } \ \  {\mathbf u}^+_{k, N}\, :=\,   \tu^+_{J_N+k,N}(\gG)-N\, .
\end{equation}
For the next statement let us introduce $\gO:=\{(h_n)\in \bbR^\bbZ:\, \limsup_n S_{n}=\limsup_n S_{-n}=- \liminf_n S_{n}=- \liminf_n S_{-n}= \infty\}$ 	
so  (with abuse of notation)  $\bbP(\gO)=1$. Restricting to $\gO$ ensures in particular that $({\mathbf u}_{k, N})_k$ and $({\mathbf u}^+_{k, N})_k$
are infinite sequences, for every $N$.

\medskip

\begin{lemma}
\label{th:N-stability} 
Assume that $(h_n) \in \gO$. 
Then for every $k\in \bbZ$ we have the existence of the limits
\begin{equation}
\label{eq:N-stability-existence}
 {\mathbf u}_{k}(\gG) \,:=\, 
\lim_{N \to \infty}  {\mathbf u}_{k, N} 
 \ \ \text { and } \ \   {\mathbf u}^+_{k} (\gG) \,:=\, \lim_{N \to \infty}  {\mathbf u}^+_{k, N} \, ,
\end{equation}
and there exists $\mathtt{s}\in \{-1, 0, +1\}$ such that
$\tu_k(\gG)= {\mathbf u}_{k+\mathtt{s}}(\gG)$, as well as $\tu^+_k(\gG)= {\mathbf u}^+_{k+\mathtt{s}}(\gG)$, for every
$k\ge \max(1, 1-\mathtt{s})$.
\end{lemma}

\medskip

In short, Lemma~\ref{th:N-stability}  builds the bilateral sequence of $\gG$-extrema and the   $\gG$-extrema at non-negative locations
 coincide with the one-sided sequence of $\gG$-extrema, except possibly for the first entry. And, in any case, at $(\tu_k)_{k=2,3, \ldots}$
 there are $\gG$-extrema both in the one and two-sided case.
\medskip

\begin{proof}
First of all \eqref{eq:N-stability-existence} directly follows because
 for every $L \in \bbN$ and  every $M> N \ge L$ we have
\begin{equation}
\label{eq:N-stability2}
{\mathbf u}_{k, M}\, =\, {\mathbf u}_{k, N} 
 \ \ \text { and } \ \ {\mathbf u}^+_{k, M}\, =\, {\mathbf u}^+_{k, N} 
\end{equation}
for $k> \min\{j:\, {\mathbf u}_{j, N}\ge -L\}=: j_{N,L}$.
In fact, \eqref{eq:N-stability2} holds since there is  a $\gG$-extremum at ${\mathbf u}_{j_{N, L}, N}$ 
(say, a maximum: the argument is strictly analogous if it is a minimum) so there exists $n\in [\![ {\mathbf u}_{j_{N, L}, N}, {\mathbf u}_{j_{N, L}+1, N}]\!]$ such that $S_n-S_{ {\mathbf u}_{j_{N, L}+1, N}}\ge \gG$. And of course there is
$n'\in [\![ {\mathbf u}_{j_{N, L}+1, N}, {\mathbf u}_{j_{N, L}+2, N}]\!]$ such that $S_{n'}-S_{ {\mathbf u}_{j_{N, L}+1, N}}\ge  \gG$, therefore 
 it follows from the construction in Section~\ref{sec:Fisher} that, regardless from where the sequential procedure to discover
 the $\gG$-extrema starts (before $-L$), the $\gG$-minimum at  ${\mathbf u}_{j_{N, L}+1, N}$ is going to be detected.
  
 For the second part of the statement we start from the fact that at $\tu_2(\gG)$ there is a $\gG$-minimum for the one-sided case (by definition), but also for the two-sided cases because of  the argument we just developed. Then recall the notation introduced in \eqref{eq:t-down-up}:
 either $\ttt^\downarrow(\gG)< \ttt^\uparrow(\gG)$ or $\ttt^\downarrow(\gG)> \ttt^\uparrow(\gG)$.
\smallskip

\begin{itemize}[leftmargin=0.5 cm]
\item  If $\ttt^\downarrow(\gG)< \ttt^\uparrow(\gG)$ then either $\mathtt{u}_1 =\mathbf{u}_1$ and in this case $\mathtt{u}_k=\mathbf{u}_k$ and $\mathtt{u}^+_k=\mathbf{u}^+_k$ for every $k \in \bbN$, or
at $\mathtt{u}_1$ is not the location of a  $\gG$-extremum in $\bbZ$ and $\mathtt{u}_{k+1}=\mathbf{u}_k$, as well as $\mathtt{u}^+_{k+1}=\mathbf{u}^+_k$,  for every $k \in \bbN$.
\item  If $\ttt^\downarrow(\gG)> \ttt^\uparrow(\gG)$ then $\mathtt{u}_1$ is the location of a $\gG$-maximum also in the whole of $\bbZ$. 
If in $[\![0, \mathtt{u}_1]\!]$ there is no $\gG$-minimum (in $\bbZ$) then $\mathtt{u}_k=\mathbf{u}_k$ and $\mathtt{u}^+_k=\mathbf{u}^+_k$ for every $k \in \bbN$. Otherwise
$\mathtt{u}_k=\mathbf{u}_{k+1}$ and $\mathtt{u}^+_k=\mathbf{u}^+_{k+1}$ for every $k \in \bbN$.
\end{itemize}

This completes the proof of Lemma~\ref{th:N-stability}.
\end{proof}

\medskip

For $(h_n) \in \gO$  and  $n \in \bbZ$ we set
\begin{equation}
\label{eq:sF}
 s_n^{(F)}\, :=\, \begin{cases}
 +1 & \text{ if there exists } k \text{ such that }  
 n \in [\![{\mathbf u}^+_k(\gG)+1,  {\mathbf u}_{k+1}(\gG)]\!],
 \\ &  \ \ \ {\mathbf u}^+_k(\gG) \text{ is the location of a }\gG\text{-minimum and } S_{  {\mathbf u}_{k+1}(\gG)}-S_{{\mathbf u}^+_k(\gG)}>\gG
 \\
 -1 & \text{ if there exists } k \text{ such that } 
 n \in [\![{\mathbf u}^+_k(\gG)+1, {\mathbf u}_{k+1}(\gG)]\!],
  \\ &  \ \ \  {\mathbf u}^+_k(\gG) \text{ is the location of a }\gG\text{-maximum and } S_{{\mathbf u}^+_k(\gG)}-S_{  {\mathbf u}_{k+1}(\gG)}>\gG
 \\ \ 0 & \text{ otherwise.}
 \end{cases}
\end{equation} 
If $(h_n) \not\in \gO$ we set $s_n^{(F)}=0$ for every $n$.
Note that a direct consequence of Lemma~\ref{th:N-stability} is that in $\gO$ we have 
\begin{equation}
\label{eq:stability-+}
 s_n^{(F)}\, =\,  s_n^{(F,+)} \ \text{ for } \ n \ge \tu_2\, ,
\end{equation}
and we recall that $s_n^{(F,+)}$ is given in \eqref{eq:sF+}.

\begin{rem}\label{rem:defFisher}
In Remark \ref{rem-descrground}, we will describe the ground state configurations of the infinite-volume RFIC (in some sense, the ``maximisers'' of the Hamiltonian). 
As it appears, there are sites to which the ground state configurations do not all assign the same spin : 
in the case of multiple $\gG$-extrema, i.e. when $\mathtt{u}_k(\gG)<\mathtt{u}_k^+(\gG)$,  
or  when a $\gG$-slope has height exactly equal to $\gG$, there are several configurations with the same ``maximal'' energy. 
Instead of taking $ s^{(F)} $ to be one arbitrarily chosen ground state configuration, we decided  to set $s_n^{(F)}= 0$ at sites where this ambiguity occcurs. Of course, for those $n$'s we will have automatically $\sigma \neq s_n^{(F)}$, so the version of Theorem \ref{th:main} that we decided to state and prove is stronger than for any arbitrary choice of $ s^{(F)} $ among ground state configurations.
Our analysis will show in particular that the regions where we set $s_n^{(F)}=0$ have a density dominated by ${\loglog\gG}/ \gG$.
\end{rem}

\subsection{The (hard wall) reflected random walk}
\label{sec:constrained}

In strict analogy with the processes  $(l_n)$ and $(r_n)$ in Section~\ref{sec:tm} we  introduce also the Markov chain
\begin{equation}
\label{eq:MCDShat}
\ls_{n}\, =\, \, \widehat f_{h_{n}}\left( \ls_{n-1}\right) \, ,
\end{equation} 
as well as the process
\begin{equation}
\rs_{n}\, =\, \widehat f_{h_{n}}\left( \rs_{n+1}\right)\, ,
\end{equation} 
with $f_{h_{n}}(x) := \widehat b_\gG \circ \theta_{h_n}(x)$, where 
\begin{equation}
\widehat b_\gG(x) :=
\begin{cases}
-\gG & \text{if } x\le -\gG,\\
x & \text{if } -\gG \le x \le \gG,\\
\gG & \text{if } x\ge \gG,
\end{cases}
\qquad
\text{and, as before,}
\qquad 
\theta_{h_n}(y) : = y+ 2 h_n,
\end{equation}
see Fig.~\ref{fig:bGamma}. 

In particular the time reversal discussion that one can find at the end of  Section~\ref{sec:tm}
applies also to $(\ls_n)$ and $(\rs_n)$, so results established on $(\ls_n)$ can be transferred to  $(\rs_n)$ via  time reversal.
\medskip

\begin{rem}
In Appendix \ref{sec:appendzerotemp}, we will illustrate the fact that the processes $\ls$ and $\rs$ can be seen as zero-temperature analogues of the processes $l$ and $r$.
\end{rem}

We consider both $\ls$ and $\rs$ to be running at \emph{stationarity}: the uniqueness of the invariant measure is given in the next statement.
\medskip

\begin{lemma}
\label{th:stationary-hat}
The Markov chain defined by the random iteration \eqref{eq:MCDShat} has a unique invariant probability and there exists 
a unique stationary Markov chain $(\ls_n)_{n \in \bbZ}$ that obeys \eqref{eq:MCDShat} for every $n$.
\end{lemma}

\medskip

By time inversion, Lemma~\ref{th:stationary-hat} provides the definition of  the stationary process $(\rs_n)_{n \in \bbZ}$.

\medskip

 \begin{proof} 
 Since $\widehat f_h(\cdot)$ is $C^0$ for every $h\in \bbR$, one readily sees that this Markov chain is Feller, i.e. 
 $x \mapsto \bbE\left[F(\widehat f_{h_1}(x))\right]$ is $C^0$ (and bounded) whenever $x \mapsto F(x)$ is $C^0$ and bounded. 
 Since the state space $[-\gG, \gG]$ is compact one readily infers that there exists an invariant probability (see for example \cite[Sec.~12.3]{cf:MC}).
 Let us  remark that if
the Markov chain $\left(\ls_{n,-N,x}\right)_{n=-N, -N+1, \ldots}$ is defined by the same recursion obeyed by the $\ls$ chain, but with
$\ls_{-N,-N,x}= x \in [-\gG,\gG]$, then $\lim_{N \to \infty} \ls_{n,-N,x}$ exists a.s. and it is independent of the value of $x$. 
To see this it suffices to observe  that if  we set 
\begin{equation}
 \label{eq:completenotation1-2}
 \ttup(\gG,-N)\,:=\, \inf\left\{n> -N: S^\uparrow_{-N, n}\ge\gG\right\} \ \text{ and }
 \ \ttdown(\gG, -N)\,:=\, \inf\left\{n> -N: S^\downarrow_{-N, n}\ge \gG\right\},
 \end{equation}
then if $n\ge  \ttup(\gG,-N) \wedge \ttdown(\gG, -N)$ we have that $\ls_{n,-N,x}$ becomes independent of the value of $
x$. In particular we can consider a process
$\ls_{k, -M,y}$ with $M>N$, for $k=-N,-N+1, \ldots$ and $\ls_{k,-M,y}=\ls_{n, -N,x}$  for every  $n\ge  \ttup(\gG,-N) \wedge \ttdown(\gG, -N)$. This implies that $\lim_{N \to \infty}\ls_{n, -N,x}=: \ls _n$ exists a.s. for every $n$ and it is straightforward to check that 
$\ls_{n+1}= \widehat f_{h_{n+1}}\left(\ls_n\right)$. Moreover, the very same coupling argument yields uniqueness of the invariant probability, trajectorial uniqueness of the process and that 
$\left(\ls_n\right)_{n \in \bbZ}$ is stationary. 
\end{proof}

\smallskip

We are now going to state three results  on the reflected walks $\ls$ and $\rs$ and then show that Theorem~\ref{th:main} is a consequence of these results.

The first one 
    gives a rather explicit formula for $\ls_0$ and $\rs_1$ (and, by translation, one can obtain a formula for $\ls_n$ and $\rs_n$, $n\in \bbZ$).
For this we recall the notations given in \eqref{eq:t-down-up} and 
add
\begin{equation}\label{eq:defudown}
\tudown(\gG)\,:=\, \min\left\{ n \in \lbra 0,  \ttt^\downarrow(\gG)\rbra: S_n= \max_{i \in \lbra 0, \ttt^\downarrow(\gG)\rbra} S_i\right\}\,,
\end{equation}
as well as 
\begin{equation}
\tuup(\gG)\,:=\, \min\left\{ n \in \lbra 0,  \ttt^\uparrow(\gG)\rbra: S_n= \min_{i \in \lbra 0, \ttt^\uparrow(\gG)\rbra} S_i\right\}\,.
\end{equation}

\noindent Recall \eqref{eq:t-down-up} for $\ttt^\downarrow(\gG)$  and $\ttt^\uparrow(\gG)$. By replacing $S$ with $\Srev$, the time-reversed random walk defined in \eqref{eq:Srev},  in the definitions of $\tudown(\gG)$, $\tuup(\gG)$,   $\ttt^\downarrow(\gG)$  and $\ttt^\uparrow(\gG)$, we add the subscript $\Srev$ to the corresponding random variables and obtain the following definitions: \begin{equation} \label{def-supdown}
\mathtt{s}^\uparrow (\gG)\, :=\, -\ttup_{\Srev}(\gG),
\ 
\mathtt{s}^\downarrow (\gG)\, :=\, -\ttdown_{\Srev}(\gG),
\
\mathtt{v}^\downarrow (\gG)\, :=\,-\tudown_{\Srev}(\gG) \text{ and }
\mathtt{v}^\uparrow (\gG)\, :=\,-\tuup_{\Srev}(\gG)\, .
\end{equation}

\noindent See Figure \ref{fig:l0} for an illustration of  $\mathtt{s}^\downarrow (\gG)$ and $\mathtt{v}^\downarrow (\gG)$. 
\medskip

\begin{proposition}
\label{th:explicit-hat}
We have that 
\begin{equation}
\label{eq:explicit-hat-l}
\widehat l_0\, =\begin{cases}
 +\gG -  2 S_{\mathtt{v}^\downarrow (\gG)}
 & \text{ if } \mathtt{s}^\downarrow(\gG)> \mathtt{s}^\uparrow(\gG)\, ,
\\ 
 -\gG -  2 S_{\mathtt{v}^\uparrow (\gG)}
& \text{ if } \mathtt{s}^\uparrow(\gG)> \mathtt{s}^\downarrow(\gG)\, ,
\end{cases}
\end{equation}
and 
\begin{equation}
\label{eq:explicit-hat-r}
 \rs_1\, =\begin{cases}
 -\Gamma+2 S_{\tudown(\gG)}
 & \text{ if } \ttdown(\Gamma)<\ttup(\Gamma)\, ,\\
 +\Gamma+2 S_{\tuup(\gG)}
  &  \text{ if } \ttup(\Gamma)<\ttdown(\Gamma)\, .
\end{cases}
\end{equation}
\end{proposition}

\medskip

Proposition~\ref{th:explicit-hat} is crucial to the arguments in Section~\ref{sec:comp}. The relevance of the next result needs no comment:

\medskip

\begin{proposition}
\label{th:lr-Fisher}
We have for all $n\in\bbZ$ that
\begin{equation}
\label{eq:lr-Fisher}
s^{(F)}_n\, =\, \sign\left( \ls_{n-1}+2h_n + \rs_{n+1} \right)\,.
\end{equation}
with the convention $\sign(0)=0$.
\end{proposition}

\medskip

The third result is about the proximity of the reflected walks with the processes that arise from the transfer matrix analysis:
\medskip

\begin{lemma}
\label{th:modelsclose2}
    For every $\kappa>0$, there exists $C_\kappa>0$ such that
     with probability $1-O(\Gamma^{-\kappa})$ we have
    \begin{equation}
    \label{lapproxls}
    \left| l_0-\ls_0\right |\, \le \, \loglog \Gamma + C_\kappa\qquad \text{ and } \qquad \left| r_0-\rs_0\right|\, \le \, \loglog \Gamma + C_\kappa\,.
    \end{equation}
    \end{lemma}

\medskip

By translating, Lemma~\ref{th:modelsclose2} holds also if $0$ is replaced by any $n\in \bbZ$ (the event with probability $1-O(\Gamma^{-\kappa})$ on which the result holds depends on $n$).
\smallskip

The proofs of Proposition~\ref{th:explicit-hat} and of Proposition~\ref{th:lr-Fisher} are in Appendix~\ref{sec:proofs-constrained}.
The proof of Lemma~\ref{th:modelsclose2} is in 
Section~\ref{sec:comp-nu}.

\subsection{Proof of Theorem~\ref{th:main}}
\label{sec:mainproof}

Recall \eqref{eq:sF+} and \eqref{eq:sF} for the definitions of $\left(s^{(F, +)}_n\right)_{n\in \{1,2,  \ldots\}}$ and   $\left(s^{(F)}_n\right)_{n\in \bbZ}$.
Let us set 
\begin{equation}
\label{eq:DgGdef}
D_\gG \, :=\, \bbE\left[\bP_{J, h}\left(\gs_0\neq s^{(F)}_0\right)\right] = \bbE \left[
\frac {1}{1+ \exp\left(s^{(F)}_0\left(l_{-1} +2h_0+ r_{1}\right)\right)} 
\right] + \frac12 \bbP \left(s^{(F)}_0 =0 \right) \, .
\end{equation}

Recall \eqref{eq:D_N} for $D_N(\cdot, \cdot)$. We remark that, since $s^{(F, +)}$ and $s^{(F)}$ differ only at finitely many sites (see \eqref{eq:stability-+}), \eqref{eq:main1} of 
Theorem~\ref{th:main} is equivalent to showing that
for almost every realization of $(h_n)$ 
\begin{equation}
\label{eq:main1-Theta-inv}
\lim_{N \to \infty} D_N \left(\gs,s^{(F)}\right)
\, =\, D_\gG\, ,
\end{equation} 
in 
$ \bP_{N, J, h}^{ab}$ probability. But \eqref{eq:main1-Theta-inv}
is a direct consequence of Proposition \ref{th:ergstatK}, so \eqref{eq:main1} is established
with explicit expression for $D_\gG$.

\smallskip

We are therefore left with showing \eqref{eq:main2}, that is
$D_\gG=O( \loglog \gG/ \gG)$. For this we first remark by \eqref{eq:generalergstat} that 
for almost every realization of $(h_n)_{n\in\bbZ}$
\begin{equation}
D_N\left(s^{(m)}, s^{(F)}\right)\underset{N\to \infty}{\longrightarrow} \bbP\left(s_0^{(m)}\neq s_0^{(F)}\right)\, .
\end{equation}
So, in view of Proposition~\ref{th:discrepancy0} and observing that $D_N(\sigma, s^{(F)})\leq D_N(\sigma, s^{(m)})+D_N(s^{(m)}, s^{(F)})$, it suffices to show that
\begin{equation}
\label{eq:PsmneqsF}
\bbP\left(s_0^{(m)}\neq s_0^{(F)}\right)\,=\,O\left(\frac{\loglog \gG}{\gG}\right)\, .
\end{equation}

To establish \eqref{eq:PsmneqsF} we use Lemma~\ref{th:modelsclose2} that yields
that with probability $1-O(1/\Gamma)$ and $\gG$ sufficiently large
    \begin{equation}
    \left| l_{-1}-\ls_{-1}\right |\, \le \,  2 \loglog \Gamma \qquad \text{ and } \qquad \left| r_1-\rs_1\right|\, \le \, 2 \loglog\Gamma\,.
    \end{equation}
which implies in particular that
\begin{equation}\label{lapproxls2}
\left|\left(l_{-1}+2h_0+r_1\right)-\left(\ls_{-1}+2h_0+\rs_1\right)\right|\,\le\, 4\loglog\gG\, .
\end{equation}
Recall now that $s_0^{(m)}=+1$ if $l_{-1}+2h_0+r_1\ge 0$ and $-1$ otherwise (see \eqref{eq:defsm}) and that $s_0^{(F)}=\sign(\ls_{-1}+2h_0+\rs_1)$, with the convention $\sign(0)=0$ (see Proposition \ref{th:lr-Fisher}), so on the event 
on which \eqref{lapproxls2} holds, $ s^{(m)}_0 \neq s^{(F)}_0$ implies $|l_{-1}+2h_0+r_1|\le 4\loglog\gG$, i.e., $l_{-1}$ belongs to $ [-2h_0-r_1-4\loglog\gG, -2h_0-r_1+4\loglog\gG]$. Furthermore, \eqref{eq:fctrl-inv} and the independence of $l_{-1}$, $h_0$ and $r_1$ imply that \begin{equation}
\bbP\left(l_{-1}\in \left[-2h_0-r_1-4\loglog\gG, -2h_0-r_1+4\loglog\gG\right]\right)\,=\,O \left(\frac {\loglog \gG}{\gG}\right)\, ,
\end{equation}
which yields \eqref{eq:PsmneqsF} and therefore \eqref{eq:main2}. The proof of Theorem~\ref{th:main} is therefore complete.
\qed

\section{Comparison estimates and the control of the invariant probability }
\label{sec:comp-nu}

\subsection{On the proximity of random walks with different reflection mechanisms}
\label{sec:comp}
We consider the Markov chain $(l_n)$: $(r_{-n+1})$ is the same process. We recall that it satisfies 
\begin{equation}
l_{n+1}\, =\, l_n+2h_{n+1}+ \log \left( \frac{1+ e^{-\gG -l_n-2h_{n+1}}}{1+ e^{-\gG +l_n+2h_{n+1}}} \right)\,=\, b_\gG\circ \theta_{h_{n+1}}(l_n).
\end{equation}
The main result is about the stationary sequence $(l_n)_{n \in \bbZ}$, but this is clearly inessential for the first result that applies  to 
$(l_n)_{n=n_0, n_0+1, \dots}$ with arbitrary law of $l_{n_0}$. 

Recall that 
  $(S_k)_{k\in \bbZ}$  is the random walk with $S_0=0$ and $S_{m}-S_n= \sum_{j=n+1}^m h_j$ for every $n< m$.  
\medskip

 \begin{lemma} 
 \label{th:forNrandom} 
 We have for all time $n_0$ and all $n\ge n_0$,  
 \begin{align}
 l_n \, &\le\, \log\left(e^{l_{n_0}+ 2 (S_n-S_{n_0})} + \gep \sum_{k={n_0}+1}^{n} e^{-2(S_k-S_{n})} \right),  \label{eq:ln<}  \\
  l_n \,&\ge\, -\log \left(e^{-(l_{n_0}+ 2 (S_n-S_{n_0}))}+ \gep \sum_{k={n_0}+1}^{n} e^{2(S_k-S_{n})} \right), \label{eq:ln>}
 \end{align}
 with the convention $\sum_\emptyset:=0$.
 \end{lemma}

\medskip

\begin{proof} 
 Recall that $\gep= e^{-\gG}$. We observe that for every $j$,
\begin{equation}
\exp \left(l_{j+1}\right)
\, =\,  \exp \left(  l_{j} + 2h_{j+1}\right)   \frac{1+ e^{-\gG -l_j-2h_{j+1}}}{1+ e^{-\gG +l_j+2h_{j+1}}} 
\, \le \, \exp \left(  l_{j}+ 2h_{j+1}\right)+ \gep\, .
\end{equation}
By iterating this bound starting from $j=n_0$ we obtain that for every $n=n_0, n_0+1, \ldots$
\begin{equation}
\begin{split}
   \exp \left(  l_{n}\right)\, &\le \,   \exp \left(  l_{n_0}\right)\prod_{k={n_0}+1}^{n} \exp(2h_k) 
   +\gep \sum_{k={n_0}+1}^{n} \prod_{i=k+1}^{n+1} \exp(2h_i)  
   \\ &= \,  \exp \left(  l_{n_0}\right) e^{2(S_{n}-S_{n_0})}+\gep \sum_{k={n_0}+1}^{n} e^{2(S_{n}-S_k)}\, ,
   \end{split}
   \end{equation}
which, by taking the logarithm, becomes \eqref{eq:ln<}. 
\smallskip

To obtain \eqref{eq:ln>} it suffices to remark that $(-l_n)_{n\in\bbZ}$
 satisfies the same recurrence relation as $(l_n)_{n\in\bbZ}$, with the only change that $S$ is replaced by $-S$ (see Remark \ref{rem:time-reversal}). Therefore \eqref{eq:ln>} follows from \eqref{eq:ln<}.
\end{proof}

\medskip
We are now ready to prove  Lemma~\ref{th:modelsclose2}.

\medskip

\begin{proof}[Proof of Lemma~\ref{th:modelsclose2}]
Choose $\kappa>0$. We give a proof of the bound on $ l_0-\ls_0$ in \eqref{lapproxls}: the result on $r_0-\rs_0$ follows by changing $S$ in $-\Srev$, and then applying a shift by 1 (see Remark \ref{rem:time-reversal}). 
For the bound on $ l_0-\ls_0$
we focus on the event that $\tsdown(\gG)>\tsup(\gG)$ (see \eqref{def-supdown} for the definitions of $\tsdown(\gG)$ and $\tsup(\gG)$). The result in the case $\tsup(\gG)>\tsdown(\gG)$ follows by changing $S$ to $-S$. 

\smallskip 

Recall from Proposition \ref{th:explicit-hat} that on the event $\{\tsdown(\Gamma)>\tsup(\Gamma)\}$, we have $\ls_0=\gG-2 S_{\tvdown(\gG)}$.
Our program is resumed in two (independent) bounds: on the event $\{\tsdown(\Gamma)>\tsup(\Gamma)\}$, excluding an event of probability $O(\gG^{-\kappa})$ and by suitably choosing a constant  $\widetilde{C}_\kappa<\infty$  \smallskip 
  \medskip
  \begin{itemize} [leftmargin=9pt]
  \item \emph{Upper bound}: exploiting that $l_{\tvdown(\gG)}\le \gG$ and that from time $\tvdown(\gG)$ to time 0 the process $l_\cdot $ evolves approximately according to $2 S_\cdot$ we aim at showing that
  \begin{equation}\label{eq:upp-lv}
  l_0\,\le\, \gG-2 S_{\tvdown(\gG)}+\log\left(1+\widetilde{C}_\kappa\log \gG\right)\,;
  \end{equation} 
  \item \emph{Lower bound}: exploiting that there is a rise in $S$ of size not smaller than $\gG$ between times $\tsdown(\Gamma)$ and $\tvdown(\gG)$ and that afterwards, up to time 0, the process $l_\cdot $ evolves approximately according to $2 S_\cdot$ we aim at showing that
  \begin{equation}\label{eq:low-lv} l_0 \, \ge \, \gG-2 S_{\tvdown(\gG)}-\log\left(1+\widetilde{C}_\kappa\log \gG\right)\, . 
  \end{equation}
  \end{itemize}
  \medskip

\medskip 
We proceed now with the proofs of the two bounds, which rely on Lemma \ref{th:forNrandom} and a control of the error terms provided by Lemma \ref{th:expomoments} in the appendix. The proofs are independent.
  
\subsubsection{Proof of the upper bound}
    By applying Lemma~\ref{th:forNrandom} with $n_0=\tvdown(\Gamma)$ and $n=0$,
   we obtain that
  \begin{equation}
  \begin{aligned}
  l_0\, & \le\, \log\left(e^{l_{\tvdown(\Gamma)}+2(S_0-S_{\tvdown(\Gamma)})} +\gep \sum_{n=\tvdown(\Gamma)+1}^0 e^{-2\left(S_n-S_0\right)}\right)\, \\
  & = \, \gG-2S_{\tvdown(\Gamma)}+ \log\left(e^{l_{\tvdown(\Gamma)}-\gG} + \sum_{n=\tvdown(\Gamma)+1}^0 e^{2\left(S_{\tvdown(\gG)}-S_n-\gG\right)}\right)\, .
  \end{aligned}
  \end{equation}
Since $l_{v^\downarrow(\gG)}  \in[-\gG, \gG]$, the first term inside the logarithm is bounded by 1.
It follows from Lemma \ref{th:expomoments}(ii) (applied to $\Srev$), that with probability $1-O(\Gamma^{-\kappa})$, on the event $\{\tsdown(\Gamma)>\tsup(\Gamma)\}$, we have 
\begin{equation}
 \sum_{n=\tvdown(\Gamma)+1}^0 e^{2\left(S_{\tvdown(\gG)}-S_n-\gG\right)}\,\le\, \widetilde{C}_\kappa\log \gG\, .
\end{equation}
The upper bound \eqref{eq:upp-lv} follows.

\subsubsection{Proof of the lower bound}
For the lower bound, we can actually work without the assumption that $\sd(\Gamma)>\supa(\Gamma)$. The bound will of course remain true on the event $\{\sd(\Gamma)>\supa(\Gamma)\}$. 
By applying Lemma~\ref{th:forNrandom} with $n_0=\sd(\Gamma)$ and $n=0$, we have that
  \begin{equation}\label{eq:llow} 
  \begin{aligned}
   l_0 & \, \ge \,  -\log\left(e^{-\left(l_{\tsdown(\gG)}+2(S_0-S_{\tsdown(\Gamma)})\right)}+\gep \sum_{n=\tsdown(\Gamma)+1}^{0} e^{2\left(S_n-S_0\right)}\right)\\   
   & \, = \, \gG-2S_{\tvdown(\Gamma)}-\log\left(e^{\gG-l_{\sd(\Gamma)}-2(S_{\tvdown(\Gamma)}-S_{\tsdown(\Gamma)})}+\sum_{n=\tsdown(\Gamma)+1}^{0} e^{2\left(S_n-S_{\tvdown(\gG)}\right)}\right)\, ,
   \end{aligned}
  \end{equation}
  Since $l_{\tsdown(\Gamma)}\in[-\gG, \gG]$ and $S_{\tvdown(\Gamma)}-S_{\tsdown(\Gamma)}\ge\Gamma$, the first term inside the logarithm is bounded by 1.
It follows from Lemma \ref{th:expomoments} (i) (applied to $\Srev$) that with probability $1-O(\Gamma^{-\kappa})$ we have
\begin{equation}
 \sum_{n=\tsdown(\Gamma)+1}^{0} e^{2\left(S_n-S_{\tvdown(\gG)}\right)}\,\le\,  \widetilde{C}_\kappa\log \gG\, .
\end{equation}
This completes the proof of the lower bound \eqref{eq:low-lv} and therefore also the proof of Lemma~\ref{th:modelsclose2} is complete.
\end{proof}

\subsection{An estimate on the invariant probability}
\label{sec:nu}

To establish \eqref{eq:fctrl-inv} we use Lemma~\ref{th:modelsclose2} to see that for every interval $I=[a, b] \subset [-\Gamma, \Gamma]$
 \begin{equation} 
 \label{eq:pre-obs-inv}
 \p\left(r_1 \in [a, b]\right) \le O(\Gamma^{-1}) + \p\left( \widehat r_1 \in [a- 2 \log\log \gG, b+ 2\log \log \gG\right)\, ,   
\end{equation}
so we reduce the issue to estimating $\p\left(\rs_1 \in [a, b]\right)$ and we recall that for $\rs_1$ we have the expression  \eqref{eq:explicit-hat-r}. We recall also that the processes $\ls$ and $\rs$ are stationary sequences with the same marginal law. A simulation of such an invariant law is in Fig.~\ref{fig:inv-p}.

\medskip

\begin{lemma} 
\label{th:ctrl-inv}
  There exists  $c_1>0$ and $\gG_0>0$ such that for every $I =[a, b]\subset [-\Gamma, \Gamma]$ and  $\Gamma\ge \Gamma_0$   \begin{equation}
  \label{eq:ctrl-inv}
    \p\left(\widehat r_1\in I\right) \le c_1 \frac{1+b-a}{\Gamma}\, ,
 \end{equation}  
 and
\begin{equation}
  \label{eq:ctrl-inv2}
    \p\left( r_1\in I\right) \le c_1 \frac{5\loglog \gG+b-a}{\Gamma}\, ,
 \end{equation}   
\end{lemma}

\medskip 

\begin{proof}
 Let us introduce the (strictly) ascending ladder epochs: $\varrho_0:=0$ and   \begin{equation}     \varrho_k:= \inf\{n> \varrho_{k-1}: S_n > S_{\varrho_{k-1}}\}, \qquad k=1, 2, \ldots
  \label{def-rhok}\end{equation}
     
\noindent By definition of $\tudown(\gG)$ in \eqref{eq:defudown},  we have that  $\tudown(\gG)=\varrho_ {\mathcal K}$ with  \begin{equation}   
   {\mathcal K} \,:= \,  \inf\{k\ge 0: \max_{\varrho_k \le n < \varrho_{k+1}} (S_{\varrho_k}-S_n) \ge \Gamma\}\,. \label{j*} 
    \end{equation} 
 
\noindent Consider the case $\{\ttdown(\Gamma)<\ttup(\Gamma)\}$.  By  \eqref{eq:explicit-hat-r}, $\widehat r_1= -\Gamma+ 2 S_{\varrho_{\mathcal K}}$. 
It follows that 
\begin{equation} 
\begin{split}
\p\left(\widehat r_1\in [a, b], \, \ttdown(\Gamma)<\ttup(\Gamma)\right) \,
&\le\, 
\p\left(-\gG+ 2 S_{\varrho_{\mathcal K}} \in [a, b]\right) \,\\
&\le\, 
\sum_{k=0}^\infty \p\left( -\Gamma+ 2 S_{\varrho_k} \in [a, b], \min_{\varrho_k \le j < \varrho_{k+1}} S_j \le S_{\varrho_k} - \Gamma\right)
\\ &=\, 
 \sum_{k=0}^\infty \p\left( -\Gamma+ 2 S_{\varrho_k} \in [a, b]\right) \, \p\left( \min_{ \varrho_{k}\le   j < \varrho_{k+1}} S_j \le S_{\varrho_k}  - \Gamma\right)
 \\
&=:\,  U\left( \left[(a+\Gamma)/2, (b+\Gamma)2\right]\right) \, \p\left( \min_{0\le  j < \varrho_1} S_j \le - \Gamma\right),
\end{split}
 \end{equation}
 with standard definition of the renewal function $U$ associated with the increasing random walk, or renewal process,  $S_{\varrho_k}$. By  \cite[Ch.~IV, p.~199]{cf:Spitzer} we know that $\e[S_{\varrho_1}]<\infty$ and since $U([x, y]) \le U([0, y-x])$ (see for example \cite[Th.~2.4, Ch. V]{cf:Asm}), the elementary version
of the Renewal Theorem, i.e.  $U([0,x])\sim x/\e[S_{\varrho_1}]$ for $x \to \infty$ (a consequence of  the Law of Large Numbers),  implies that there exists  $c_{2}>0$ such that for every $0\le x <y$ 
\begin{equation}  U([x, y])\,  \le\,  c_{2} (1+ y-x)\,.  
\end{equation} 
 Since by  \eqref{exitproba1} we have  $\p( \min_{0\le j < \varrho_1} S_j \le - \Gamma) \le {c_4}/{\Gamma}$, we obtain
 \eqref{eq:ctrl-inv}. And, by \eqref{eq:pre-obs-inv}, from \eqref{eq:ctrl-inv}  we obtain \eqref{eq:ctrl-inv2}.
 The proof of 
 Lemma~\ref{th:ctrl-inv} is therefore complete. 
\end{proof}

\begin{figure}[h]
\begin{center}
\includegraphics[width=15 cm]{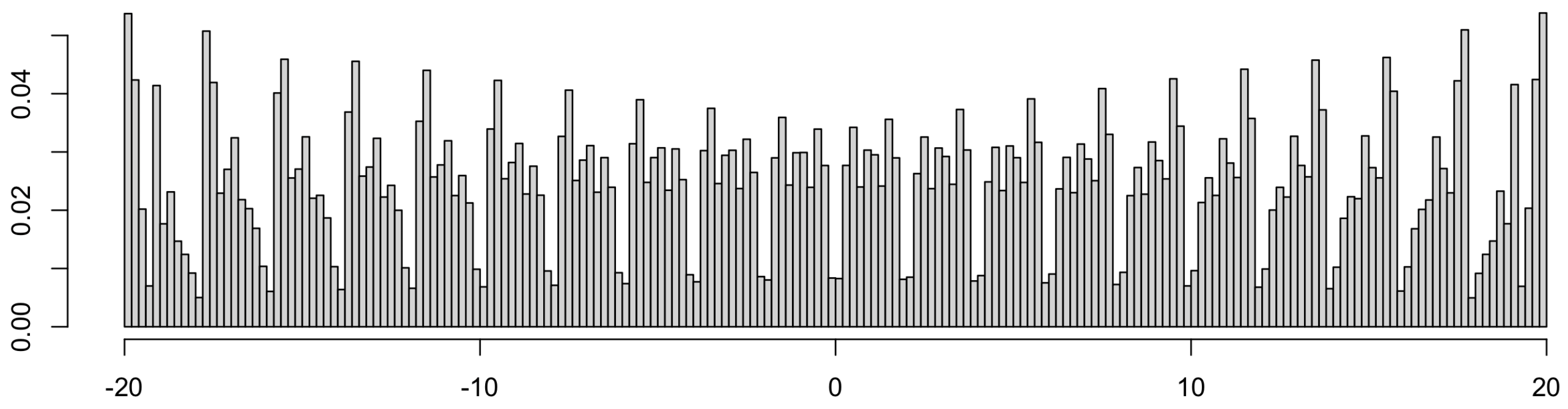}
\includegraphics[width=15 cm]{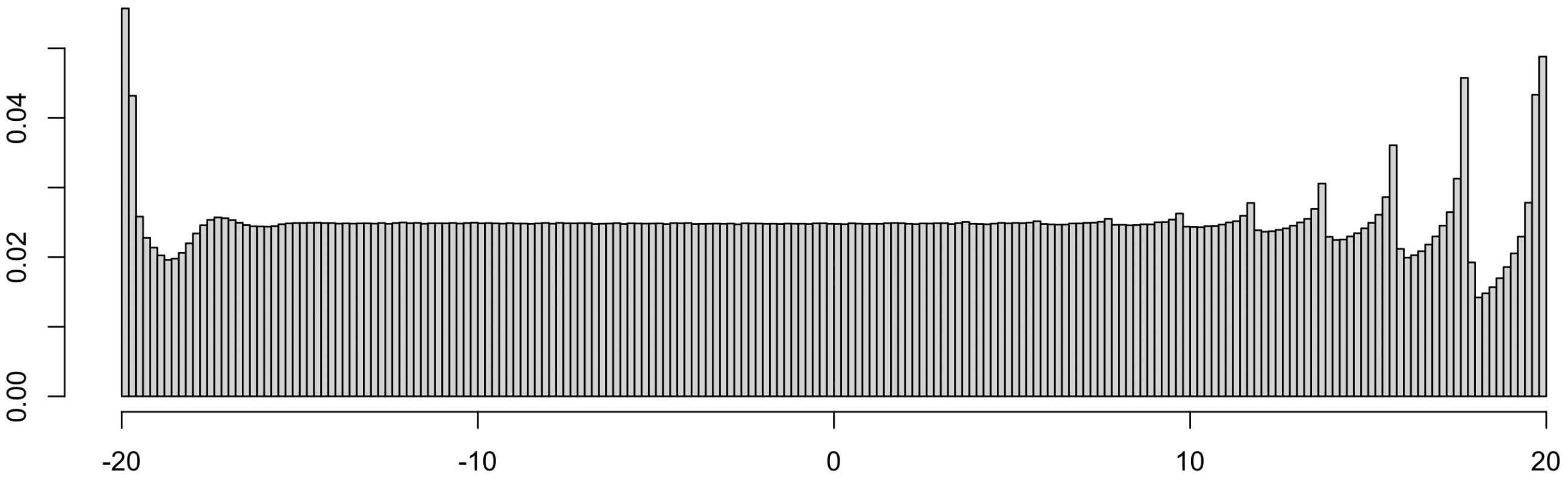}
\end{center}
\caption{\label{fig:inv-p} A simulation of the $(l_n)_{n=0,\ldots,N}$, with $l_0=0$ and $N=  5\cdot 10^7$, $\gG=20$. In the first case 
the law of $h_1$ is $(\gd_{-2}+\gd_{+2})/2$ and in the second case the law is $(\gd_{-2}+ \gamma)/2$  where $\gamma$ is the law of a Gaussian 
of mean $2$ and variance $1/ \sqrt{2}$. {These two histograms are expected  to give an idea of how the invariant probability looks like:
in fact, the simulation is very stable with respect to increasing $N$ and changing the randomness. Note that this empirical observation is compatible in the second case with  the invariant probability being very close to (a multiple) of the Lebesgue measure away from the boundaries (this result is proven in \cite{cf:GG22} under stronger regularity properties of the law of $h_1$). However, in the first case a periodic phenomenon seems to be present and    proximity to the Lebesgue measure appears to be plausible only if averages over large boxes are taken}.
} 
\end{figure}

\section{Proof of Proposition \ref{th:ergstatK}}
\label{sec:proofergstat}

We are going to prove Proposition \ref{th:ergstatK}. So we are given a mesurable function $K:\bbR^\bbZ\to \{-1, 0, +1\}$, and we set $s_n^{(K)}=K(\Theta^{\circ n} h)$.
Our aim is to show the convergence of 
\begin{equation}
D_N\left(\sigma, s^{(K)}\right)=\frac1N\ \left \vert \left\{ n=1,2 \ldots, N:\, \gs_n \neq s_n^{(K)}\right\} \right \vert=\frac1N \sum_{n=1}^N \ind_{\gs_n \neq s_n^{(K)}}\, , 
\end{equation}
which we will write for short $D_N$, towards 
\begin{equation}
\bbE\left[\frac1{1+\exp\left(s^{(K)}_0 m_0\right)}\right]+\frac12 \bbP\left(s_0^{(K)}=0\right)\,,
\end{equation}
in $\bP^{ab}_{N, J, h}$-probability, as well as in $\bP_{J, h}$-probability.

Our strategy is to get rid of the boundary effects to turn to the infinite volume model, and then use ergodicity. Paramount to our proof is the uniform contractive estimate \eqref{eq:17JLct} in the proof of Lemma \ref{th:from-infty}. We reproduce it here:
for every $k\ge 0$ and every $h_1, \dots, h_k\in\R$
\begin{equation}
\label{eq:lipshitzbound}
\sup_{s,t\in [-\gG,\gG]}
\left|f_{h_k}\circ f_{h_{k-1}} \circ \dots \circ f_{h_1}(t)-f_{h_k}\circ f_{h_{k-1}} \circ \dots \circ f_{h_1}(s)\right|  \, \le \, 2\gG  \exp(-k\gep) \, .
\end{equation}
Of course this observation  directly applies to control the  difference of  two $l$ (or $r$) processes  started at different initial values
and/or different initial times. 

\medskip

We first have a look at the expectation of $D_N$ under $\bP_{N, J, h}^{ab}$. 
We are going to get rid of the boundaries by turning to the infinite volume model; precisely we are going to show that:
\begin{equation}\label{eq:proof33aim1}
 \bE^{ab}_{N, J, h}\left[D_N\right]
\, =\, \bE_{J, h}\left[D_N\right]+O\left(\frac1N\right),
\end{equation}
where $O\left(1/N\right)$ denotes a quantity whose absolute value is bounded by $1/N$ multiplied by a constant depending on $\gG$.
Observe that the quantity on the left-hand-side of \eqref{eq:proof33aim1} is:
\begin{equation}\label{eq:proof33linearity}
\begin{aligned}
 \bE^{ab}_{N, J, h}\left[D_N\right]
&\,  = \, \frac 1N \sum_{n=1}^N \bP^{ab}_{N, J, h}\left( \gs_n \neq s^{(K)}_n\right) \\
 & \, = \, \frac 1N \sum_{n=1}^N \frac {1}{1+ \exp\left( s^{(K)}_n\, m^{(ab)}_{1, n, N}\right)}+\frac12 \ind_{s_n^{(K)}=0}
\end{aligned}
\end{equation}
and similarly, $\bE_{J, h}\left[D_N(\sigma)\right]$ has the same expression with $m^{(ab)}_{1, n, N}$ replaced by $m_n$. We recall that $m^{(ab)}_{1, n, N}=l^{(a)}_{1, n-1} +2h_n+ r^{(b)}_{n+1, N}$ and that $m_n=l_{n-1} +2h_n+ r_{n+1}$.
Using that $x\mapsto 1/{(1+\exp(x))}$ is 1-Lipschitz and then \eqref{eq:lipshitzbound} we deduce, for $n=1, \dots, N$,
\begin{equation}\label{eq:contrlboundeffect}
\begin{aligned}
\left \vert  
\frac 1 {1+ \exp\left( s^{(F)}_n m^{(ab)}_{1, n, N}\right)} 
-\frac 1 {1+ \exp\left( s^{(F)}_n m_n\right)}
\right \vert
& \le |m^{(ab)}_{1, n, N}-m_n|\\
& \le |l^{(a)}_{1, n-1}-l_{n-1}| + |r^{(b)}_{n+1, N}-r_{n+1}| \\
& \le 2\gG  \left(\exp(-(n-1)\gep)+ \exp(-(N-n)\gep)\right) ,
\end{aligned}
\end{equation}
and therefore \eqref{eq:proof33aim1} follows by summing over $n=1, \dots, N$.

Now, let us turn to controlling the second moment of $D_N$ under $\bP^{ab}_{ N, J, h}$. We have
\begin{equation}\label{eq:controlsecondmom}
\begin{aligned}
\bE^{ab}_{N, J, h}\left[D_N^2\right]=\frac1N\bE^{ab}_{N, J, h}[D_N]+\frac2{N^2}\sum_{1 \le n < m\le N} \bP^{ab}_{N, J, h}\left( \gs_n \neq s^{(K)}_n , \gs_m \neq s^{(K)}_m\right
).
\end{aligned}
\end{equation}
We note that for $1 \le n < m\le N$ 
\begin{equation}\label{eq:forcontrolsecondmom}
\bP_{ N, J, h}^{ab}\left( \gs_n \neq s^{(K)}_n , \gs_m \neq s^{(K)}_m\right)
\,=\,
\bP_{N, J, h}^{ab}\left( \gs_m \neq s^{(K)}_m\right)
\, \bP_{m-1, J, h}^{a(-s^{(K)}_m)}\left( \gs_n \neq s^{(K)}_n\right).
\end{equation}
Proceeding as in \eqref{eq:contrlboundeffect}, we use \eqref{eq:lipshitzbound} to bound the difference between $\bP_{N, J, h}^{ab}\left( \gs_m \neq s^{(K)}_m\right) $ and $\bP_{J, h}\left( \gs_m \neq s^{(K)}_m\right)$. Analogously we  bound  the difference between 
 $ \bP_{m-1, J, h}^{a(-s^{(K)}_m)}\left( \gs_n \neq s^{(K)}_n\right)$ and $ \bP_{J, h}\left( \gs_n \neq s^{(K)}_n\right)$. Observing that $|xy-x'y'|\le |x-x'|+|y-y'|$ when $x,y, x', y'\in [0, 1]$, we derive from \eqref{eq:controlsecondmom} and \eqref{eq:forcontrolsecondmom} that
\begin{equation}\label{eq:controlsecondmoment}
\bE^{ab}_{N, J, h}\left[D_N^2\right]=\left(\bE_{J, h}\left[D_N\right]\right)^2+O\left(\frac1N\right)\, .
\end{equation}
Putting together \eqref{eq:proof33aim1} and \eqref{eq:controlsecondmoment} and using that $D_N$ is bounded, the variance of $D_N$ under $\bP^{ab}_{ N, J, h}$ is $O\left(1/N\right)$. Furthermore, 
$\bE^{ab}_{N, J, h}[D_N]$ converges to the announced quantity since $\bE_{J, h}[D_N]$ does, as already mentioned in \eqref{eq:ergstatconvinexpect}. 
This completes the proof of
Proposition \ref{th:ergstatK} for what concerns the statements in $\bP_{N,J, h}^{ab}$-probability. 
The proof in $\bP_{J, h}$-probability can be obtained either by modifications of the argument we just developed or
by remarking that in reality 
 we have established the convergence for all boundary conditions in $\{-1, +1\}^2$, even random ones,  and 
 from this the case of $\bP_{J, h}$ follows. 
 \qed

\appendix 

 \section{Random walk estimates}
 
\label{sec:A}
Recall that, by \eqref{hyp-S} and \eqref{eq:Sn},  $(S_n)$ is a random walk with centered finite variance increments and $S_0=0$. In this section we consider also the case in which
$S_0=x\in \bbR$ and the law of this process is denoted by $\bbP_x$. When $x=0$ we just drop the dependence on $x$: $\bbP=\bbP_0$.
We will use the following estimate several times, referred to as the Gambler's ruin problem. Let  
\begin{equation}
T^+_0\,:=\, \inf\left\{n\ge 1: S_n >0\right\}\,.
\end{equation}
 By  \cite[Theorem 5.1.7]{cf:LaLi} there exist  positive constants  $x_0$, $c_3$ and $c_4$ such that  for any $x\ge x_0, $
   \begin{equation}
     \,\p\left( \min_{0\le k\le T^+_0}S_k \le -x \right) \,\in\, \left[\frac{c_3}x,   \frac{c_4}x\right]. 
        \label{exitproba1}\end{equation}

The objective of this subsection is to establish Lemma \ref{th:expomoments}, which serves as the main technical tool in proving Lemma \ref{th:modelsclose2}.

We begin with a general fact that follows from the study in \cite{Goldie-Grubel} on a stochastic recurrence equation, in the case  where the solution exhibits a thin tail. 

\begin{lemma}  
\label{l:subexpo} Let $(\eta_j, \Delta_j)_{j\ge 1}$ be an IID sequence of random vectors distributed like $(\eta, \Delta)$. Suppose that $\eta$ and $\Delta$ are nonnegative random variables such that $\p(\Delta>0)>0$ and  that  $\e e^{c \, \eta }< \infty$ for some $c>0$. Then there exists $\rho>0$  such that
  \begin{equation} 
  \sup_{k\ge 0} \e \exp\left( \rho \, \sum_{j=1}^{k+1} \eta_j e^{- \sum_{\ell=j}^{k} \Delta_\ell}  \right)<\infty\, ,
    \end{equation}
with the convention $ \sum_{\ell=k+1}^{k} \Delta_\ell=0$.
\end{lemma}

\begin{proof}
 Let $Y_j:= \Delta_1+...+\Delta_j$ for any $j\ge 1$ and  $Y_0:=0$. 
 Since $(\eta_{k+1-j}, \gD_{k+1-j})_{j=1, \ldots, k+1}$ has the same law as $(\eta_j, \gD_j)_{j=0,1, \ldots, k}$,  we see that, for each fixed $k\ge 0$,  $\sum_{j=1}^{k+1} \eta_j e^{- \sum_{\ell=j}^{k} \Delta_\ell}$  has the same law as $\sum_{\ell=0}^k \eta_\ell e^{-Y_\ell} $, which is bounded above by $\eta_0+\sum_{\ell=1}^\infty \eta_\ell e^{-Y_{\ell-1} } $. By \cite{Goldie-Grubel}, Theorem 2.1, there exists some $\rho>0$ such that $\e \exp(\rho \sum_{\ell=1}^\infty \eta_\ell e^{-Y_{\ell-1} } ) < \infty$, yielding the Lemma. \end{proof}

\medskip

The following estimate ensures the existence of small exponential moments, a requirement for the assumption in Lemma \ref{l:subexpo}.

\medskip

\begin{lemma}
\label{l:rw1}
  Assume \eqref{hyp-S}.  There exist  positive constants $\rho$ and $c_5$ such that for every $x\le 0$  
  \begin{equation} \e_x \left[ e^{\rho  \sum_{n=0}^\infty e^{ 2 S_n} \ind_{\{\max_{0\le i\le n} S_i \le  0\}} }\right]\, \le\,  c_5 (1+|x|)\, .  
  \end{equation}
   \end{lemma}

\medskip

\begin{proof} 
By applying \cite[Corollary 4.2]{kersting-vatutin}  { for  $\lambda<2$} we obtain that there exists  $c_6>0$ such that for every $x \le 0$
\begin{equation*} 
\e_x \left[e^{2 S_n}  (1+ |S_n|) \ind_{\{\max_{0\le i\le n} S_i  \le  0\}}\right] 
\,\le\,  c_6 \, (1+|x|)\, n^{-3/2} \, ,  
 \end{equation*}
  which yields that there exists $c_7>0$ such that
  \begin{equation} 
  \e_x \left[\sum_{n=0}^\infty  e^{2 S_n}  (1+ |S_n|) \ind_{\{\max_{0\le i\le n} S_i \le  0\}}\right]\, \le\,  c_7\,(1+|x|), \qquad \text{for every } x\le 0\,. 
  \label{eq:sume2Sn} 
  \end{equation}
\noindent For every integer $p\ge 2$, applying the Markov property successively at ${\ell_{p-1}}, ..., \ell_1$ in the following sum, we deduce from \eqref{eq:sume2Sn} that for any $x\le 0$,  \begin{equation*} 
\sum_{0\le \ell_1\le ...\le  \ell_p}\e_x \left[ \prod_{i=1}^{p} e^{2 S_{\ell_i}} \ind_{\{\max_{0\le j \le \ell_p} S_j \le 0\}}\right]
\,\le\,   c_7^p \, (1+|x|)\,. 
 \end{equation*}

\noindent It follows that for every integer $p\ge 1$ and $x\le 0$ \begin{align*}
 \e_x  \Bigg[ \Bigg(\sum_{n=0}^\infty e^{ 2 S_n}  \ind_{\{\max_{0\le i\le n} S_i \le  0\}}\Bigg)^p\Bigg]
\, & \le\,  p! \sum_{0\le \ell_1\le ...\le  \ell_p}\e_x \left[ \prod_{i=1}^{p} e^{2 S_{\ell_i}} \ind_{\{\max_{0\le j \le \ell_p} S_j \le 0\}}\right]
\\
\, & \le\, 
p! \,   c_7^p \, (1+|x|)\,,   
  \end{align*}

%
%


\noindent which implies that for every $\rho \in (0, 1/c_7)$
 \begin{equation*}
 \e_x\left[e^{\rho \sum_{n=0}^\infty e^{ 2 S_n} \ind_{\{\max_{0\le i\le n} S_i \le  0\}}}\right]\, \le \, \frac{1+|x|}{1- \rho  \, c_7},
 \end{equation*}
and the proof of Lemma~\ref{l:rw1} is complete.
\end{proof}
\medskip 

Below are our main technical estimates: they are  used in the proof of Lemma \ref{th:modelsclose2}. {Recall \eqref{eq:t-down-up} for the definitions of $\td(\gG)$ and $\tup(\gG)$ and \eqref{eq:defudown} for $\tudown(\Gamma)$, as well as \eqref{eq:completenotation1}, \eqref{eq:completenotation2} for ${\tt t}_1(\Gamma)$ and $\ss_1(\Gamma)$. In particular, $\td(\gG)= {\tt t}_1(\Gamma)$ and $\tudown(\Gamma)=\ss_1(\Gamma)$.} 
\medskip

\begin{lemma}\label{th:expomoments}  Assume \eqref{hyp-S}. Then for every $\kappa>0$, there exists $\widetilde{C}_\kappa>0$ such that we have
\smallskip 

(i) 
\begin{equation}
	\bbP\left(\sum_{n=0}^{\td(\Gamma)-1} e^{2\left(S_n-S_{\tudown(\Gamma)}\right)}
	\ge  \widetilde{C}_\kappa \log \gG\right)\,=\,O\left(\gG^{-\kappa}\right)\,. \label{expon1}
\end{equation}	


(ii) 
\begin{equation}
\bbP\left(\sum_{n=0}^{\tudown(\gG)-1}  
e^{2\left(S_{\tudown(\gG)}-S_n-\Gamma\right)}
\ge  \widetilde{C}_\kappa \log \gG \, , \, \td(\gG)<\tup(\gG)\right)\,
=\,O\left(\gG^{-\kappa}\right)\,. \label{expon2} 
\end{equation}
\end{lemma}

\medskip
\begin{proof} 
We start with item $(i)$. Recall \eqref{def-rhok}, \eqref{j*} and $\tudown(\Gamma)=\varrho_{\mathcal K}$.     By definition,     
   we have the following bound
 \begin{equation} \label{leqK}
 \sum_{n=0}^{\td(\Gamma)-1} e^{2 \left(S_n- S_{\tudown(\Gamma)}\right)}\, \le \, 
\sum_{j=1}^{{\mathcal K}+1}
 \sum_{n=\varrho_{j-1}}^{\varrho_{j}-1} e^{ 2 \left(S_n- S_{\varrho_{\mathcal K}}\right)} .\end{equation}

\noindent For every {{ $j= 1, \dots, \mathcal K+1$}} we have
\begin{equation} 
\sum_{n=\varrho_{j-1}}^{\varrho_{j}-1} e^{ 2 (S_n- S_{\varrho_{\mathcal K}})} =e^{2 (S_{\varrho_{j-1}}- S_{\varrho_{\mathcal K}})} \, \sum_{n=\varrho_{j-1}}^{\varrho_{j}-1} e^{ 2 (S_n- S_{\varrho_{j-1}})} =  \eta_{j} e^{-   \sum_{\ell=j}^{\mathcal K} \Delta_\ell}, 
 \end{equation}
 with {{ the convention $\sum_{\ell= {\mathcal K}+1}^{\mathcal K} \Delta_\ell=0$ and where we denoted}} for $j=1,2, \ldots$  
 \begin{equation}
 \eta_j\,:=\, \sum_{n=\varrho_{j-1}}^{\varrho_{j}-1} e^{ 2 (S_n- S_{\varrho_{j-1}})}\ \text{ and } \  
 \Delta_j\,:=\,  2( S_{\varrho_j}- S_{\varrho_{j-1}})>0\, .
  \end{equation}
  
Remark that the law of ${\mathcal K}$ is geometric: 
\begin{equation} 
\p({\mathcal K}=k)\,=\,  p_\Gamma (1-p_\Gamma)^{k}, \qquad k= 0,1, \ldots , 
 \end{equation}
with  $p_\Gamma\,:=\,\p\left( \min_{0\le k\le \varrho_1} S_k \le -\Gamma \right) $. By \eqref{exitproba1}, $p_\Gamma\,\in\, \left[\frac{c_3}\Gamma,   \frac{c_4}\Gamma\right] $    for all  large $\Gamma$. It follows that
 \begin{equation} \p\left({\mathcal K} \ge \Gamma^2\right) \,\le\,  \Gamma^{-\kappa}\,. 
   \label{j*>}\end{equation}

  \noindent By \eqref{leqK} and \eqref{j*>},   
  \begin{equation}  \label{I_1}
  \p\Big(\sum_{n=0}^{\td(\Gamma)-1} e^{2 \left(S_n- S_{\tudown(\Gamma)}\right)} \ge \widetilde{C}_\kappa \log \Gamma\Big) \le
 \Gamma^{-\kappa} +  \sum_{k=0}^{\Gamma^2-1} \p\Big(\sum_{j=1}^{k+1} \eta_j e^{-   \sum_{\ell=j}^k \Delta_\ell} \ge C_\gk' \log \Gamma\Big)   . 
 \end{equation}

 Note that $(\eta_j, \Delta_j)_{j\ge 1}$ are IID and that $\eta_j$ is distributed as $\sum_{n=0}^{\varrho_1-1} e^{2 S_n}$ which has finite small exponential moments by  Lemma \ref{l:rw1}. Therefore we can apply Lemma \ref{l:subexpo} and obtain that for some positive constant $c_8$, 
\begin{equation}  \p\Big(\sum_{n=0}^{\td(\Gamma)-1} e^{2 \left(S_n- S_{\tudown(\Gamma)}\right)} \ge \widetilde{C}_\kappa \log \Gamma\Big)  
\le
 \Gamma^{-\kappa} +    c_8\, \Gamma^2\,   e^{-\rho \,  \widetilde{C}_\kappa \log \Gamma}\, , 
 \end{equation}
 yielding \eqref{expon1}, as soon as $\widetilde{C}_\kappa>(\gk+2)/\rho$.  

 \medskip

For what concerns $(ii)$  first we claim that there exists some $b=b(\kappa)>0$ such that for every large $\Gamma$
 \begin{equation}  
 \p\left(\tudown(\Gamma)\ge \Gamma^b\right) \,\le \, 2\, \Gamma^{-\kappa}\,.
  \label{eq:u>} 
  \end{equation}
 In fact, we use  $\tudown(\Gamma)=\varrho_{\mathcal K}$ with ${\mathcal K}$ given in \eqref{j*} and estimated in \eqref{j*>}: we obtain
 \begin{equation}
   \p\left(\tudown(\Gamma)\ge \Gamma^b\right) \,\le \, \Gamma^{-\kappa} + \p\left(\varrho_{\Gamma^2} \ge \Gamma^b\right)\,. 
    \label{eq:u>1}
 \end{equation}  
  By  \cite[Theorem 4.6]{kersting-vatutin}, there exists  $c_{9}>0$ such that 
  \begin{equation}  
  \p\left(\varrho_1  \ge  n\right) \stackrel{n \to \infty}\sim  \frac{c_{9}}{\sqrt{n}}\, . 
  \end{equation}
Since $\varrho_k$ is distributed as the sum of $k$ IID copies of $\varrho_1$, we have that there exists $c_{10}>0$ such that
\begin{equation}
 \p\left(\varrho_k \ge n\right) \,\le\,  k\, \p\left(\varrho_1 \ge  \frac{n}{k}\right)\,  \le\,  c_{10} \frac{k^{3/2}}{\sqrt{n}}\, .
 \end{equation}  
 This implies that  $\p(\varrho_{\Gamma^2} \ge \Gamma^b) \le c_{2} \Gamma^{3-b/2} \le \Gamma^{-\kappa}$ if we choose {{ $b> 2\kappa+6$}} ($\Gamma$ being large). Therefore \eqref{eq:u>} is established.

Now we give the proof of \eqref{expon2}. On the event $\{\td(\Gamma)< \tup(\Gamma)\}$,  for every $n \le \tudown(\Gamma)$, $S_{\tudown(\Gamma)} - S_n < \Gamma$ and  
\begin{equation}   
\p\left( \sum_{n=0}^{\tudown(\Gamma)-1}  
e^{2\left(S_{\tudown(\gG)}-S_n-\Gamma\right)}
 \ge \widetilde{C}_\kappa \log \Gamma, \,\td(\Gamma)< \tup(\Gamma)\right)
  \le 
  \p\left(\tudown(\Gamma)\ge \Gamma^b\right)  + \sum_{k=0}^{\Gamma^b-1} J_k,
   \label{eq:sum_u}
  \end{equation}
where 
\begin{equation}
J_k\,:=\,   \p\left(\sum_{n=0}^{k-1} e^{2 \left(S_k-S_n-\Gamma\right)} \ge \widetilde{C}_\kappa \log \Gamma , \,  \max_{0\le n < k} (S_k-S_n) < \Gamma\right)\,.
\end{equation}
 Since $(S_k-S_{k-j})_{j=0,\ldots, k-1}$ is distributed like $(S_j)_{j=1,\ldots, k}$, we get that 
 \begin{eqnarray}  J_k &=& \p\left(\sum_{j=1}^k e^{2 (S_j-\Gamma)} \ge \widetilde{C}_\kappa \log \Gamma , \,  \max_{1\le j\le k} S_j < \Gamma \right)
 \nonumber
\\
&=& \p_{-\Gamma}\left(\sum_{j=1}^k e^{2 S_j} \times \ind_{\{\max_{1\le i\le k} S_i< 0\}} \ge \widetilde{C}_\kappa \log \Gamma \right) 
\nonumber \\
&\le&
\p_{-\Gamma}\left(\sum_{j=0}^\infty e^{2 S_j} \ind_{\{\max_{1\le i\le j} S_i< 0\}} \ge \widetilde{C}_\kappa \log \Gamma \right) 
\, \le \,  \Gamma^{-\kappa-b}, 
\end{eqnarray}
for large $\Gamma$: the last inequality holds for $\widetilde{C}_\kappa$ large enough, depending on $\kappa$ and is due to Lemma \ref{l:rw1}. By \eqref{eq:sum_u} and \eqref{eq:u>}, we get that 
\begin{equation}
\p\left( \sum_{n=0}^{\tudown(\Gamma)} e^{-2 (S_{\tudown(\Gamma)}-S_n-\gG)} \ge \widetilde{C}_\kappa\log \Gamma, \,\td(\Gamma)< \tup(\Gamma)\right) \,\le\, 3\, \Gamma^{-\kappa}\, ,
\end{equation}
and the proof of Lemma~\ref{th:expomoments} is complete.
\end{proof}

\medskip

The tools introduced in this section are useful to establish also the following result: set  ${\tt t}_0(\Gamma):=0$ and recall the definition 
of the sequences of random times $({\tt t}_n(\Gamma))_{n\in \bbN}$ and $(\ss_n(\Gamma))_{n\in \bbN}$ from Section~\ref{sec:Fisher}.

\medskip

\begin{lemma}
\label{th:ren-u}
   The two sequences of random vectors  in $\bbR^2$   
   \begin{equation}
    \left(\ss_n(\Gamma)- {\tt t}_{n-1}(\Gamma), S_{\ss_n(\Gamma)}- S_{{\tt t}_{n-1}(\Gamma)}\right)_{n\in \bbN}
      \text{ and }
      \left({\tt t}_n(\Gamma)- \ss_n(\Gamma), S_{{\tt t}_n(\Gamma)}- S_{\ss_n(\Gamma)}\right)_{n\in \bbN}
      \end{equation}   
   are independent and, both of them, are sequences of independent random vectors.     
       Moreover the two (sub)sequences 
 \begin{equation}
    \left(\ss_n(\Gamma)- {\tt t}_{n-1}(\Gamma), S_{\ss_n(\Gamma)}- S_{{\tt t}_{n-1}(\Gamma)}\right)_{n\in 2\bbN}
      \text{ and }
      \left({\tt t}_n(\Gamma)- \ss_n(\Gamma), S_{{\tt t}_n(\Gamma)}- S_{\ss_n(\Gamma)}\right)_{n\in 2\bbN}
      \end{equation}         
  are IID and the same holds if we consider the (sub)sequences with $n \in 2\bbN-1$.
 \end{lemma}
 \medskip

\begin{proof}
Using the fact that $\ss_1(\Gamma)=\varrho_{\mathcal K}$ with ${\mathcal K}$ defined in \eqref{j*}, we remark that, conditioned on $\sigma( \ss_1(\Gamma),  (S_j)_{j=1, 
\ldots,\ss_1(\Gamma)})$, the law of $(S_{\ss_1(\Gamma)+i}-S_{\ss_1(\Gamma)})_{i=0, \ldots, {\tt t}_1(\Gamma)}$ 
coincides with the law of  $(S_i)_{i=0, \ldots,  T_S(-\Gamma)}$ conditioned on $\{T_S(-\Gamma)< \varrho_1\}$, where $T_S(-\Gamma):= \min\{n\ge 0: S_n \le -\Gamma\}$ and $\varrho_1$ is defined in \eqref{def-rhok}. It follows that $({\tt t}_1(\Gamma)-\ss_1(\Gamma),  S_{{\tt t}_1(\Gamma)}- S_{\ss_1(\Gamma)})$ is   independent of $(\ss_1(\Gamma), S_{\ss_1(\Gamma)})$.

Since ${\tt t}_2(\Gamma)-{\tt t}_1(\Gamma)$ is the first time of $\Gamma$-increase of  $(S_{{\tt t}_1(\Gamma)+i}- S_{{\tt t}_1(\Gamma)})_{i=0, 1, \ldots}$, and $\ss_2(\Gamma)-{\tt t}_1(\Gamma)$  is the corresponding time at which  the minimum $\min_{0\le i \le {\tt t}_2(\Gamma)-{\tt t}_1(\Gamma)} (S_{{\tt t}_1(\Gamma)+i}- S_{{\tt t}_1(\Gamma)})$ is achieved, we deduce from the strong Markov property that  $(\ss_2(\Gamma)-{\tt t}_1(\Gamma), S_{\ss_2(\Gamma)}- S_{{\tt t}_1(\Gamma)})$  is independent of $\sigma(S_n, n=0, \ldots, {\tt t}_1(\Gamma))$, in particular independent of $({\tt t}_1(\Gamma)-\ss_1(\Gamma),  S_{{\tt t}_1(\Gamma)}- S_{\ss_1(\Gamma)})$ and $(\ss_1(\Gamma), S_{\ss_1(\Gamma)})$. 

The proof is completed by iterating the arguments above.
\end{proof}

\section{Proofs for the reflected random walk: analytic approach}\label{sec:appendanalytic}
\label{sec:proofs-constrained}

\begin{proof}[Proof of Proposition~\ref{th:explicit-hat}]
For what concerns 
\eqref{eq:explicit-hat-l} we consider the case $ \mathtt{s}^\downarrow(\gG)> \mathtt{s}^\uparrow(\gG)$ (see Fig.~\ref{fig:l0}) and
 observe that  there is  an increase  of
at least $\gG$ for $(S_{n})_{n= \mathtt{s}^\downarrow (\gG), \ldots, \mathtt{v}^\downarrow (\gG)}$. Recalling that $\ls _n\in [-\gG, \gG]$ and that 
 $\ls$ is driven by $2 S$, we see that 
the process $(\ls_{n})_{n= \mathtt{s}^\downarrow (\gG), \mathtt{s}^\downarrow (\gG)+1,  \ldots}$ hits $\gG$ for a value of $n\in 
\{ \mathtt{s}^\downarrow (\gG)+1, \ldots, \mathtt{v}^\downarrow (\gG)\}$ and, using the definition of $ \mathtt{v}^\downarrow (\gG)$, we see also that 
$\ls_{ \mathtt{v}^\downarrow (\gG)}=\gG$ too:  the process hits $\gG$ not later than $\mathtt{v}^\downarrow (\gG)$ because
the driving process increases of at least $2\gG$, but $\ls$ is set to $\gG$ if the driving process \emph{tries} to make it larger than $\gG$. From time $n$ the process tries again to follow the driving process, except that it cannot become larger than $\gG$.

 \begin{figure}[h]
\begin{center}
\includegraphics[width=13 cm]{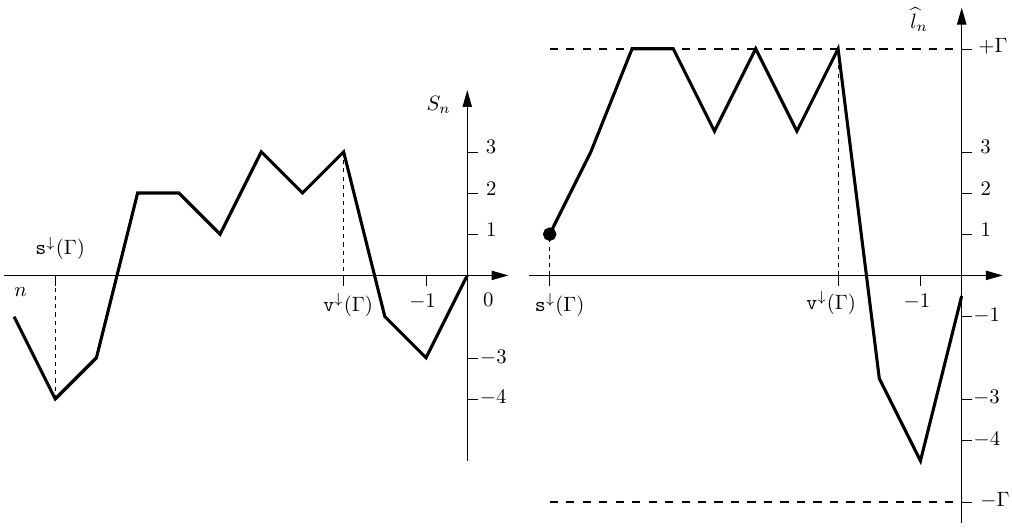}
\end{center}
\caption{\label{fig:l0} In this figure $\gG=5.5$.  
The $\ls$ process, on the right, is driven by twice the random walk trajectory, on the left. We look at the trajectory from $\mathtt{s}^{\downarrow}(\gG)$ and we use $\ls_{\mathtt{s}^{\downarrow}(\gG)}\in [-\gG, \gG]$. The fact that the $S$ trajectory increases of at least $\gG$ from $\mathtt{s}^{\downarrow}(\gG)$ to $\mathtt{v}^{\downarrow}(\gG)$ guarantees that the $\ls$ trajectory hits 
$\gG$ not later than $\mathtt{v}^{\downarrow}(\gG)$. Moreover, $\ls_{\mathtt{v}^{\downarrow}(\gG)}=\gG$ too. After 
$\mathtt{v}^{\downarrow}(\gG)$ and up to time 0 the $\ls$ trajectory copies the increments of $2S$.
} 
\end{figure}

Now we observe that if $\mathtt{v}^\downarrow (\gG)=0$ we have $\ls_0=\gG$ and therefore 
 \eqref{eq:explicit-hat-l} holds in this case. If instead $\mathtt{v}^\downarrow (\gG)<0$, see Fig.~\ref{fig:l0}, the increments of $(\ls_{n})_{n= \mathtt{v}^\downarrow (\gG),   \ldots, 0}$
 coincide with the increments of $(2S_{n})_{n= \mathtt{v}^\downarrow (\gG),   \ldots, 0}$, and also in this case  \eqref{eq:explicit-hat-l} holds.
 
 \begin{figure}[h]
\begin{center}
\includegraphics[width=13 cm]{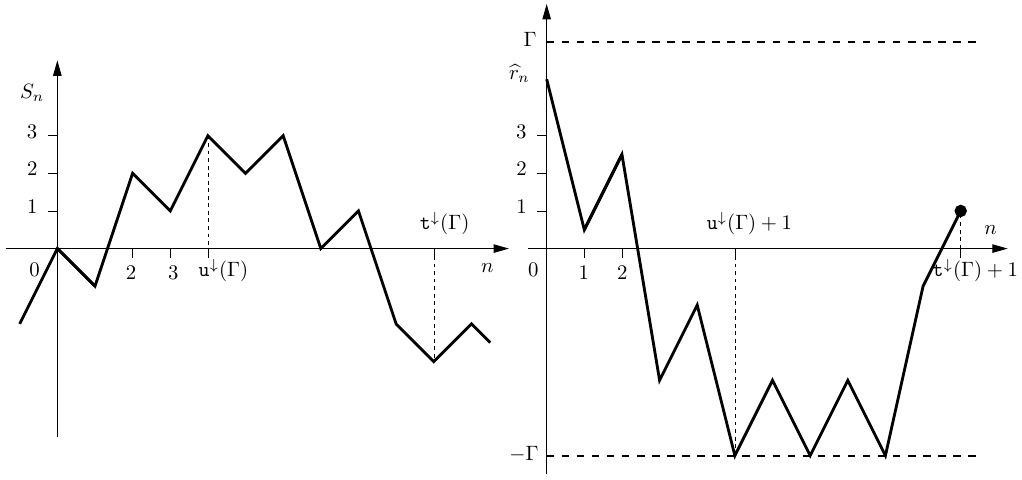}
\end{center}
\caption{\label{fig:r1} In this figure $\gG=5.5$. Equation \eqref{eq:explicit-hat-r} holds because the evolution of $(\rs_{-n})$, i.e. we are looking from right to left,
 repeats the increments of $(-2S_{-n-1})$ except that the $\rs$ process is confined to $[-\gG, \gG]$. Always arguing with the reversed time arrow, we see that, 
 regardless of the  value of $\rs_{\mathtt{t}^\downarrow (\gG) +1   }\in [-\gG, \gG]$,
the $\rs$ process is going to hit $\gG$ not later than $\mathtt{u}^\downarrow (\gG)$ and that, in any case, $\rs_{\mathtt{u}^\downarrow (\gG)+1}= -\gG$. After that moment and up to time 1 the 
evolution of $\rs$ reproduces the increments of $(-2S_{-n-1})$.}
\end{figure}

The argument for  $ \mathtt{s}^\downarrow(\gG)< \mathtt{s}^\uparrow(\gG)$ is entirely analogous. Alternatively,    we can map the case $ \mathtt{s}^\downarrow(\gG)< \mathtt{s}^\uparrow(\gG)$ to $ \mathtt{s}^\downarrow(\gG)> \mathtt{s}^\uparrow(\gG)$ by replacing $S$ with $-S$ and $\ls$ with $-\ls$. 

For what concerns \eqref{eq:explicit-hat-r} one can repeat the very same argument or use the time reversal symmetry of the problem: 
in any case one has to take into account the shift of one site pointed out in Remark \ref{rem:time-reversal}, see Fig.~\ref{fig:r1} and its caption.
\end{proof}

\begin{proof}[Proof of Proposition  \ref{th:lr-Fisher}]
{We use the notion of $\gG$-extrema over $\bbZ$ introduced in Section~\ref{sec:Fisher-Z}, but we present an interesting characterisation of these which may be of help for the reader: a site $\tu\in\bbZ$ is a $\gG$-maximum of $(S_n)_{n\in\bbZ}$ over $\bbZ$ if and only if there exist $\ts, \ttt\in\bbZ$ such that $\ts\le \tu\le \ttt$, $S$ is maximal at site $\tu$ on the interval $[\![\ts, \ttt]\!]$,  $S_{\ts}\le S_{\tu}-\gG$ and $S_{\ttt}\le S_{\tu}-\gG$. $\gG$-minima over $\bbZ$ are characterised similarly. With this characterisation in mind, we introduce a slightly richer notation for the \emph{multiple} $\gG$-extrema.} In fact, here we denote by $\tU_j, j\in\bbZ$ all maximal sets of adjacent $\gG$-extrema (adjacent means ``not separated by opposite $\gG$-extrema"), and observe that $S$ is constant over each $\tU_j, j\in\bbZ$. The sets $\tU_j , j\in\bbZ$ are ordered in the natural way, inherited from the usual order on $\bbZ$. We set the indexation by requiring that $\tU_1$ be the left-most class which is included in $\N$. Then, the indexation corresponds to that of Section~\ref{sec:Fisher-Z}, and we notice that $u_j=\min \tU_j$ and $u_j^+=\max \tU_j$.

Recall that $s^{(F)}$ is defined by: 
\begin{itemize}
\item $s^{(F)}_n=1$ if there exists $j\in\bbZ$ such that $\tu_j^+< n\le \tu_{j+1}$ and $(\tu_j^+, \tu_{j+1})$ is a $\gG$-ascending stretch of height strictly larger than $\gG$;
\item $s^{(F)}_n=-1$ if there exists $j\in\bbZ$ such that $\tu_j^+< n\le \tu_{j+1}$ and $(\tu_j^+, \tu_{j+1})$ is a $\gG$-descending stretch of height strictly larger than $\gG$;
\item $s^{(F)}_n=0$ otherwise.
\end{itemize}

We compute $\ls$ and $\rs$ on a $\gG$-ascending stretch, see Fig.~\ref{fig:ascend}.
Take $j\in\bbZ$ and assume that $(\tu_j^+, \tu_{j+1})$ is a $\gG$-ascending stretch, i.e. 
$ S_{\tu_{j+1}}-S_{\tu_j^+}\ge \gG$ and $S_n-S_m> -\gG$ whenever $\tu_j^+\le m<n\le \tu_{j+1}$. We are going to express $\ls$ on $\{\tu_j, \dots, \tu_{j+1}\}$, then $\rs$ on $\{\tu_j+1, \dots, \tu_{j+1}+1\}$, and finally $\ms$ on $\{\tu_j+1, \dots, \tu_{j+1}\}$.

\begin{figure}[h]
\begin{center}
\includegraphics[width=12 cm]{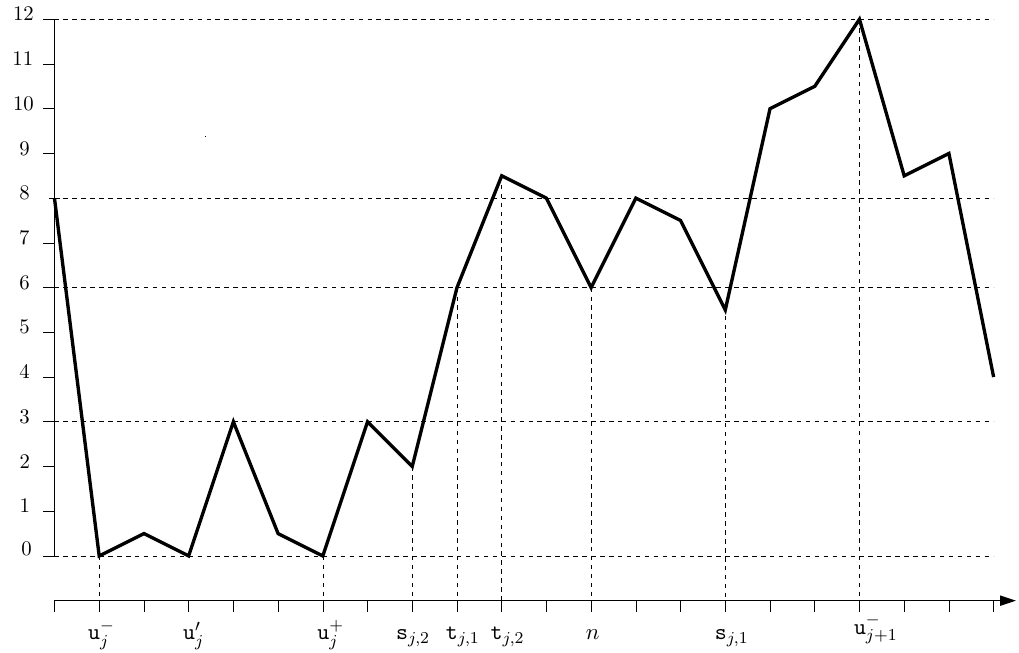}
\end{center}
\caption{\label{fig:ascend} In this figure we present a $\gG$-ascending stretch of $S$, for $\gG\in [3, 8)$: the set of $\Gamma$-minima
$\mathtt{U}_j=\{\mathtt{u}_j, \mathtt{u}'_j, \mathtt{u}_j^+\}$ and the set of $\gG$-maxima $\mathtt{U}_j=\{\mathtt{u}_{j+1}\}$
are independent of the choice of $\gG$ in $ [3, 8)$. We use $\mathtt{t}_{j,1}$ for $\mathtt{t}_{j}(\gG)$ and $\mathtt{s}_{j,1}$ for $\mathtt{s}_{j}(\gG)$  when $\gG \in [3, 6)$.
When $\gG \in [6.5, 8)$ instead we use $\mathtt{t}_{j,2}$ for $\mathtt{t}_{j}(\gG)$ and $\mathtt{s}_{j,2}$ for $\mathtt{s}_{j}(\gG)$.
Note that $\mathtt{t}_{j,1}< \mathtt{s}_{j,1}$, while  $\mathtt{t}_{j,2}>\mathtt{s}_{j,2}$. So if we consider the site $n$  for $\gG\in [3, 6)$ we have to apply 
the middle option in the right-hand side of \eqref{eq:<case}: we obtain that $\widehat{m}_n=2\gG -6>0$. If instead we consider the site $n$  for $\gG\in [6.5, 8)$, we 
have to apply 
the lowest option in the right-hand side of \eqref{eq:<case}: we obtain that $\widehat{m}_n=24>0$. Note that when $\gG \in [6, 6.5)$, then $\mathtt{t}_j(\gG)$ is $\mathtt{t}_{j,2}$,
but $\mathtt{s}_j(\gG)$ still coincides with $\mathtt{s}_{j,1}$. Note also that if $\gG< 3$ then $\tt{U}_j$ is no longer a set of $\Gamma$-minima. And if $\gG\ge 8$ the portion of random walk in the image is insufficient to determine the $\gG$-extrema.  
}
\end{figure}

We set :
\begin{equation}
\begin{split}
\ttt_j & = \inf\{\tu_j \le n\le \tu_{j+1} : S_n\ge S_{\tu_j}+\gG\}\\
\ts_j & = \sup\{\tu_j \le n\le \tu_{j+1} : S_n\le S_{\tu_{j+1}}-\gG\}\,.
\end{split}
\end{equation}
Note that $\ttt_j$ and $\ts_j$ both belong to $\{\tu_j^+, \dots, \tu_{j+1}\}$. We point out that they are a priori not ordered:  we may have either $\ttt_j\le\ts_j$ or $\ts_j < \ttt_j$.

By construction we have that $\ls_{\tu_j}=-\gG$ and, therefore, that
\begin{equation}
\ls_n=\begin{cases} -\gG +2\left(S_n-S_{\tu_j}\right)& \text{ for } \tu_j\le n< \ttt_j\\
\gG-2 \max_{\tu_j \le i \le n} \left(S_i-S_n\right) & \text{ for } \ttt_j\le n \le \tu_{j+1}
\end{cases}
\end{equation}
By exploiting the time reversal properties (see Remark \ref{rem:time-reversal}), we get the similar result for $\rs$:
\begin{equation}
\rs_n=
\begin{cases} 
-\gG + 2\left(S_{\tu_{j+1}}-S_{n-1}\right)& \text{ for } \ts_j< n\le \tu_{j+1}+1\\
\gG-2 \max_{n-1 \le i \le \tu_{j+1} } \left(S_{n-1} - S_i\right) & \text{ for } \tu_j\le n \le \ts_j
\end{cases}
\end{equation}

Using these expressions, we can now compute $\ms_n$. 
We distinguish between two cases. 
If $\ttt_j\le\ts_j$, then
\begin{equation}
\label{eq:<case}
\ms_n= \ls_{n-1}+2h_n+\rs_{n+1}=
\begin{cases} 
2\min_{n\le i\le \tu_{j+1}} \left(S_i - S_{\tu_j}\right)  
& \text{ for } \tu_j < n< \ttt_j \, , \\
2\gG - 2\max_{\tu_j\le i\le n \le j \le \tu_{j+1}} \left(S_i-S_j\right) 
& \text{ for } \ttt_j \le n \le \ts_j \, , \\
2 \max_{\tu_j\le i\le n} \left(S_{\tu_{j+1}}-S_i \right) 
& \text{ for } \ts_j < n \le \tu_{j+1} \, .
\end{cases}
\end{equation}
If $\ts_j<\ttt_j$ instead
\begin{equation}
\label{eq:>case}
\ms_n= \ls_{n-1}+2h_n+\rs_{n+1}=
\begin{cases} 
2\min_{n\le i\le \tu_{j+1}} \left(S_i - S_{\tu_j}\right)  
& \text{ for } \tu_j < n \le \ts_j \, , \\
-2\gG + 2\left(S_{\tu_{j+1}}-S_{\tu_j}\right)  
& \text{ for } \ts_j < n < \ttt_j \, , \\
2 \max_{\tu_j\le i\le n} \left( S_{\tu_{j+1}}-S_i \right) 
& \text{ for } \ttt_j < n \le \tu_{j+1} \, .
\end{cases}
\end{equation}

Using these explicit expressions for $\ms_n$, we can complete the proof. In fact one readily checks that
\begin{itemize}
\item if $\tu_j< n \le \tu_j^+$, then $\ms_n=0$;
\item if $S_{\tu_{j+1}}-S_{\tu_j}>\gG$, then we have $\ms_n>0$ when $\tu_j^+<n\le \tu_{j+1}$, using the fact that $\max_{\tu_j\le i \le j \le \tu_{j+1}} \left(S_i-S_j\right) < \gG$;
\item if $S_{\tu_{j+1}}-S_{\tu_j}=\gG$, then we have $\ts_j=\tu_j^+$ and $\ttt_j=\tu_{j +1}$ and we derive $\ms_n=0$ for all $n$ such that $\tu_j^+<n\le \tu_{j+1}$.
\end{itemize}

To cover $\gG$-descending stretches, it suffices to replace $S$ by $-S$. 
\end{proof}

\section{Proofs for the reflected random walk: zero-temperature approach.}\label{sec:appendzerotemp}

In this appendix we investigate the link between the simplified processes $\ls$ and $\rs$ and the configurations that maximise the Hamiltonian $H$. 
The zero-temperature approach that we will use gives alternative proofs of Propositions \ref{th:explicit-hat} and \ref{th:lr-Fisher}.  For the sake of conciseness, we leave some  details of the proofs to the reader.

\subsection{Finite-volume system}

We start by working on the Random Field Ising Chain on a finite segment $[\![\ell, \err]\!]$, see \eqref{eq:defZellerr} and \eqref{eq:defHellerr}.

Let us write $\left(\ls_{\ell, n}^{\,\,(a)}\right)_{n\ge \ell-1}$ for the process defined in \eqref{eq:MCDShat} with initial value $\ls_{\ell,\ell-1}^{\,\,(a)}= a \gG$, where $a\in\{-1, +1\}$. We claim that
\begin{equation} \label{eq:explicitls}
\ls_{\ell, n}^{\,\,(a)} = \max_{\sigma\in\{-1, +1\}^{[\![\ell, n]\!]}} H_{\ell, n, J, h}^{a+}(\sigma) -  \max_{\sigma\in\{-1, +1\}^{[\![\ell , n]\!]}} H_{\ell, n, J, h}^{a-}(\sigma).
\end{equation}

We sketch two ways to show identity \eqref{eq:explicitls}. 
The first is simply to check by hand that the right-hand side follows the recursion defining  $\left(\ls_{\ell, n}^{\,\,(a)}\right)_{n\ge \ell-1}$. 
The second method is more interesting and exploits the zero-temperature approach. 
We consider the partition function defined in \eqref{eq:Gibbs} but we choose to multiply the Hamiltonian $H$ by a parameter $\beta>0$, or equivalently we decide to multiply $J$ and $h$ by $\beta$. Denoting 
 $l_{\ell, n}^{(a), \beta}= \frac{1}{\beta} \log\left(\frac{Z_{\ell, n, \beta J, \beta h}^{a+}}{Z_{\ell, n,  \beta J, \beta h}^{a-}}\right)$, our analysis in Section \ref{sec:tm-mainproof} yields that
$(l_{\ell, n}^{(a), \beta})_{n\ge \ell-1}$ evolves in $(-\gG, \gG)$, with function $f^\gb_h=b_\gG^\gb \circ \theta_h$ (recall \eqref{eq:f_h}), where $b_\gG^\gb$ is defined by:
\begin{equation}
b_\gG^\gb(x)=x+\frac1\gb \log\left(\frac{1+e^{-\gb(\gG+x)}}{1+e^{-\gb(\gG-x)}}\right), \qquad x\in \bbR\, .
\end{equation}
We observe that the functions $b_\gG^\gb$ are 1-Lipschitz and converge pointwise, as $\beta$ goes to infinity, towards $\widehat b_\gG$. Consequently, for fixed $n$ (and fixed disorder sequence $h$),  $l_{\ell, n}^{(a), \beta} \to_{\beta\to\infty} \ls_n^{\,\,(a)}$.
On the other hand, it is clear that $\ls_{\ell, n}^{\,\,(a), \beta}= \frac{1}{\beta} \log\left(\frac{Z_{\ell, n, \beta J, \beta h}^{a+}}{Z_{\ell, n, \beta J, \beta h}^{a-}}\right)$ converges towards the right-hand side in \eqref{eq:explicitls}.

Similarly, for the backward process $(\rs^{\,\,(b)}_{n , \err})_{n\le \err+1}$ defined by $\rs^{\,\,(b)}_{\err+1 , \err}= b \gG$, where $b\in\{-1,+1\}$ and $\rs_{n,\err}^{\,\,(b)}= \widehat{b}_\gG(\rs_{n+1, \err}^{\,\,(b)} + 2 h_n)$ for $n\le \err$, we have:
\begin{equation} \label{eq:explicitrs}
\rs^{\,\,(b)}_{n , \err} = \max_{\sigma\in\{-1, +1\}^{[\![ n, \err]\!]}} H_{ n,\err, J, h}^{+b}(\sigma) -  \max_{\sigma\in\{-1, +1\}^{[\![ n, \err]\!]}} H_{ n, \err, J, h}^{-b}(\sigma).
\end{equation}
\medskip
Let us now investigate on the sign of $\ls_{\ell, n-1}^{\,\,(a)}+2 h_n + \rs_{n+1, \err}^{\,\,(b)}$. 
Putting \eqref{eq:explicitls} and \eqref{eq:explicitrs} together, we get that, for $\ell \le n \le \err$:
\[\ls_{\ell, n-1}^{\,\,(a)}+2 h_n + \rs_{n+1, \err}^{\,\,(b)}=  \max_{\sigma\in\{-1, +1\}^{[\![\ell ,\err]\!]}: \sigma_n=+1} H_{\ell, \err, J, h}^{ab}(\sigma) -  \max_{\sigma\in\{-1, +1\}^{[\![\ell , \err]\!]}: \sigma_n=-1} H_{\ell,  \err, J, h}^{ab}(\sigma).\]
From the above formula, we retain that the sign of  $\ls_{\ell, n-1}^{\,\,(a)}+2 h_n + \rs_{n+1, \err}^{\,\,(b)}$ reflects which one of the two maxima appearing there is larger than the other. Hence the sign of $\ls_{\ell, n-1}^{\,\,(a)}+2 h_n + \rs_{n+1, \err}^{\,\,(b)}$ reflects the spin assigned to site $n$ by the maximsers of $H_{\ell, \err, J, h}^{ab}$. Precisely:
\begin{lemma}\label{lem:maximiserssign}
The quantity $ \ls_{\ell, n-1}^{\,\,(a)}+2 h_n + \rs_{n+1, \err}^{\,\, (b)}$ is
\begin{itemize}
\item positive if all the maximisers of $H_{\ell, \err, J, h}^{ab}$ satisfy $\sigma_n=+1$,
\item negative if all the maximisers of $H_{\ell, \err, J, h}^{ab}$ satisfy $\sigma_n=-1$,
\item zero otherwise (i.e., there is at least one maximiser satisfying $\sigma_n=+1$ and at least one maximiser satisfying $\sigma_n=-1$).
\end{itemize}
\end{lemma}

\begin{rem} 
The argument we just developed to identify the sign of $ \ls_{\ell, n-1}^{\,\,(a)}+2 h_n + \rs_{n+1, \err}^{\,\, (b)}$ may be understood in a more transparent way  
if  the law of $h$ has no atom. In this case for every fixed $\gG$, almost surely, $\ls_{\ell, n-1}^{\,\,(a)}+2 h_n + \rs_{n+1, \err}^{\,\,(b)}$ is nonzero (the three terms being independent). Furthermore, for every fixed $\gG>0$, almost surely, the configurations on $[\![\ell, \err]\!]$ take all different values by $H_{\ell, \err, J, h}^{ab}$, which has therefore exactly one maximiser.
Hence the above result can be reformulated as: $\sign( \ls_{\ell, n-1}^{\,\,(a)}+2 h_n + \rs_{n+1, \err}^{\,\,(b)})=\sigma_n$, where  $\sigma$  is the unique maximiser of $H_{\ell, \err, J, h}^{ab}$.

In order to illustrate once again the zero-temperature approach, we give another proof of this fact, under the assumptions that $\ls_{\ell, n-1}^{\,\,(a)}+2 h_n + \rs_{n+1, \err}^{\,\,(b)}$ is nonzero and $H_{\ell, \err, J, h}^{ab}$ has exactly one maximiser.

Let us consider the Gibbs measure $\bP_{\ell, \err, \beta J, \beta h}^{ab}$ which corresponds to the Gibbs measure defined in \eqref{eq:Gibbs} but with the Hamiltonian $H$ multiplied by $\beta$. 
Then, using the process $(l_{\ell, n}^{(a), \beta})_{n\ge \ell-1}$ introduced above, as well as the corresponding process $(r_{n ,\err}^{\,\,(b), \beta})_{n\le \err+1}$, \eqref{eq:Psigmaj} gives for $\ell\le n\le \err$:
\begin{equation}
\label{eq:Psigmajbeta}
\bP_{\ell, \err, \beta J, \beta h}^{ab}\left( \gs_n=+1\right)\, =\, 
\frac 1 {1+ \exp\left(-\gb\left(l_{\ell,n-1}^{(a),\gb} +2h_n+ r_{n+1, \err}^{\,\,(b),\gb}\right)\right)}\, .
\end{equation}
On the one hand, the measure  $\bP_{\ell, \err, \beta J, \beta h}^{ab}$ weakly converges towards the Dirac measure associated to the maximiser of $H_{\ell, \err, J, h}^{ab}$. 
Hence, $\bP_{\ell, \err, \beta J, \beta h}^{ab}\left( \gs_n=+1\right)$ converges as $\beta\to\infty$
\begin{itemize}
\item towards 1 if the maximiser satisfies $\sigma_n=+1$,
\item  towards 0 if the maximiser satisfies $\sigma_n=-1$.
\end{itemize}
On the other hand, using as before the pointwise convergence of $b_{\gG}^{\beta}$ towards $\widehat{b}_\gG$, 
it appears that $l_{\ell,n-1}^{(a),\gb} +2h_n+ r_{n+1, \err}^{\,\,(b),\gb}$ converges towards $\ls_{\ell, n-1}^{\,\,(a)}+2 h_n + \rs_{n+1, \err}^{\,\,(b)}$ as $\beta\to \infty$. 
If this limit is positive (respectively negative), then we get that $\bP_{\ell, \err, \beta J, \beta h}^{ab}\left( \gs_n=+1\right)$ converges, as $\beta\to\infty$, towards 1 (respectively 0). 
\end{rem}

In view of Lemma \ref{lem:maximiserssign}, we are naturally interested in describing the configurations which maximise $H_{\ell, \err, J, h}^{ab}$. To do so, let us introduce the following notion. 
A site $\tu\in[\![ \ell, \err ]\!]$ is said to be a $\gG$-maximum of $S=(S_n)_{n\in\bbZ}$ over $[\![ \ell, \err ]\!]$ with boundary conditions $(a,b)$ 
if and only if there exist $\ts, \ttt\in[\![\ell-1, \err]\!]$ such that
\begin{itemize}
\item $\ts\le \tu\le \ttt$, 
\item $S$ is maximal at site $\tu$ on the interval $[\![\ts, \ttt]\!]$,
\item $S_{\ts}\le S_{\tu}-\gG$ or ($\ts=\ell-1, a=+1$),
\item and $S_{\ttt}\le S_{\tu}-\gG$ or ($\ttt=\err$ and $b=+1$). 
\end{itemize}
The notion of $\gG$-minimum is defined similarly: $\tu\in[\![ \ell, \err ]\!]$  is a $\gG$-minimum for $S$ with boundary conditions $(a, b)$ if and only if it is a $\gG$-maximum for  $-S$ with boundary conditions $(-a, -b)$.
Observe that the left-most (respectively right-most) $\gG$-extremum is a $\gG$-maximum if $a=+1$ (respectively if $=+1$) and a $\gG$-minimum otherwise.

We say that two $\gG$-extrema are adjacent if they are of the same type and not separated by opposite $\gG$-extrema, and we denote by $\tU_1(\gG), \tU_2(\gG), \dots, \tU_k(\gG)$ the maximal sets of adjacent $\gG$-extrema, ordered using the usual order on $\Z$ (note that the parity of $k$ is determined by the boundary conditions $(a, b)$). We further denote $\tu_j(\gG)=\min \tU_j(\gG)$ and $\tu_j^+(\gG)=\max \tU_j(\gG)$. In these notations, we dropped the dependence in $\ell, \err, a$ and $b$, but it should not be forgotten. We now claim the following.

\begin{lemma}\label{lem:descrmaxim}
The configurations that maximise $H_{\ell, \err, J, h}^{ab}$ are exactly the following:
\begin{enumerate}[leftmargin=0.7 cm]
\item the configurations that are obtained by picking one element in each $\tU_j(\gG)$ and then setting +1 in ascending stretches (from a minimum to a maximum), -1 in descending stretches (from a maximum to a minimum),
\item the configurations obtained from those in (1) by removing stretches which have height exactly $\gG$: removing a stretch means switching the spins to opposite sign, the consequence is to merge three neighbouring spin domains in one. 
\end{enumerate}
\end{lemma}
\begin{proof}
We note that all those configurations have the same image by $H_{\ell, \err, J, h}^{ab}$. To show that they are maximisers, it is enough to show that any configuration which is not as described can be modified in order to increase its image by $H_{\ell, \err, J, h}^{ab}$. The details are left to the reader.
\end{proof}

Using Lemmas \ref{lem:maximiserssign} and \ref{lem:descrmaxim}, we conclude that $\ls_{\ell, n-1}^{\,\,(a)}+2 h_n + \rs_{n+1, \err}^{\,\,(b)}$ is
\begin{itemize}[leftmargin=0.5 cm]
\item positive on ascending stretches having height strictly larger than $\gG$: precisely on $[\![ \ell, \tu_1(\gG)]\!]$ if $a=+1$, on $[\![\tu_j^+(\gG)+1, \tu_{j+1}(\gG)]\!]$ for $j$'s such that $\tu_{j+1}(\gG)$ is a $\gG$-maximum and $S_{\tu_{j+1}(\gG)}-S_{\tu_j^+(\gG)}>\gG$, and, if $b=+1$, denoting by $\tu_k^+(\gG)$ the right-most $\gG$-extremum, on $[\![ \tu_k^+(\gG), \err]\!]$, 
\item negative on descending stretches having height strictly larger than $\gG$ (correspondingly) 
\item zero between adjacent $\gG$-extrema, i.e. on $[\![\tu_j(\gG)+1, \tu_j^+(\gG)]\!]$ for all $j$'s, and also on stretches of heigth exactly $\gG$, i.e. on $[\![\tu_j^+(\gG)+1, \tu_{j+1}(\gG)]\!]$ for all $j$'s such that $| S_{\tu_{j+1}(\gG)}-S_{\tu_j^+(\gG)}|=\gG $. 
\end{itemize}

\subsection{Infinite-volume system}

Our aim is now to understand what happens in the limit $\ell\to -\infty$ and $\err\to \infty$.

We first stress the fact that the notion of $\gG$-extremum over a segment is local, in the following sense. If $\tu$ is a $\gG$-maximum over $[\![\ell, \err]\!]$ which does not belong to $\tU_1(\gG)$, the left-most set of adjacent $\gG$-extrema, neither to $\tU_k(\gG)$, the right-most set of adjacent $\gG$-extrema, then it is also a $\gG$-maximum over any segment $[\![\ell', \err']\!]$ containing $[\![\ell, \err]\!]$. 

Using this, we see that, for fixed $n\in \Z$, $\ls_{\ell, n}^{\,\,(a)}$ is equal to $\ls_n$ as soon as $\ell$ is small enough. 
Indeed, let us use \eqref{eq:explicitls}. The configurations that achieve the maxima that appear there differ near the right boundary, due to the differing boundary conditions + and -, but they coincide elsewhere (up to modifications that do not change the value of the Hamiltonian), hence the difference between the maxima is eventually constant.
 
\begin{rem}
This fact could also be proved using the analytic approach of Appendix \ref{sec:appendanalytic}. Indeed, the case $n=0$ can be obtained by adapting the proof of Proposition \ref{th:explicit-hat}: the proof holds regardless of the value of $\ls_{\tsdown(\gG)}$ (respectively $\ls_{\tsup(\gG)}$) so it yields the same expression for all the $\ls_{\ell, 0}^{\,\,(a)}=\ls_0$ with $\ell \le \tsdown(\gG)$ (respectively  $\ell \le \tsup(\gG)$). Extension to any $n$ is immediate. 

We further stress that our above arguments also give an alternative proof of Proposition \ref{th:explicit-hat}: since we have described explicitly the maximising configurations, we can actually derive the explicit formula for $\ls_0$ and this gives exactly \eqref{eq:explicit-hat-l}.
\end{rem}
The same holds of course for the backward process: for fixed $n\in\Z$, when $\err $ is large enough, $\rs^{\,\,(b)}_{n , \err}=\rs_n$.
Thus, for fixed $n$, when $\ell$ is small enough and $\err$ is large enough, for every boundary conditions $(a, b)$,
\[\ls_{\ell, n-1}^{\,\,(a)}+2 h_n + \rs_{n+1, \err}^{\,\,(b)}=  \ls_{ n-1}+2 h_n + \rs_{n+1}.\]

The final piece of the puzzle is the following consequence of the fact that the notion of $\gG$-extremum is local. Provided that $S$ be recurrent on both sides (which is almost surely the case), we have that for any fixed segment $[\![\ell, \err]\!]$, the $\gG$-extrema over $[\![\ell', \err']\!]$ are, in restriction to $[\![\ell, \err]\!]$, constant for small enough $\ell'$ and large enough $\err'$. We claim that the $\gG$-extrema obtained in this way are exactly the $\gG$-extrema over $\Z$ defined in Section \ref{sec:auxiliary}.
It allows us to conclude that the sign of $\ls_{ n-1}+2 h_n + \rs_{n+1}$ is exactly $s^{(F)}_n$ defined in \eqref{eq:sF}. We have thus given a second proof of Proposition \ref{th:lr-Fisher}.

\begin{rem}\label{rem-descrground}
For completeness, let us describe (without proof) the ground state configurations of the infinite volume RFIC, which we define now. Given two configurations $\sigma, \sigma'\in \{-1, +1 \}^\Z$ which differ only at finitely many sites, we define 
\[\Delta H_{J, h} (\sigma, \sigma'):= J\sum_{n\in \Z} \left( \sigma_{n-1} \sigma_n - \sigma'_{n-1} \sigma'_n \right) + \sum_{n\in \Z} h_n (\sigma_n -\sigma_n')\]
(this is well-defined since there are only finitely many nonzero terms in the sums).
A configuration $\sigma\in \{-1, +1\}^\Z$ is called a ground state configuration if $\Delta H_{J, h} (\sigma, \sigma')\ge 0$ for every configuration $\sigma'$ which differs from $\sigma$ at only finitely many sites.
We state (without proof) that the ground state configurations are exactly the following:
\begin{enumerate}[leftmargin=0.7 cm]
\item the configurations that are obtained by picking one element in each $U_j(\gG)$ and then setting +1 in ascending stretches (from a minimum to a maximum), -1 in descending stretches (from a maximum to a minimum),
\item the configurations obtained from the ones in (1) by removing stretches which have height exactly $\gG$: removing a stretch means switching the spins to opposite sign, the consequence is to merge three neighbouring spin domains in one. 
\end{enumerate}
Proposition \ref{th:lr-Fisher} can be interpreted as follows: the sign of $\ls_{ n-1}+2 h_n + \rs_{n+1}$ reflects the spin given to site $n$ by the ground state configurations.
\end{rem}

\section{Scaling to the Neveu-Pitman process}
\label{sec:scaling}

This section relies on the Donsker Invariance Principle for one dimensional centered random walks with finite variance increments.
This  is often established in $C^0([0, 1]; \bbR)$, but for us it is more practical to work in
$C^0([0, \infty); \bbR)$ \cite[Sec.~1.3]{cf:StroockVaradhan}: Donsker Invariance Principle in $C^0([0, \infty); \bbR)$ is given as an exercise in 
\cite[Ex.~2.4.2]{cf:StroockVaradhan} or it follows by applying the general result \cite[Th.~11.2.3]{cf:StroockVaradhan}. Here is the statement we are going to use: set for $t\ge 0$ such that $ t\gG^2/\vartheta^2$ is integer 
\begin{equation}
\label{eq:BgGt}
B^{(\gG)}_t:= \, \frac 1 \gG S_{t \,\gG^2 / \vartheta^2}\,,
\end{equation}
and extend this definition to every $t\ge 0$ by affine interpolation, so $B^{(\gG)}$ is a stochastic process with  trajectories in $C^0([0, \infty); \bbR)$, like the standard Brownian motion $B$.
Donsker Invariance Principle states, with standard notation \cite{cf:Bill}, that $B^{(\gG)} \Rightarrow B$ as $\gG \to \infty$. 

\medskip

\begin{proof}[Proof of Proposition~\ref{th:scaling}]
We start by observing that Lemma~\ref{th:ren-u} directly implies that the sequence in \eqref{eq:ren-struc}
is a sequence of independent random vectors and that the odd and even subsequences are IID. We are therefore left with proving the convergence part of the statement, i.e. \eqref{eq:scaling}. For this we introduce the  continuum analog of the notation of Section~\ref{sec:Fisher}. 
We therefore consider the set $C^0_\infty$ of  functions  
$b_\cdot \in C^0([0, \infty); \bbR)$ such that  $\limsup_{t\to \infty} b_t=\infty$ and  $\liminf_{t\to \infty} b_t=-\infty$: note that $C^0_\infty$   is open 
in $C^0([0, \infty); \bbR)$.
We introduce the first time of first (unit) decrease for $b\in C^0_\infty$ as
\begin{equation}
t_1(b) \, :=\, \min\left\{t\ge 0:\, \text{there exists } s \in [0,t] \text{ such that } b_s-b_t=1\right\},
\end{equation}
and in turn the time of the first absolute maximum in $[0, t_1(b) ]$:
\begin{equation}
u_1(b)\, :=\, \min\left \{t \in  [0, t_1(b) ]:\, b_t= \max_{s \in [0, t_1(b) ]} b_s \right\}.
\end{equation}
Then we keep going by looking for the first time of unit increase after $t_1(b)$: 
\begin{equation}
t_2(b) \, :=\, \min\left\{t\ge t_1(b):\, \text{there exists } s \in [t_1(b),t] \text{ such that } b_t-b_s=1\right\},
\end{equation}
and the first absolute minimum in $[t_1(b), t_2(b)]$
\begin{equation}
u_2(b)\, :=\, \min\left \{t \in  [t_1(b), t_2(b)]:\, b_t= \min_{s \in [t_1(b), t_2(b)]} b_s \right\}.
\end{equation}
And then one iterates this construction generating the  infinite sequence of $(t_k(b))_{k \in \bbN}$ 
of alternating unit decrease/increase
and, above all,  the  infinite sequence  $(u_k(b))_{k \in \bbN}$ of $1$-extrema for the function $b$.

A first key observation is that, with 
$B^{(\gG)}$ the continuous trajectory built from the random walk trajectory $S$ in \eqref{eq:BgGt}, we have
\begin{equation}
u_k \left(B^{(\gG)}\right) \, =\, \frac{\vartheta^2}{\gG^2} \mathtt{u}_k(\gG)
\ \ \text{ and } \ \ 
B^{(\gG)}_{u_k \left( B^{(\gG)} \right)}\, =\, 
\frac 1 \gG S_{\mathtt{u}_k(\gG)}\, .
\end{equation}

Now we remark that, for every $k$ the function $u_k: C^0_\infty \to [0, \infty)$ is continuous in a neighborhood of $b$
if for every $j=1,\ldots, k-1$ we have $\vert b_{u_{j+1}(b)}- b_{u_j (b)}\vert >1$ and at $u_k(b)$ there is a unique 
absolute maximum (if $k$ is odd) or minimum    (if $k$ is even) for $b$ on the interval $[u_{k-1}(b), u_{k+1}(b)]$. 
This is an event that happens with probability one if $b$ is a trajectory of a standard Brownian motion. From this 
we directly infer an analog explicit condition on the trajectory $b$ in order to have the continuity of 
\begin{equation}
b \mapsto 
 \left( u_{j+1}(b)-u_j(b), \left \vert b_{u_{j+1}(b)}-b_{u_{j}(b)}\right \vert  \right)_{j=1,2, \ldots, n},
\end{equation}
in the neighborhood of a given function $b$
and check that these properties are almost surely verified if $b$ is a Brownian trajectory. 
Therefore by  Donsker Invariance Principle and by the Mapping Theorem \cite[Th.~2.7, Ch.~1]{cf:Bill}
we obtain the convergence in law as $\gG\to \infty$   of the $2n$ dimensional random vector
\begin{equation}
\label{eq:1hdu64}
 \left( \frac{\vartheta^2}{ \gG ^2}\left(\tu_{j+1}(\gG) - \tu_j(\gG)\right), \frac{\left \vert S_{\tu_{j+1}(\gG)} - S_{\tu_j(\gG)}\right\vert }{ \gG }\right)_{j=1, \ldots, n}
 \,,
\end{equation}
to 
\begin{equation}
\label{eq:1hdu64-1}
 \left(u_{j+1}(B)-u_j(B), \vert B_{u_{j+1}(B)}-B_{u_{j}(B)}\vert\right)_{j=1, \ldots, n}\,.
\end{equation}
This of course implies  
\eqref{eq:scaling}, so the proof of Proposition~\ref{th:scaling} is complete.  It is worth remarking  that,  given the independence and IID properties discussed at the beginning of the proof, 
 it suffices  to establish the convergence for the first two marginals of the sequence in  \eqref{eq:1hdu64}: in particular it suffices  to establish the convergence 
 of  \eqref{eq:1hdu64} to
\eqref{eq:1hdu64-1} for $n=2$. 
\end{proof}

\section{About the Fisher RG}
\label{sec:RG}

For conciseness in this section we assume that the law of $h_1$ has no atoms. In particular this implies that 
$\bbP(\sum_{k\in \bbN}\ind_{\{h_k=0\}}=0 )=1$ and 
$\bbP( \tu_j(\gG) =\tuplus_j(\gG)$ for every $j)=1$.
Set $\tau_0=0$ and, for $j \in  \bbN$, $\tau_{j}:=\inf\{k> \tau_{j-1}:\, \sign(h_k) \neq \sign (h_{k+1})\}$, where $\sign(0):=+1$. 
Then we build the sequence of $\bbN \times \bbR$ random variables 
\begin{equation}
\label{eq:alt-signs}
\left( \left( \tau_{j}-\tau_{j-1}, S_{\tau_{j}}-S_{\tau_{j-1}}\right)\right)_{j\in \bbN}\,.
\end{equation}
In fact $(S_{\tau_{j}}-S_{\tau_{j-1}})_j$ is an alternating sign sequence.

\begin{rem}
 Note that 
if we consider instead 
\begin{equation}
\label{eq:non-alt-signs}
\left( \left( \tau_{j}-\tau_{j-1}, \left \vert S_{\tau_{j}}-S_{\tau_{j-1}}\right\vert\right)\right)_{j\in \bbN}\,.
\end{equation}
then, 
conditionally on $h_1\ge 0$,  \eqref{eq:non-alt-signs} is a sequence of independent random variables 
and the subsequence with $j \in 2 \bbN -1$ (respectively $j \in 2 \bbN$)  is an IID  sequence with marginal law that coincides with the distribution 
of $(T,S^+_T)$, with $T$ a geometric random variable with parameter $p=\bbP(h_1\ge 0)$ and $S_j^+$ the sum of 
$j$ IID random variables distributed like $h_1$ conditioned to $h_1\ge 0$ (respectively,  
 is an IID  sequence with marginal law that coincides with the distribution 
of $(T,S^-_T)$, with $T$ a geometric random variable with parameter $1-p$ and $S_j^-$ the sum of 
$j$ IID random variables distributed like $h_1$ conditioned to $h_1<0$). The analogous statement holds conditionally on $h_1< 0$ and we remark also that \eqref{eq:non-alt-signs} is an IID sequence if the law of $h_1$ is symmetric.
\end{rem}

For every $N\ge 3$ and every realization of $(h_k)$ let us call $\underline{\gD}:= \min_{j\in\{2, \ldots, N-1\}} \vert S_{\tau_{j}}-S_{\tau_{j-1}}\vert$
and then call $\textsc{j}$ the smallest value in $j\in \{2, \ldots, N-1\}$ such that $\vert S_{\tau_{j}}-S_{\tau_{j-1}}\vert=\underline{\gD}$. 
We then  define the RG transformation $R$ that sends 
\begin{equation}
\left( \left( \tau_{j}-\tau_{j-1}, S_{\tau_{j}}-S_{\tau_{j-1}}\right)\right)_{j=1, \ldots, N}\,  =: \, 
\left( \left( \eta_j,  \gD_j \right)\right)_{j=1, \ldots, N}
\end{equation}
to a new sequence $\left( \left( \eta'_j,  \gD'_j \right)\right)_{j=1, \ldots, N-2}$
defined by 
\begin{enumerate}[leftmargin=0.7 cm]
\item $ \left( \eta'_j,  \gD'_j \right)=  \left( \eta_j,  \gD_j \right)$ for every $j\in \bbN$ such that $ j \le  \textsc{j}-2$ (such a $j$ exists unless $\textsc{j}=2$);
\item $ \left( \eta'_{j-2},  \gD'_{j-2} \right)=  \left( \eta_j,  \gD_j \right)$ for every $j\le N$ such that $ j \ge \textsc{j}+2$ (such a $j$ exists unless $\textsc{j}=N-1$);
\item $ \left( \eta'_{\textsc{j}-1}, \gD'_{\textsc{j}-1} \right)=  \left( \eta_{\textsc{j}-1}+ \eta_{\textsc{j}}+ \eta_{\textsc{j}+1},  \gD_{\textsc{j}-1}+ \gD_{\textsc{j}}+ \gD_{\textsc{j}+1}\right)$.
\end{enumerate}

This procedure can be iterated till the length of the sequence is $3$ or $4$, resulting on a final length respectively of length $1$ or $2$, see Fig.~\ref{fig:RGfig} . Note that under $R$ the value of $\underline{\gD}$
does not decrease and it increases a.s.. For every $\gG$ we keep applying $R$ 
 till $N\ge 3$ and $\underline{\gD}< \gG$: this includes the case in which $R$ cannot be applied even once. This way we define 
 $N_\gG$ and a finite sequence 
 \begin{equation}
\label{eq:alt-signs-Gamma}
\left( \left(\eta_{\gG,j}, \gD_{\gG,j}\right)\right)_{j=1, \ldots,N_\gG }\,.
\end{equation}
which has the property that $\left(\gD_{\gG,j}\right)_{j=1, \ldots,N_\gG }$ has alternating signs and the absolute value of each entry is $\gG$ or more.

We introduce also $j(N):= \sup\{j\in \bbN: \, \tu_j(\gG) \le N\}$, with $j(N)=0$ if the set is empty.

\medskip

\begin{proposition}
\label{th:RG}
Almost surely in the realization of $(h_k)_{k\in \bbN}$ we have that for every  
$N\ge 3$ we have that $j(N) \le N_\gG \le j(N)+3$ and
\begin{equation}
\label{eq:RGinth}
\left(\tu_j(\gG)\right)_{j=1, \ldots,j(N)} \subset \left(\sum_{k=1}^j \eta_{\gG,j} \right)_{j=1,\ldots,  N_\gG}\, .
\end{equation}
In particular, if $j(N)=N_\gG$ the inclusion in 
\eqref{eq:RGinth} is an equality. Otherwise 
the only elements in the right-hand side  in \eqref{eq:RGinth} that may not belong to the left-hand side 
are $ \eta_{\gG,1}$,  $\sum_{k=1}^{N_\gG-1}  \eta_{\gG,j}$ and $\sum_{k=1}^{N_\gG} \eta_{\gG,j}$.
\end{proposition}

\medskip

Note that  $N_\gG$ is nondecreasing $N$ increases and a.s. 
$N_\gG$ tends to $\infty$ when $N$ is sent to $\infty$. So the RG procedure does capture all the $\gG$-extrema, and captures possibly 
(at most three) spurious points near the boundaries.

\medskip

\begin{figure}[h]
\begin{center}
\includegraphics[width=15 cm]{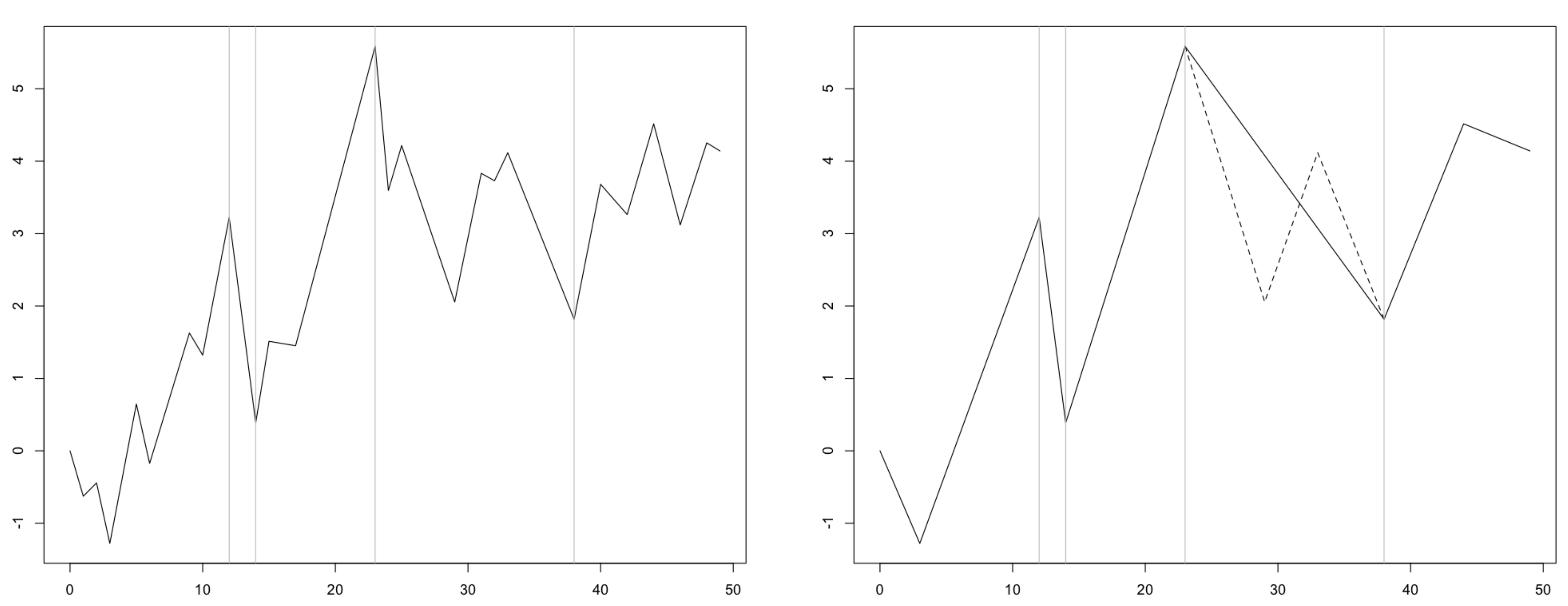}
\end{center}
\caption{\label{fig:RGfig} We apply the RG procedure $R$ with $\gG=2.5$ to the finite portion of random walk that corresponds to $N=50$ with the notations introduced for  Proposition~\ref{th:RG}. The image on the left is not the original random walk $S$ with space-time increments $((1,h_k))_{k=1,2, \ldots}$, but rather the coarse grained walk with alternating height signs 
with space-time increments given in \eqref{eq:alt-signs}. The $\gG$-extrema are at
12, 14, 23 and 38 and they are marked by grey vertical lines (the height differences between these 4 points are respectively $-2.835$,  $5.198$ and $-3.768$). In the image on the right there is final path:  the iteration has stopped because all the increments, except the first and the last one, are larger than $\gG$.  In the last image we show also the last RG step. Note that if we consider the bend points of the last trajectory, these bend points are the $\gG$-extrema plus 2 spurious boundary points. With the notations of the proof also the very last point is a bend point, so the spurious points are 3.
} 
\end{figure}

\begin{proof}
We start by remarking that $\left(\tu_j(\gG)\right)_{j=1, \ldots,j(N)} \subset (\tau_j))_{j=1, \ldots, N}$. 
By applying $R$ we just \emph{expel} two nearest neighbor points in $(\tau_j))_{j=1, \ldots, N}$
and, by definition of $\gG$-extrema and of $R$, this decimation procedure cannot involve
the location of a $\gG$-extremum. Next we remark that all the bonds   between two $\gG$-extrema $\tu_{j-1}(\gG)$ and $\tu_j(\gG)$, if more than one and they are always an odd number,  will eventually be decimated, because  otherwise we will have neighbor bonds both
with height steps of at least $\gG$and there would be a $\gG$-extremum in $(\tu_{j-1}(\gG),\tu_j(\gG))\cap \bbN$, which is impossible by construction. This in general cannot be said
for the intervals $\{1, \ldots,  \tu_1(\gG)\}$ and $\{\tu_{j(N)}, \ldots, \tau_N\}$. In the first case the problem is that 
the first increment cannot be affected by $R$ if the absolute value of the second increment is $\gG$ or larger, or if this happens at some point in the iteration. With $\{\tu_{j(N)}, \ldots, \tau_N\}$ the analogous problem is present with the last increment, but, added to that, there is also the problem that, unless the absolute value of the last height  increment is $\gG$ or more,  whether  
$\sum_{k=1}^{N_\gG-1} \eta_{\gG,j}$ is a $\gG$-extremum or not depends on $h_k$ for $k> \tau_N$ (that we haven't even sampled).
\end{proof}

\section*{Acknowledgments}
 Y.H. acknowledges the support of ANR, project LOCAL.  G.G.\ acknowledges the support of the Cariparo Foundation.

\end{document}